\documentclass[UTF-8,reqno]{amsart}
\usepackage{enumerate}
\usepackage{mhequ}
\usepackage[margin=1.2in]{geometry}
\usepackage{amssymb,url,color, booktabs,nccmath}
\usepackage{mathrsfs}
\usepackage{enumitem}
\usepackage{graphicx}
\usepackage{tikz}
\usepackage{amsthm}
\usepackage{color}
\usepackage[colorlinks=true]{hyperref}
\hypersetup{
    linkcolor=blue,          
    citecolor=red,        
    filecolor=blue,      
    urlcolor=cyan
}

\definecolor{darkergreen}{rgb}{0.0, 0.5, 0.0}

\numberwithin{equation}{section}
\def\theequation{\arabic{section}.\arabic{equation}}
\newcommand{\be}{\begin{eqnarray}}
\newcommand{\ee}{\end{eqnarray}}
\newcommand{\ce}{\begin{eqnarray*}}
\newcommand{\de}{\end{eqnarray*}}
\newtheorem{theorem}{Theorem}[section]
\newtheorem{lemma}[theorem]{Lemma}
\newtheorem{remark}[theorem]{Remark}
\newtheorem{definition}[theorem]{Definition}
\newtheorem{proposition}[theorem]{Proposition}
\newtheorem{Examples}[theorem]{Example}
\newtheorem{corollary}[theorem]{Corollary}
\newtheorem{assumption}{Assumption}[section]
\newtheorem*{theorem*}{Theorem}
\newtheorem*{remark*}{Remark}

\def\[{{\Big[}}
\def\]{{\Big]}}
\def\<{{\langle}}
\def\>{{\rangle}}
\def\({{\Big(}}
\def\){{\Big)}}

\def\bx{{\mathbf{x}}}

\def\={&\!\!=\!\!&}

\def\1{{\mathbf{1}}}

\def\geq{\geqslant}
\def\leq{\leqslant}

\def\le{\leqslant}

\def\[{{\Big[}}
\def\]{{\Big]}}
\def\<{{\langle}}
\def\>{{\rangle}}
\def\({{\Big(}}
\def\){{\Big)}}

\def\bx{{\mathbf{x}}}

\def\={&\!\!=\!\!&}
\def\bt{\begin{theorem}}
\def\et{\end{theorem}}
\def\bl{\begin{lemma}}
\def\el{\end{lemma}}
\def\br{\begin{remark}}
\def\er{\end{remark}}
\def\bx{\begin{Examples}}
\def\ex{\end{Examples}}
\def\bd{\begin{definition}}
\def\ed{\end{definition}}
\def\bp{\begin{proposition}}
\def\ep{\end{proposition}}
\def\bc{\begin{corollary}}
\def\ec{\end{corollary}}

\def\geq{\geqslant}
\def\leq{\leqslant}

\def\le{\leqslant}

\def\<{\langle} \def\>{\rangle}

\allowdisplaybreaks

\begin{document}

\title[Fluctuating hydrodynamics of The Ising-Kac-Kawasaki model]{\large Fluctuating Hydrodynamics of The Ising-Kac-Kawasaki Model and Nonlinear Fluctuations Near Criticality}

\author[Zhengyan Wu]{\large Zhengyan Wu}
\address[Z. Wu]{Department of Mathematics, Technische Universit\"at M\"unchen, Boltzmannstr. 3, 85748 Garching, Germany}
\email{wuzh@cit.tum.de}

\begin{abstract}
We study the scaling limit behavior of a family of conservative SPDEs as  the fluctuating Ising-Kac-Kawasaki dynamics. Precisely, we show that there exists a sequence of the one-dimensional rescaled fluctuating Ising-Kac-Kawasaki equation converges to the solution of the stochastic Cahn-Hilliard equation. This solves a simple version of the conjecture concerning the nonlinear fluctuation phenomenon, proposed by [Giacomin, Lebowitz, Presutti; Math. Surveys Monogr., 1999]. Furthermore, we prove a multi-scale dynamical large deviations in a small noise regime. Finally, we show the $\Gamma$-convergence of the rate function for the rescaled fluctuating Ising-Kac-Kawasaki equation to the rate function of the Cahn-Hilliard equation. 
\end{abstract}


\date{\today}

\maketitle

\setcounter{tocdepth}{1}

\tableofcontents

\section{Introduction}
The Ising-Kac model was introduced to recover the van der Waals theory of phase transitions, in which interactions are nonlocal; see \cite{HKU} for further details. The Kawasaki dynamics provides a natural evolution for the Ising-Kac model, whereby neighboring particles exchange spins at a prescribed rate. In \cite{GJE99}, formal singular stochastic PDEs were proposed as approximations of the Ising-Kac-Kawasaki dynamics by comparing their generators, allowing for heuristic computations that illustrate the nonlinear fluctuation phenomena. These stochastic PDEs can be interpreted as the fluctuating hydrodynamic equations corresponding to the Ising-Kac-Kawasaki model.  

We first summarize the content of nonlinear fluctuations conjecture proposed by \cite{GJE99} for the fluctuating Ising-Kac-Kawasaki equation. Let $\beta>0$, the following one-dimensional informal singular stochastic PDE has been considered in \cite[(7.34)]{GJE99}: 
\begin{align}\label{intro-spde-1}
\partial_tm=\partial_{xx}^2m-\beta\partial_x\Big[(1-m^2)\partial_xJ\ast m\Big]+\gamma^{1/2}\sqrt{2}\partial_x\Big(\sqrt{1-m^2}\xi\Big), 	
\end{align}
where $J$ denotes a smooth interaction kernel and $\xi$ denotes a space-time white noise. By the scaling properties of the white noise, in order to describe the evolution of the rescaled fluctuation field 
\begin{equation}\label{fluctuation-field-intro}
	u_{\gamma}(x,t):=\gamma^{-1/3}m(\gamma^{-1/3}x,\gamma^{-4/3}t),\ \ \beta=1+a\gamma^{2/3}, 
\end{equation}
for some $a\in\mathbb{R}$, \cite{GJE99} derive the following equation (see \cite[(7.36)]{GJE99}) 
\begin{align}\label{intro-spde-2}
	\partial_tu_{\gamma}=\gamma^{-2/3}\Big[\partial_{xx}^2 u_{\gamma}-(1+a\gamma^{2/3})\partial_x\Big((1-\gamma^{2/3}u_{\gamma}^2)J_{\gamma^{1/3}}\ast\partial_xu_{\gamma}\Big)\Big]+\sqrt{2}\partial_x\Big(\sqrt{1-\gamma^{2/3}u_{\gamma}^2}\xi\Big),
\end{align}
where $J_{\gamma^{1/3}}(\cdot)=\gamma^{-1/3}J(\gamma^{-1/3}\cdot)$. We defer the precise derivation of the above rescaled equation to Subsection \ref{subsec-equation}. Based on an informal computation of the Taylor expansion of the nonlocal term $J_{\gamma^{1/3}}\ast\partial_xu_{\gamma}$, Giacomin, Lebowitz and Presutti \cite{GJE99} conjectured that the solution $u_{\gamma}$ of \eqref{intro-spde-2} converges to the solution of the stochastic Cahn-Hilliard equation 
\begin{equation}\label{intro-CahnHilliard}
\partial_tu=\partial_{xx}^2\Big[V'(u)-\frac{D}{2}\partial_{xx}^2 u\Big]-\sqrt{2}\partial_x\xi
\end{equation}
as $\gamma\rightarrow0$, where $V(u)=\frac{u^4}{4}-\frac{au^2}{2}$, $D=\int J(x)|x|^2dx$. This is not a Gaussian limit; therefore, it is known as a nonlinear fluctuation phenomenon.

The Giacomin-Lebowitz-Presutti conjecture provides a framework for analyzing the fluctuation behavior of long-range interacting spin systems in the continuum limit. Under the Kac scaling, which simulates interactions long-range but weak, the microscopic Ising-Kac-Kawasaki model is expected to be governed by a nonlocal, nonlinear parabolic equation describing the fluctuations. A key insight of the conjecture is that, at the critical temperature, the standard linearization around the mean fails, and the leading-order fluctuations are intrinsically nonlinear. These nonlinear fluctuations are predicted to capture fundamental aspects of the critical dynamics, including non-Gaussian statistics, thereby offering a characterization of phase transitions in systems with weak long-range interactions. 

The above description of the Giacomin-Lebowitz-Presutti conjecture remains at an informal stage due to the lack of a rigorous, general mathematical theory of supercritical singular stochastic PDEs. This results in the ill-definedness of \eqref{intro-spde-2}. Consequently, a rigorous study of the Giacomin-Lebowitz-Presutti conjecture requires a modification of its statement at present.

Referring to the theory of fluctuating hydrodynamics, conservative stochastic PDE models of type \eqref{intro-spde-2} are used to describe the mesoscopic picture. Consequently, the driving noise has a small spatial correlation length. For generality, we replace the noise~$\xi$ with a spatially correlated noise of correlation length~$\delta>0$.

Furthermore, both physical and analytical considerations suggest modifying the type of noise. In thermodynamics \cite{Ottinger}, \"Ottinger considers the stochastic GENERIC equation driven by Klimontovich noise in order to preserve the dynamics invariant under the Gibbs measure, albeit informally. In \cite{FG24}, an analytic approach to the Dean-Kawasaki equation indicated that Stratonovich noise provided stochastic coercivity, and this is essential for obtaining well-posedness.    

In comparison to Klimontovich noise, the Stratonovich-It\^o correction is half of the Klimontovich-It\^o correction, and this factor of $1/2$ makes no difference in the Giacomin-Lebowitz-Presutti scaling limit. Therefore, we ignore this difference and modify \eqref{intro-spde-2} into a stochastic PDE driven by Stratonovich-correlated noise: 
\begin{align}\label{intro-spde-3}
	\partial_tu_{\gamma}=\gamma^{-2/3}\Big[\partial_{xx}^2 u_{\gamma}-(1+a\gamma^{2/3})\partial_x\Big((1-\gamma^{2/3}u_{\gamma}^2)J_{\gamma^{1/3}}\ast\partial_xu_{\gamma}\Big)\Big]+\sqrt{2}\partial_x\Big(\sqrt{1-\gamma^{2/3}u_{\gamma}^2}\circ \xi_{\delta}\Big),
\end{align}
where $\xi_{\delta}=\xi\ast\eta_{\delta}$, with $\eta_{\delta}$ being a standard convolution kernel. The analytic approach developed by \cite{FG24,FG23} has been extended by Wang, the author, and Zhang \cite{WWZ22} to more general cases that encompass nonlocal interactions, and the well-posedness of equation \eqref{intro-spde-3} has also been established by \cite{WWZ22} for every fixed $\gamma,\delta>0$. 

In \cite{FG24,FG23,DFG,GWZ24}, several techniques have been employed to address the irregularity of the square-root coefficient, including the renormalized kinetic solution theory by \cite{LPT94}, the doubling variables approach for uniqueness, entropy dissipation estimates, a diagonal argument for $L^1$-tightness, and so on. However, in the study of the Giacomin-Lebowitz-Presutti conjecture, there are still limitations in applying these techniques, and the irregularity of the square-root remains a challenge. Further discussion of this point will be provided in Section \ref{sec-5}. To simplify the problem, we replace the square-root coefficient with a smooth approximation.

To conclude the above analysis, the first goal of this paper is to investigate the fluctuation field of the fluctuating Ising-Kac-Kawasaki equation 
\begin{align}\label{equgamma-0}
\partial_tu_{\gamma}=&\gamma^{-2/3}\partial_{xx}^2 u_{\gamma}-\gamma^{-2/3}(1+a\gamma^{2/3})\partial_{x}\Big[(1-\gamma^{2/3}u_{\gamma}^2)J_{\gamma^{1/3}}\ast\partial_{x} u_{\gamma}\Big]-\sqrt{2}\partial_{x}\Big(\sigma(1-\gamma^{2/3}u_{\gamma}^2)\circ \xi_{\delta}\Big),   
\end{align}
with initial data $u_{\gamma,0}:\mathbb{T}^1\rightarrow[-\gamma^{-1/3},\gamma^{-1/3}]$, where $\mathbb{T}^1$ denotes the one-dimensional torus with volume 1, with convention $\mathbb{T}^1=[-1/2,1/2]$. Furthermore, $\ast$ denotes the spatial convolution, $\circ$ denotes the Stratonovich integration, $J\in C^{\infty}_c((-1/2,1/2))$ is a smooth kernel. We still denote the periodized $J$ by $J$, and let $J_{\gamma^{1/3}}=\gamma^{-1/3}J(\gamma^{-1/3}\cdot)$. Moreover, $\sigma(\cdot)$ is a smooth approximation of the square-root function. We devote to showing that $u_{\gamma}\rightarrow u$ along a subsequence of $(\gamma,\delta)\rightarrow(0,0)$, where $u$ solves the Cahn-Hilliard equation:  
\begin{equation}\label{CahnHilliard}
\partial_tu=\partial_{xx}^2\Big[V'(u)-\frac{D}{2}\partial_{xx}^2 u\Big]-\sqrt{2}\partial_x\xi,\ \ u(0)=u_0, 
\end{equation}
with $V(u)=\frac{u^4}{4}-\frac{au^2}{2}$, $D=\int_{\mathbb{T}^1}J(x)|x|^2dx$, the space-time white noise $\xi$ being the limit of $\xi_{\delta}=\xi\ast\eta_{\delta}$ in \eqref{equgamma-0} and $u_0$ being the limit point of $u_{\gamma,0}$.  

This simplifies the Giacomin-Lebowitz-Presutti conjecture in two aspects: replacing the singular noise with a regular Stratonovich noise and regularizing the square-root coefficient. The former simplification introduces the limitation that we can only obtain convergence along a subsequence of $(\gamma,\delta)\rightarrow(0,0)$, since for fixed $\gamma\in(0,1]$, letting $\delta\rightarrow0$ does not lead to a nontrivial limit at present. Furthermore, the conjecture stated in \cite{GJE99} is considered on $\mathbb{R}^1$, whereas we consider all equations on a torus $\mathbb{T}^1$.

In the second part of this paper, we study a multi-scale large deviation principle for the fluctuation field in a small noise regime. Precisely, we denote the noise intensity of \eqref{equgamma-0} by an additional parameter $\varepsilon^{1/2}>0$:   
\begin{align}\label{smallnoiseeq}
\partial_tu_{\varepsilon,\gamma,\delta,n}=\gamma^{-2/3}\partial_{xx}^2 u_{\varepsilon,\gamma,\delta,n}-&\gamma^{-2/3}(1+a\gamma^{2/3})\partial_{x}\Big[(1-\gamma^{2/3}u_{\varepsilon,\gamma,\delta,n}^2)J_{\gamma^{1/3}}\ast\partial_{x} u_{\varepsilon,\gamma,\delta,n}\Big]\notag\\
&-\sqrt{2}\varepsilon^{1/2}\partial_{x}\Big(\sigma_n(1-\gamma^{2/3}u_{\varepsilon,\gamma,\delta,n}^2)\circ \xi_{\delta}\Big),\  u_{\varepsilon,\gamma,\delta,n}(0)=u_{\gamma,0}, 	
\end{align}
where $(\sigma_n(\cdot))_{n\geq0}$ is a sequence of smooth approximation of the square-root coefficient. Throughout this work, we fix a time horizon $T > 0$ and aim to establish a large deviation principle for the trajectories of the family of processes $(u_{\varepsilon} = u_{\varepsilon, \gamma(\varepsilon), \delta(\varepsilon), n(\varepsilon)})_{\varepsilon > 0}$ on the interval $[0, T]$. We consider these trajectories under the scaling regime
$$
(\varepsilon, \gamma(\varepsilon), \delta(\varepsilon), n(\varepsilon)) \to (0, 0, 0, \infty),
$$
as $\varepsilon \to 0$. Our goal is to characterize the asymptotic behavior of the probability of deviations of $u_{\varepsilon}$ from its typical behavior in this limit, within an appropriate topology for trajectories on $[0, T]$.

\begin{align*}
\mathbb{P}(u_{\varepsilon}\in C)\sim\exp\Big(-\varepsilon^{-1}\inf_{u\in C}\mathcal{I}_{CH}(u)\Big),\ \ \text{as }\varepsilon\rightarrow0, \ \text{for some rare events }C,  	
\end{align*}
where the rate function $\mathcal{I}_{CH}$ is defined by 
\begin{align}\label{Cahn-Hilliar-rate-intro}
\mathcal{I}_{CH}(u)=\inf\Big\{\frac{1}{2}\|g\|_{L^2([0,T];L^2(\mathbb{T}^1))}^2:\partial_tu=\partial_{xx}^2\Big[V'(u)-\frac{D}{2}\partial_{xx}^2 u\Big]-\sqrt{2}\partial_xg,\ u(0)=u_0\Big\}, 	
\end{align}
where $u_0$ is the limit point of $u_{\gamma(\varepsilon),0}$. 

For every fixed $\gamma\in(0,1]$, the small-noise large deviations under the scaling regimes $(\varepsilon, \delta(\varepsilon), n(\varepsilon))\rightarrow(0,0,\infty)$ for \eqref{smallnoiseeq} have been established by the author and Zhang \cite{WZ24}\footnote{\cite{WZ24} proved large deviations for \eqref{smallnoiseeq} without regularizing the square-root coefficient, which is even more challenging.}. Specifically, they show that the large deviations hold with respect to the rate function  
\begin{align}\label{rate-gamma-intro}  
\mathcal{I}_{IKK,\gamma}(u) = \frac{1}{2} \inf \Big\{\|g\|_{L^2([0,T];L^2(\mathbb{T}^1))}^2 : \partial_t u = & \gamma^{-2/3} \partial_{xx}^2 u - \gamma^{-2/3} (1 + a\gamma^{2/3}) \partial_x \Big[(1 - \gamma^{2/3} u^2) J_{\gamma^{1/3}} \ast \partial_x u\Big] \notag \\  
& - \sqrt{2} \partial_x \Big(\sqrt{1 - \gamma^{2/3} u^2} g \Big), \quad u(0) = u_{\gamma,0} \Big\}.  
\end{align}  

This result demonstrates the consistency between fluctuating hydrodynamics and interacting particle systems with long-range interactions \cite{ME07} in terms of large deviation rates. In this paper, we further let the parameter $\gamma \to 0$ to investigate multi-scale large deviations, aiming to understand how the Cahn-Hilliard rate function $\mathcal{I}_{CH}$ predicts the exponential decay of rare events. Furthermore, we also study the $\Gamma$-convergence of $\mathcal{I}_{IKK,\gamma}$ to the rate function $\mathcal{I}_{CH}$ as $\gamma \to 0$.

\subsection{Main results}
Before we state the main results, we firstly introduce an assumption for the initial data $u_{\gamma,0}$, $\gamma\in(0,1]$. We define a rescaled entropy function: 
\begin{align*}
\Psi_{\gamma}(\zeta) = &\frac{\gamma^{-1/3}}{2} \Big[(1+\gamma^{1/3}\zeta)\log(\gamma^{1/3}\zeta+1) + (1-\gamma^{1/3}\zeta)\log(1-\gamma^{1/3}\zeta) -2\Big], \quad \zeta \in [-\gamma^{-1/3},\gamma^{-1/3}].  
\end{align*}  
\begin{description}
\item[Assumption (A1)] The initial data $u_{\gamma,0}$ is a sequence of $\mathcal{F}$-measurable random variables with  

(i) $-\gamma^{-1/3}\leq u_{\gamma,0}\leq \gamma^{-1/3}$, almost surely. 

(ii) Almost surely, for every $\gamma\in(0,1]$, $\Psi_{\gamma}(u_{\gamma,0})\in L^1(\mathbb{T}^1)$. 

(iii) Almost surely, $\sup_{\gamma\in(0,1]}\|u_{\gamma,0}\|_{L^1(\mathbb{T}^1)}<\infty$. 

(iv) $\mathbb{E}\int_{\mathbb{T}^1}\Psi_{\gamma}(u_{\gamma,0})+\gamma^{-1/3}dx\lesssim \gamma^{1/3}$ holds for every $\gamma\in(0,1]$. 

(v) Almost surely, $\|u_{\gamma,0}-u_0\|_{H^{-1}(\mathbb{T}^1)}\rightarrow0$, as $\gamma\rightarrow0$, for some $u_0\in H^{-1}(\mathbb{T}^1)$.   
\end{description}

We noting that $u_{\gamma,0}=0$ is an example of Assumption (A1). To construct a nontrivial example, for every $\gamma\in(0,1]$, let $\rho_{0,\gamma}$ be a random variable such that, almost surely, $\rho_{0,\gamma}\in C^{\infty}_c(-1/2,1/2)$ and 
\begin{equation}\label{entropy-bound}
\mathbb{E}(\Psi_1(\rho_{0,\gamma})+1)\leq\gamma^{1/3}\quad\text{and }\mathbb{E}\|\rho_{0,\gamma}\|_{L^1(\mathbb{T}^1)}\leq C.
\end{equation}
We periodize $\rho_{0,\gamma}$ as a function on $\mathbb{T}^1$, and denote it by $\bar{\rho}_{0,\gamma}$. Define $u_{\gamma}(0)=\gamma^{-1/3}\bar{\rho}_{0,\gamma}(\gamma^{-1/3}\cdot)$. Using the change variables formula, for every $\gamma\in(0,1]$
\begin{align*}
\int_{\mathbb{T}^1}|u_{\gamma,0}|dx=&\int_{\mathbb{T}^1}|\gamma^{-1/3}\bar{\rho}_{0,\gamma}(\gamma^{-1/3}x)|dx=\int_{\mathbb{R}^1}|\gamma^{-1/3}\rho_{0,\gamma}(\gamma^{-1/3}x)|dx=\int_{\mathbb{R}^1}|\rho_{0,\gamma}(y)|dy\leq C.
\end{align*}
Similarly, combining \eqref{entropy-bound} and the change variables formula, we can verify (iv). Since $\rho_0$ is smooth, $u_{\gamma,0}$ is a mollifier, and thus (v) holds as well.  

The first main results show that along a subsequence, the fluctuation field \eqref{equgamma-0} converges to the solution of the stochastic Cahn-Hilliard equation \eqref{CahnHilliard}, which solves a simplified version of the Giacomin-Lebowitz-Presutti conjecture. 
\begin{theorem}\label{thm-1-intro}
Assume that the Assumption (A1) holds. Suppose that $J\in C^{\infty}(\mathbb{T}^1)$. For every $\gamma\in(0,1]$ and $\delta>0$, let $u_{\gamma,\delta}$ be the weak solution of \eqref{equgamma-0} with initial data $u_{\gamma,0}$. Suppose that $a<0$, then there exists a subsequence of $(u_{\gamma,\delta})_{\gamma,\delta>0}$, denoted by $(u_n:=u_{\gamma(n),\delta(n)})_{n\geq1}$, such that 
\begin{align*}
\|u_{n}-u\|_{L^2([0,T];L^2(\mathbb{T}^1))}\rightarrow0,	
\end{align*}
as $n\rightarrow\infty$ in probability, where $u$ is the weak solution of the Cahn-Hilliard equation \eqref{CahnHilliard} with initial data $u_0$.  
\end{theorem}

In the following, we consider a family of deterministic initial data $u_{\gamma,0}$, $\gamma \in (0,1)$, rather than stochastic initial data. We assume that this family satisfies conditions similar to those stated in Assumption~(A1). Precisely, 
\begin{description}
\item[Assumption (A2)] The initial data $u_{\gamma,0}$ satisfies that 

(i) $-\gamma^{-1/3}\leq u_{\gamma,0}\leq \gamma^{-1/3}$. 

(ii) For every $\gamma\in(0,1]$, $\Psi_{\gamma}(u_{\gamma,0})\in L^1(\mathbb{T}^1)$. 

(iii) $\sup_{\gamma\in(0,1]}\|u_{\gamma,0}\|_{L^1(\mathbb{T}^1)}<\infty$. 

(iv) $\int_{\mathbb{T}^1}\Psi_{\gamma}(u_{\gamma,0})+\gamma^{-1/3}dx\lesssim \gamma^{1/3}$ holds for every $\gamma\in(0,1]$. 

(v) $\|u_{\gamma,0}-u_0\|_{H^{-1}(\mathbb{T}^1)}\rightarrow0$, as $\gamma\rightarrow0$, for some $u_0\in H^{-1}(\mathbb{T}^1)$.   
\end{description}

The second result stated that under a certain scaling regime, the solution of \eqref{smallnoiseeq} satisfies a dynamical large deviation principle, with respect to the Cahn-Hilliard rate function \eqref{Cahn-Hilliar-rate-intro}. 
\begin{theorem}\label{thm-2-intro}
	Assume that the Assumption (A2) holds. For every $\varepsilon,\gamma,n,\delta>0$, let $u_{\varepsilon}$ be the weak solution of \eqref{smallnoiseeq} with initial data $0$. Suppose that $a<0$. Let $u_{\gamma,0}$ be a sequence of functions satisfy Assumption (A2), and let $u_0$ be its limit. Under the scaling regime $(\varepsilon,\gamma(\varepsilon),\delta(\varepsilon),n(\varepsilon))\rightarrow(0,0,0,+\infty)$ such that 
\begin{align*}
	\varepsilon\Big(\delta(\varepsilon)^{-2/3}+\gamma(\varepsilon)^{4/3}\delta(\varepsilon)^{-1/3}\|\sigma_{n(\varepsilon)}'(\cdot)\|_{L^{\infty}(\mathbb{R})}^2\Big)\rightarrow0, 
\end{align*}
then the laws of $(u_{\varepsilon})_{\varepsilon>0}$ satisfy a large deviation principle with a good rate function $\mathcal{I}_{CH}$ defined by \eqref{Cahn-Hilliar-rate-intro}. 
\end{theorem}

Finally, we show a result of $\Gamma$-convergence of the rate function \eqref{rate-gamma-intro} to the rate function \eqref{Cahn-Hilliar-rate-intro}. 
\begin{theorem}\label{thm-3-intro}
Let $a<0$. For every $\gamma\in(0,1]$, let $\mathcal{I}_{IKK,\gamma}$ be defined by \eqref{rate-gamma-intro} with a sequence of initial data $u_{\gamma,0}$ satisfying Assumption (A2), and let $\mathcal{I}_{CH}$ be defined by \eqref{Cahn-Hilliar-rate-intro}. Then 
\begin{align}
\mathcal{I}_{IKK,\gamma}\rightarrow\mathcal{I}_{CH}, 
\end{align}
as $\gamma\rightarrow0$, in the sense of $\Gamma$-convergence. 
\end{theorem}

\subsection{Key ideas and technical comments}
We first summarize the key ideas in the proof of Theorem \ref{thm-1-intro}. The convergence result stated in Theorem \ref{thm-1-intro} is established through a two-step analysis. To emphasize the dependence on the correlation length parameter $\delta > 0$, we denote the solution of \eqref{equgamma-0} by $u_{\gamma,\delta}$.  

First, we show that $u_{\gamma,\delta} \to u_{\delta}$ in probability, where $u_{\delta}$ is the weak solution of the stochastic Cahn-Hilliard equation driven by a correlated noise:  

\begin{equation}\label{CahnHilliard-delta-intro}  
\partial_t u_{\delta} = \partial_{xx}^2 \Big[V'(u_{\delta}) - \frac{D}{2} \partial_{xx}^2 u_{\delta} \Big] - \sqrt{2} \partial_x \xi_{\delta}, \quad u_{\delta}(0) = u_0.  
\end{equation}  

Next, we establish the convergence $u_{\delta} \to u$, where $u$ is the solution of \eqref{CahnHilliard}. This latter convergence is proved using the Da Prato-Debussche trick, with a detailed proof provided in Appendix \ref{sec-app-B}. The main focus of our analysis is the first convergence, $u_{\gamma,\delta} \to u_{\delta}$. For each fixed $\delta > 0$, we prove this convergence by applying the Aubin-Lions compactness method. To verify the compactness criterion, we require two uniform estimates.  

\textbf{Uniform rescaled entropy dissipation estimates.}  
As noted in \cite{FG24,DFG,WWZ22}, entropy dissipation estimates play a crucial role in the study of the existence of the Dean-Kawasaki equation and related structured equations. We emphasize that in \cite[(7.6)]{WWZ22}, the author and Zhang introduced the mathematical entropy of \eqref{intro-spde-1}:  
\begin{align}\label{entropy-intro}
\Psi(\zeta)=&\frac{1}{2}\Big[(1+\zeta)\log(\zeta+1)+(1-\zeta)\log(1-\zeta)-2\Big],\ \ \zeta\in[-1,1]. 
\end{align}
In order to find a suitable uniform estimate for \eqref{equgamma-0}, we define a rescaled entropy function according to the scaling considered in \eqref{fluctuation-field-intro}:  

\begin{align}\label{rescaled-entropy-intro}
\Psi_{\gamma}(\zeta) = &\frac{\gamma^{-1/3}}{2} \Big[(1+\gamma^{1/3}\zeta)\log(\gamma^{1/3}\zeta+1) + (1-\gamma^{1/3}\zeta)\log(1-\gamma^{1/3}\zeta) -2\Big], \quad \zeta \in [-\gamma^{-1/3},\gamma^{-1/3}].  
\end{align}  

By applying a small trick based on basic mathematical analysis, we are able to bound the rescaled entropy function $\Psi_{\gamma}$ from below by a quadratic function: $\frac{1}{2} \gamma^{1/3} \zeta^2 - \gamma^{-1/3}$,  
see Lemma \ref{lem-tech-1}. The scaling factor $\gamma^{1/3}$ enables us to cancel out other scaling parameters, allowing us to obtain the final uniform estimate; see Proposition \ref{prp-entropydissipation} for more details. We also emphasize that the estimates of the Kac-interaction term strongly rely on the choice $a<0$, which excludes the double-well potential case of $V$. Eventually, we obtain an $L^2(\Omega\times[0,T];H^1(\mathbb{T}^1))$-uniform estimate for $(u_{\gamma,\delta})_{\gamma\in(0,1]}$.  

\textbf{Uniform time-regularity estimates.}  
We observe that the interaction term can be divided into two parts:  

\begin{align*}
	-\gamma^{-2/3}(1+a\gamma^{2/3})\partial_{x} \Big[(1-\gamma^{2/3}u_{\gamma,\delta}^2)J_{\gamma^{1/3}}\ast\partial_{x}u_{\gamma,\delta}\Big]  
	=& -\gamma^{-2/3}(1+a\gamma^{2/3})\partial_{x} \Big[J_{\gamma^{1/3}}\ast\partial_{x}u_{\gamma,\delta}\Big] \\
	&+ (1+a\gamma^{2/3})\partial_x \Big[u_{\gamma,\delta}^2 J_{\gamma^{1/3}}\ast\partial_x u_{\gamma,\delta}\Big].  
\end{align*}  

These two terms have different orders of $\gamma$ and therefore require different orders of the Taylor expansion for $u_{\gamma,\delta}$. Based on the above analysis, we rewrite \eqref{equgamma-0} in the following form:   
\begin{align}\label{rewriteform-intro}
\partial_tu_{\gamma,\delta}=&\gamma^{-2/3}\partial_{xx}^2 u_{\gamma,\delta}-\sqrt{2}\partial_{x}(\sigma(1-\gamma^{2/3}u_{\gamma,\delta}^2)\xi_{\delta})+4\partial_{x}\Big(F_{1,\delta}\sigma'(1-\gamma^{2/3}u_{\gamma,\delta}^2)^2\gamma^{4/3}u_{\gamma,\delta}^2\partial_{x}u_{\gamma,\delta}\Big)\notag\\
&-\gamma^{-2/3}(1+a\gamma^{2/3})\partial_{x}\Big[J_{\gamma^{1/3}}\ast\partial_{x}u_{\gamma,\delta}\Big]+(1+a\gamma^{2/3})\partial_x\Big[u_{\gamma,\delta}^2J_{\gamma^{1/3}}\ast\partial_xu_{\gamma,\delta}\Big]\notag\\
=&\gamma^{-2/3}\partial_{xx}^2 u_{\gamma,\delta}-\sqrt{2}\partial_{x}(\sigma(1-\gamma^{2/3}u_{\gamma,\delta}^2)\xi_{\delta})+4\partial_{x}\Big(F_{1,\delta}\sigma'(1-\gamma^{2/3}u_{\gamma,\delta}^2)^2\gamma^{4/3}u_{\gamma,\delta}^2\partial_{x}u_{\gamma,\delta}\Big)\notag\\
&-\gamma^{-2/3}(1+a\gamma^{2/3})\partial_{x}\Big[J_{\gamma^{1/3}}\ast\partial_{x}u_{\gamma,\delta}-\partial_xu_{\gamma,\delta}-\gamma^{2/3}\frac{D}{2}\partial_{xxx}^3u_{\gamma,\delta}\Big]\notag\\
&+(1+a\gamma^{2/3})\partial_x\Big[u_{\gamma,\delta}^2\Big(J_{\gamma^{1/3}}\ast\partial_xu_{\gamma,\delta}-\partial_xu_{\gamma,\delta}\Big)\Big]\notag\\
&-\gamma^{-2/3}(1+a\gamma^{2/3})\Big[\partial_{xx}^2u_{\gamma,\delta}+\gamma^{2/3}\frac{D}{2}\partial_{xxxx}^4u_{\gamma,\delta}\Big]+(1+a\gamma^{2/3})\partial_x(u_{\gamma,\delta}^2\partial_xu_{\gamma,\delta})\notag\\
=&-\frac{D}{2}\partial_{xxxx}^4u_{\gamma,\delta}+\partial_{xx}^2\Big[\frac{1}{3}u_{\gamma,\delta}^3-au_{\gamma,\delta}\Big]+\sqrt{2}\partial_{x}(\sigma(1-\gamma^{2/3}u_{\gamma,\delta}^2)\xi_{\delta})\notag\\
&+R_{1,\gamma,\delta}+R_{2,\gamma,\delta}+R_{3,\gamma,\delta}+R_{4,\gamma,\delta},
\end{align}
where $D=\int_{\mathbb{T}^1}J(x)|x|^2dx$, and the remainder terms defined by 
\begin{align*}
&R_{1,\gamma,\delta}=-\frac{D}{2}a\gamma^{2/3}\partial_{xxxx}^4u_{\gamma,\delta}+\frac{1}{3}a\gamma^{2/3}\partial_{xx}^2u_{\gamma,\delta}^3,\\
&R_{2,\gamma,\delta}=(1+a\gamma^{2/3})\partial_x\Big[u_{\gamma,\delta}^2(J_{\gamma^{1/3}}\ast\partial_xu_{\gamma,\delta}-\partial_xu_{\gamma,\delta})\Big],\\
&R_{3,\gamma,\delta}=-\gamma^{-2/3}(1+a\gamma^{2/3})\partial_{x}\Big[J_{\gamma^{1/3}}\ast\partial_{x} u_{\gamma,\delta}-\partial_{x} u_{\gamma,\delta}-\gamma^{2/3}\frac{D}{2}\partial_{xxx}^3u_{\gamma,\delta}\Big],\\
&R_{4,\gamma,\delta}=4\partial_{x}\Big(F_{1,\delta}\sigma'(1-\gamma^{2/3}u_{\gamma,\delta}^2)^2\gamma^{4/3}u_{\gamma,\delta}^2\partial_{x}u_{\gamma,\delta}\Big).
\end{align*}
Here $R_{4,\gamma,\delta}$ is the Stratonovich-It\^o correction term. We establish a uniform estimate for $u_{\gamma,\delta}$ in the space 
$W^{\alpha,2}([0,T];H^{-\beta}(\mathbb{T}^1))$, valid for all $\alpha \in (0,1/2)$ and $\beta > 13/2$. 
The primary difficulties in obtaining this estimate arise from the terms $R_{2,\gamma,\delta}$ and $R_{3,\gamma,\delta}$. 
To control $R_{2,\gamma,\delta}$, we employ Poincar\'e's inequality in conjunction with uniform entropy dissipation estimates. 
For $R_{3,\gamma,\delta}$, its linear structure allows us to transfer the fourth-order derivative terms into the Sobolev norm 
$\|\cdot\|_{H^{-\beta}(\mathbb{T}^1)}$. 
By Sobolev embedding, we obtain the bound 
$\|\cdot\|_{H^{-\beta+4}(\mathbb{T}^1)} \lesssim \|\cdot\|_{L^1(\mathbb{T}^1)}$. 
Furthermore, by applying a third-order Taylor expansion to
\begin{align*}
U_{\gamma} := -(-\partial_{xx}^2)^{-\frac{3}{2}} u_{\gamma}, 
\end{align*}
and utilizing the uniform $L^2([0,T];H^1(\mathbb{T}^1))$-estimate, we are able to obtain sufficient control over $R_{3,\gamma,\delta}$. Further details can be found in the proof of Proposition \ref{second-uniform-es}. 

Combining these two uniform estimates, we are able to prove the convergence $u_{\gamma,\delta} \rightarrow u_{\delta}$ via a compactness argument. Together with the convergence $u_{\delta} \rightarrow u$, this implies that, along a subsequence of $(\gamma,\delta)$, we have $u_{\gamma,\delta} \rightarrow u$ in probability. 

\textbf{Multi-scale large deviations.}  
To establish the multi-scale large deviations result stated in Theorem \ref{thm-2-intro}, we employ the weak convergence approach developed by Dupuis and Ellis \cite{DE}. Specifically, we adopt the sufficient condition for large deviation principles introduced by Budhiraja, Dupuis and Maroulas \cite{BDM11}, which applies to functionals of Brownian motions and consists of two main components. The first requires establishing the weak-strong continuity of the Cahn-Hilliard skeleton equation, which we verify using a compactness argument. The second involves proving the weak convergence of perturbations of \eqref{smallnoiseeq} along directions in the Cameron-Martin space associated with the driving Brownian motions. This leads to the analysis of a so-called stochastic controlled equation; see \eqref{eqgamma-2} below. The proof of this convergence relies fundamentally on the techniques developed in Theorem \ref{thm-1-intro}. Furthermore, utilizing the compactness argument developed in Theorem \ref{thm-1-intro}, we establish the convergence of the skeleton equation associated with \eqref{smallnoiseeq} to the corresponding Cahn-Hilliard skeleton equation; see Proposition \ref{convergence-skeleton} for details. Moreover, by verifying the sufficient conditions provided in \cite{M12}, we demonstrate the $\Gamma$-convergence of the associated rate functions, as stated in Theorem \ref{thm-3-intro}.

\subsection{Nonlinear fluctuations of the dynamical Ising-Kac model}

\ 

In this part, we briefly summarize the development of the study of nonlinear fluctuations for dynamical Ising-Kac models. Let $N\in\mathbb{N}_+$, and let $\Lambda_N = \{-N+1, \dots, N\}$ denote the discrete lattice domain with periodic boundary conditions. Let $\Sigma_N$ be the corresponding configuration space. For each configuration $\sigma \in \Sigma_N$ and any function $g:\Sigma_N \rightarrow \mathbb{R}$, consider the generators
\begin{align*}
\mathcal{L}_{Kaw} g(\sigma) &= \sum_{x \sim y \in \Lambda_N} c_{Kaw}(x, y, \gamma, \beta, \sigma) \bigl(g(\sigma^{x,y}) - g(\sigma)\bigr), 
\end{align*}
and
\begin{align*}
\mathcal{L}_{Glau} g(\sigma) &= \sum_{x \in \Lambda_N} c_{Glau}(x, \gamma, \beta, \sigma) \bigl(g(\sigma^{x}) - g(\sigma)\bigr),
\end{align*}
where $\sigma^{x,y}$ denotes the configuration obtained from $\sigma$ by swapping the spins at sites $x$ and $y$, and $\sigma^x$ denotes the configuration obtained by flipping the spin at site $x$. Here, $c_{Kaw}(x, y, \gamma, \beta, \sigma)$ and $c_{Glau}(x, \gamma, \beta, \sigma)$ denote the rates of the Kawasaki and Glauber dynamics, respectively. These rates depend on the positions $x$ and $y$, the inverse temperature $\beta$, the configuration $\sigma$, and the Hamiltonian. The Hamiltonian interaction is no longer restricted to nearest neighbors but involves the Kac potential, with interaction range $\gamma^{-1} > 0$. Precise definitions can be found in \cite{BPRS94, FR95, MW17, GMW23, IM18}.

{\bf Glauber dynamics.}
Regarding the Glauber dynamics, nonlinear fluctuation phenomena have been investigated in different spatial dimensions: see \cite{BPRS94, FR95} for the one-dimensional case, \cite{MW17} for two dimensions, and \cite{GMW23} for three dimensions. In particular, Fritz and R\"udiger \cite{FR95} provided a rigorous analysis of the Glauber dynamics in one dimension and proved that, under an appropriate scaling, the temporal evolution of spin fluctuations converges to the stochastic Allen--Cahn equation. Their approach relies on a delicate coupling between the Glauber dynamics and a linearized voter-model approximation, which makes it possible to establish energy inequalities that control the critical fluctuations. This methodology yields a mathematically precise description of the nonlinear behavior at criticality. The higher-dimensional cases ($d=2,3$) present additional challenges, since the driving space-time white noise is too singular to be handled by classical techniques. In these settings, the works \cite{MW17, GMW23} employ the modern theory of singular stochastic PDEs, such as regularity structures, to construct the limiting dynamics and to establish the convergence of the Glauber spin system to its continuum counterpart. Together, these results form a coherent picture of the universality of the stochastic Allen-Cahn equation as the scaling limit of Glauber dynamics at criticality.

{\bf Kawasaki dynamics.}
The study of nonlinear fluctuations for the Kawasaki dynamics remains largely open. To the best of the author's knowledge, the only partial progress in this direction is presented in \cite[Chapter~4]{IM18}, where a conditional result is established. More precisely, they proved tightness of the rescaled fluctuation fields and showed that any limit point satisfies the martingale problem formally associated with the stochastic Cahn-Hilliard equation. Their analysis combines martingale decompositions, block replacement arguments, and spectral gap estimates to control the nonlinear contributions arising near the critical temperature. A crucial limitation, however, is that the validity of their result hinges on an unproven conjecture, which concerns the control of a certain component of the discrete fluctuation equation; see \cite[(4.23) and Conjecture~4.2.6]{IM18} for the precise formulation. Consequently, the rigorous derivation of the stochastic Cahn-Hilliard equation from the Kawasaki dynamics at criticality remains an open and challenging problem, and its resolution would represent a significant step toward understanding nonlinear fluctuation phenomena in conservative spin systems.

{\bf Comparison with the stochastic PDE.}  
In contrast to discrete dynamical Ising-Kac particle systems, the SPDE \eqref{intro-spde-2} features a nonlocal Kac interactions and nonlinear multiplicative noise. Guided by the framework of fluctuating hydrodynamics, the mesoscopic noise is formally replaced by a temporally white but spatially correlated noise. This spatial regularization enhances the spatial regularity of solutions to \eqref{intro-spde-2}, so that, unlike in the particle setting, the solutions are function-valued. In the particle framework, the analysis primarily focuses on the convergence of the laws of the fluctuation fields, with the generators of the dynamics playing a central role. By contrast, our SPDE approach in fluctuating hydrodynamics begins with rescaled entropy dissipation estimates and investigates the properties of trajectories. As a result, both the questions addressed and the methodologies employed differ significantly between discrete particle models and their continuum SPDE analogues.

\subsection{Comments on the literature}
\ 

{\bf Fluctuating Hydrodynamics. }
The Macroscopic Fluctuation Theory (MFT) provides a systematic framework for understanding systems far from equilibrium, extending and refining classical near-equilibrium linear approximations (see Bertini et al. \cite{BDGJL}, Derrida \cite{Derrida}). At its core, MFT is based on an ansatz concerning the large deviation principles for interacting particle systems. Closely related is the theory of Fluctuating Hydrodynamics (FHD), which models microscopic fluctuations in accordance with statistical mechanics and non-equilibrium thermodynamics. In this framework, conservative stochastic PDEs are postulated to capture the essential features of fluctuations in non-equilibrium systems (see \cite{LL87}, \cite{HS}). The fundamental ansatz of MFT can thus be derived from the zero-noise large deviation principles associated with such conservative stochastic PDEs. A prototypical example within FHD is the Dean-Kawasaki equation (see \cite{D96,K98}), which effectively captures fluctuation behavior in mean-field interacting particle systems. 

{\bf Conservative stochastic PDEs}
To advance the understanding of the Dean-Kawasaki equation and structurally related models, a range of analytic techniques has been developed. Debussche and Vovelle \cite{DV10} studied the Cauchy problem for stochastic conservation laws via the concept of kinetic solutions. This notion was subsequently extended to parabolic-hyperbolic stochastic PDEs with conservative noise by Gess and Souganidis \cite{GS17}, Fehrman and Gess \cite{FG19}, and Dareiotis and Gess \cite{DG20}. These methodologies have since been adapted to the analysis of Dean-Kawasaki-type equations. In the case of local interactions, Fehrman and Gess \cite{FG24} established the well-posedness of functional-valued solutions to the Dean-Kawasaki equation with correlated noise. Building on this framework, they further investigated small noise large deviations in \cite{FG23}. They later extend the results to the whole-space case \cite{FG25}, and Fehrman extends the well-posedness result to non-stationary Stratonovich noise \cite{F25}. More recently, Clini and Fehrman \cite{CF23} developed a central limit theorem for the nonlinear Dean-Kawasaki equation with correlated noise, while Gess, the author, and Zhang \cite{GWZ24} derived higher-order fluctuation expansions for Dean-Kawasaki-type systems. For models with nonlocal interactions, Wang, the author, and Zhang \cite{WWZ22}, as well as the author and Zhang \cite{WZ24}, analyzed the well-posedness and large deviations for Dean-Kawasaki equations with singular interactions, with applications to the fluctuating Ising-Kac-Kawasaki equation. For models with Dirichlet boundary conditions, well-posedness, large deviation, and fluctuation results have been established in Popat \cite{Shyam25,Shyam25-fluc}. For weak error estimates, we refer the reader to \cite{DKP24,DJP25,CF23,CFIR23}, among others. Additionally, Martini and Mayorcas \cite{AA25,AA24} established well-posedness and large deviation principles for an additive noise approximation of the Keller-Segel Dean-Kawasaki equation. Recently, the fluctuating kinetic equations have been discussed in \cite{FMJ25} and \cite{HWZ25}. Further research on the regularized Dean-Kawasaki equation for second-order particle systems has been conducted in \cite{CS23,CSZ19,CSZ20} and has found applications in computational chemistry \cite{JLR25}.

{\bf Dynamical Ising-Kac models. }
We also highlight several works relevant to the dynamical Ising-Kac model. The long-standing conjecture regarding nonlinear fluctuations of Glauber dynamics in dimensions $d=1,2,3$ has been fully resolved: the rescaled density field converges to the dynamical $\Phi^4_d$ model. For details, see \cite{BPRS94, FR95} for $d=1$, \cite{MW17} for $d=2$, and \cite{GMW23} for $d=3$. For the case of Kawasaki dynamics, \cite{IM18} established a conditional result on nonlinear fluctuations in dimension $d=1$. Regarding the macroscopic limits of Kawasaki dynamics, we refer the reader to \cite{GL97}, with further developments and analytical results provided in \cite{VY97, GL15, Gia91, LOP91}. 

\subsection{Structure of the paper}
This paper is organized as follows. In Section \ref{sec-2}, we introduce basic notations, state the assumptions on the noise and initial data, and provide the definitions of the solution concepts. Section \ref{sec-3} is devoted to presenting two uniform estimates for \eqref{equgamma-0}, namely, the uniform entropy dissipation estimates and uniform time-regularity estimates. The convergence of \eqref{equgamma-0} to the stochastic Cahn-Hilliard equation is rigorously proved in Section \ref{sec-4}. In Section \ref{sec-5}, we establish the multi-scale large deviations for \eqref{smallnoiseeq}. Finally, in Section \ref{sec-6}, we show the $\Gamma$-convergence of the rate function \eqref{rate-gamma-intro} to the rate function \eqref{Cahn-Hilliar-rate-intro}. Appendix \ref{sec-app-A} contains a proof of the convergence of the stochastic Cahn-Hilliard equation with correlated noise to the solution with space-time white noise. Finally, Appendix \ref{sec-app-B} discusses the challenges of improving the convergence results to the case of square-root type diffusion coefficients.

\section{Preliminaries}\label{sec-2}
\subsection{Notations}
Let $(\Omega, \mathcal{F}, \mathbb{P}, \{\mathcal{F}_t\}_{t \in [0,T]}, \{B^k(t)\}_{t \in [0,T], k \in \mathbb{N}})$ be a stochastic basis. Without loss of generality, we assume that the filtration $\{\mathcal{F}_t\}_{t \in [0,T]}$ is complete and that $\{B^k(t)\}_{t \in [0,T]}$, $k \in \mathbb{N}$, are mutually independent $\{\mathcal{F}_t\}_{t \in [0,T]}$-Wiener processes taking values in $\mathbb{R}^1$. Expectation with respect to the probability measure $\mathbb{P}$ is denoted by $\mathbb{E}$.

For each $p \in [1, \infty]$, we denote by $\|\cdot\|_{L^p(\mathbb{T}^1)}$ the norm on the Lebesgue space $L^p(\mathbb{T}^1)$. The inner product on $L^2(\mathbb{T}^1)$ is denoted by $\langle \cdot, \cdot \rangle$. Let $C^\infty(\mathbb{T}^1 \times (0, \infty))$ denote the space of infinitely differentiable functions on $\mathbb{T}^1 \times (0, \infty)$, and let $C_c^\infty(\mathbb{T}^1 \times (0, \infty))$ be the subspace of compactly supported functions. For a non-negative integer $k$ and $p \in [1, \infty]$, we denote by $W^{k,p}(\mathbb{T}^1)$ the standard Sobolev space on the torus $\mathbb{T}^1$. In particular, we write $H^a(\mathbb{T}^1) := W^{a,2}(\mathbb{T}^1)$ and denote its topological dual by $H^{-a}(\mathbb{T}^1)$. Finally, we let $(e_k)_{k \geq 0}$ denote the Fourier basis of $L^2(\mathbb{T}^1)$.

\subsection{Derivation of the equation}\label{subsec-equation}
To transform the conjecture in \cite{GJE99} into a well-justified and rigorous mathematical problem, we present in this section a detailed derivation of equation \eqref{equgamma}. We begin with the original equation introduced in \cite{GJE99} at an informal level. Consider
\begin{align}\label{SPDE-1}
\partial_t \rho = \partial_{xx}^2 \rho - \beta \partial_x \big[(1 - \rho^2) \partial_x J \ast \rho \big] - \sqrt{2} \gamma^{1/2} \partial_x \big( \sigma(1 - \rho^2) \xi \big),\quad \text{on }\gamma^{-1/3}\mathbb{T}^1\times(0,\infty),  
\end{align}
with initial data $\rho_0$. Due to the singularity of the space-time white noise $\xi$, equation \eqref{SPDE-1} lies in the super-critical regime within the framework of singular SPDEs \cite{H14,GIP12}. For each $\gamma\in(0,1]$ and $a \in \mathbb{R}$, define the rescaled variable
$$
u_{\gamma}(x,t) := \gamma^{-1/3} \rho(\gamma^{-1/3} x, \gamma^{-4/3} t), \quad \text{with} \quad \beta = 1 + a \gamma^{2/3},\quad \text{for every }(x,t)\in\mathbb{T}^1\times[0,T]. 
$$

We proceed to formally derive the equation satisfied by $u_{\gamma}$. By the chain rule, we have
\begin{align*}
\partial_t u_{\gamma} 
= \gamma^{-1/3} \partial_t \rho(\gamma^{-1/3} x, \gamma^{-4/3} t) 
= \gamma^{-5/3} \big(\partial_t \rho\big)(\gamma^{-1/3} x, \gamma^{-4/3} t).
\end{align*}
Substituting \eqref{SPDE-1} into the right-hand side, we obtain
\begin{align*}
\partial_t u_{\gamma} 
= \; & \gamma^{-5/3} \big(\partial_{xx}^2 \rho\big)(\gamma^{-1/3} x, \gamma^{-4/3} t) 
- \gamma^{-5/3} (1 + a \gamma^{2/3}) \partial_x \big[(1 - \rho^2) \partial_x J \ast \rho \big](\gamma^{-1/3} x, \gamma^{-4/3} t) \\
& - \gamma^{-5/3 + 1/2} \sqrt{2} \partial_x \big( \sigma(1 - \rho^2) \xi \big)(\gamma^{-1/3} x, \gamma^{-4/3} t).
\end{align*}

For the dissipation term, applying the chain rule yields
$$
\gamma^{-5/3} \big(\partial_{xx}^2 \rho\big)(\gamma^{-1/3} x, \gamma^{-4/3} t) 
= \gamma^{-1} \partial_{xx}^2 \big(\rho(\gamma^{-1/3} x, \gamma^{-4/3} t)\big) 
= \gamma^{-2/3} \partial_{xx}^2 u_{\gamma}.
$$

Regarding the non-local term, we have
\begin{align*}
& - \gamma^{-5/3} (1 + a \gamma^{2/3}) \partial_x \big[(1 - \rho^2) \partial_x J \ast \rho \big](\gamma^{-1/3} x, \gamma^{-4/3} t) \\
= \; & - \gamma^{-4/3} (1 + a \gamma^{2/3}) \partial_x \big[ (1 - \gamma^{2/3} u_{\gamma}^2) (\partial_x J \ast \rho)(\gamma^{-1/3} x, \gamma^{-4/3} t) \big].
\end{align*}
By a change of variables, it follows that
\begin{align*}
(\partial_x J \ast \rho)(\gamma^{-1/3} x, \gamma^{-4/3} t) 
&= \int_{\gamma^{-1/3} \mathbb{T}^1} (\partial_x J)(\gamma^{-1/3} x - \tilde{y}) \rho(\tilde{y}, \gamma^{-4/3} t) d \tilde{y} \\
&= \gamma^{-1/3} \int_{\mathbb{T}^1} (\partial_x J)(\gamma^{-1/3} (x - y)) \rho(\gamma^{-1/3} y, \gamma^{-4/3} t) dy.
\end{align*}
Defining the rescaled kernel $J_{\gamma^{1/3}}(\cdot) := \gamma^{-1/3} J(\gamma^{-1/3} \cdot)$, we deduce
\begin{align*}
& - \gamma^{-5/3} (1 + a \gamma^{2/3}) \partial_x \big[(1 - \rho^2) \partial_x J \ast \rho \big](\gamma^{-1/3} x, \gamma^{-4/3} t) \\
= \; & - \gamma^{-2/3} (1 + a \gamma^{2/3}) \partial_x \big[ (1 - \gamma^{2/3} u_{\gamma}^2) J_{\gamma^{1/3}} \ast \partial_x u_{\gamma} \big].
\end{align*}

For the noise term, we write
\begin{align*}
& \gamma^{-5/3 + 1/2} \sqrt{2} \partial_x \big(\sigma(1 - \rho^2) \xi \big)(\gamma^{-1/3} x, \gamma^{-4/3} t) \\
= \; & \gamma^{-5/6} \sqrt{2} \partial_x \big( \sigma(1 - \gamma^{2/3} u_{\gamma}^2) \xi(\gamma^{-1/3} x, \gamma^{-4/3} t) \big).
\end{align*}
Applying the scaling property of space-time white noise yields the distributional identity
$$
\xi(\gamma^{-1/3} x, \gamma^{-4/3} t) \stackrel{d}{=} \gamma^{5/6} \xi,
$$
where $\stackrel{d}{=}$ denotes equality in distribution. Consequently, the noise term can be replaced by
$$
\sqrt{2} \partial_x \big( \sigma(1 - \gamma^{2/3} u_{\gamma}^2) \xi \big).
$$

Putting these together, we find that $u_{\gamma}$ formally satisfies
$$
\partial_t u_{\gamma} = \gamma^{-2/3} \partial_{xx}^2 u_{\gamma} - \gamma^{-2/3} (1 + a \gamma^{2/3}) \partial_x \big[ (1 - \gamma^{2/3} u_{\gamma}^2) J_{\gamma^{1/3}} \ast \partial_x u_{\gamma} \big] - \sqrt{2} \partial_x \big( \sigma(1 - \gamma^{2/3} u_{\gamma}^2) \xi \big),
$$
with initial data $u_{\gamma,0}(x) = \gamma^{-1/3} \rho_0(\gamma^{-1/3} x)$.

In light of the current understanding of the analytical solution theory for fluctuating hydrodynamics (see \cite{Ottinger}), it is natural to consider a Stratonovich modification of the noise and to replace the space-time white noise by a spatially correlated noise with a small correlation length. To this end, we introduce a correlation length parameter $\delta^{1/3} > 0$ independent of the scaling parameter $\gamma$. Let $\eta_{\delta} = \delta^{-1/3} \eta(\delta^{-1/3} \cdot)$ denote a standard mollifier. This motivates studying the following equation:
\begin{align*}
\partial_t u_{\gamma} = & \; \gamma^{-2/3} \partial_{xx}^2 u_{\gamma} - \gamma^{-2/3} (1 + a \gamma^{2/3}) \partial_x \big[ (1 - \gamma^{2/3} u_{\gamma}^2) J_{\gamma^{1/3}} \ast \partial_x u_{\gamma} \big] \\
& - \sqrt{2} \partial_x \big( \sigma(1 - \gamma^{2/3} u_{\gamma}^2) \circ \xi_{\delta} \big).
\end{align*}
Upon rewriting in It\^o form, we obtain
\begin{align*}
\partial_t u_{\gamma} = & \; \gamma^{-2/3} \partial_{xx}^2 u_{\gamma} - \gamma^{-2/3} (1 + a \gamma^{2/3}) \partial_x \big[ (1 - \gamma^{2/3} u_{\gamma}^2) J_{\gamma^{1/3}} \ast \partial_x u_{\gamma} \big] \\
& - \sqrt{2} \partial_x \big( \sigma(1 - \gamma^{2/3} u_{\gamma}^2) \xi_{\delta} \big) \\
& + 4 \partial_x \Big( F_{1,\delta} \sigma'(1 - \gamma^{2/3} u_{\gamma}^2)^2 \gamma^{4/3} u_{\gamma}^2 \partial_x u_{\gamma} - \frac{1}{2} F_{2,\delta} \sigma(1 - \gamma^{2/3} u_{\gamma}^2) \sigma'(1 - \gamma^{2/3} u_{\gamma}^2) \gamma^{2/3} u_{\gamma} \Big), 
\end{align*}
on $\mathbb{T}^1\times(0,\infty)$. Since the Fourier basis $(e_k)_{k \geq 0}$ is used, it follows that $F_{2,\delta} = 0$. This recovers equation \eqref{equgamma}.

\subsection{Assumptions of the noise}
Let $\eta$ be a standard convolution kernel, and let $\eta_{\delta}=\delta^{-1/3}\eta(\cdot/\delta^{1/3})$, for every $\delta>0$. For each $k \geq 0$ and $\delta > 0$, define $f_{\delta,k} := \eta_{\delta} \ast e_k$. We introduce Brownian motions 
\begin{align}\label{brownian}
W(t) := \sum_{k \geq 0} e_k B^k(t)\quad\text{and}\quad W_{\delta}(t) := \sum_{k \geq 0} f_{\delta,k} B^k(t),\quad \forall t\in[0,T]. 
\end{align}

Furthermore, define the following coefficient functions:
\begin{equation*}
	F_{1,\delta} := \sum_{k \geq 0} |f_{\delta,k}|^2, \quad F_{2,\delta} := \frac{1}{2} \sum_{k \geq 0} \partial_x f_{\delta,k}^2, \quad F_{3,\delta} := \sum_{k \geq 0} |\partial_x f_{\delta,k}|^2.
\end{equation*}

We remark that, since $(e_k)_{k \geq 0}$ is the Fourier basis, the coefficients $F_{1,\delta}$ and $F_{3,\delta}$ are constants for every $\delta > 0$. Moreover, we have the identity $F_{2,\delta} = 0$ for all $\delta > 0$. 

We denote the space-time white noise and the spatial correlated noise by the distributional time-derivatives $\xi=\frac{d}{dt}W$ and $\xi_{\delta}=\frac{d}{dt}W_{\delta}$, respectively.

\subsection{Well-posedness of the SPDEs}
We begin by stating a precise assumptions on the diffusion coefficient $\sigma(\cdot)$ and the interaction kernel $J$. 

\begin{assumption}\label{Assump-sigma}
	\textbf{(Regular coefficients)} Assume that $\sigma(\cdot) \in \mathrm{C}_{\mathrm{loc}}^1((0,\infty))$. Furthermore, $\sigma$ satisfies the following properties:
	\begin{enumerate}
		\item[(1)] $\sigma \in C([0,\infty)) \cap C^{\infty}((0,\infty))$, with $\sigma(0) = 0$, $\sigma(1) = 1$, and $\sigma' \in C_c^{\infty}([0,\infty))$;
		\item[(2)] there exists a constant $c \in (0,\infty)$ such that for all $\zeta \in [0,\infty)$,
		\begin{align}
			|\sigma(\zeta)| \le c \sqrt{\zeta};
		\end{align}
	\end{enumerate}
\end{assumption}
\begin{assumption}\label{Assump-J}
	\textbf{(Regular interactions)} Assume that $J\in C^{\infty}(\mathbb{T}^1;\mathbb{R}^1_+)$ is an even function with $supp\ J\subset(-1/2,1/2)$ and $\|J\|_{L^1(\mathbb{T}^1)}=1$. 
	\end{assumption}

In the sequel, we always assume that the coefficient $\sigma(\cdot)$ and the interaction kernel $J$ satisfy Assumption \ref{Assump-sigma} and Assumption \ref{Assump-J}, respectively, and we shall not reiterate this assumption thereafter. By calculating the covariation structure of the Brownian motions, we may rewrite equation \eqref{equgamma-0} in It\^o form as follows:
\begin{align}\label{equgamma}
\partial_tu_{\gamma,\delta} =&\ \gamma^{-2/3}\partial_{xx}^2 u_{\gamma,\delta} - \gamma^{-2/3}(1+a\gamma^{2/3})\partial_{x}\Big[(1-\gamma^{2/3}u_{\gamma,\delta}^2)J_{\gamma^{1/3}}\ast\partial_{x} u_{\gamma,\delta}\Big]\notag\\
& - \sqrt{2}\partial_{x}\Big(\sigma(1-\gamma^{2/3}u_{\gamma,\delta}^2)\xi_{\delta}\Big) + 4\partial_{x}\Big(F_{1,\delta}\sigma'(1-\gamma^{2/3}u_{\gamma,\delta}^2)^2\gamma^{4/3}u_{\gamma,\delta}^2\partial_{x}u_{\gamma,\delta}\Big).
\end{align}

Recall that the mathematical entropy $\Psi$ is defined by \eqref{entropy-intro}. We introduce the space of functions with finite entropy:
\begin{align}\label{ENT}
\overline{\mathrm{Ent}}(\mathbb{T}^{d}) = \left\{ \rho : -1 \leq \rho \leq 1 \ \text{a.e.},\ \text{and} \ \int_{\mathbb{T}^{d}} \Psi(\rho(x))\,\mathrm{d}x < \infty \right\}.
\end{align}
In what follows, we provide the definitions of weak solutions to both equation \eqref{equgamma} and the limiting equation \eqref{CahnHilliard}. 

\begin{definition}\label{def-1}
Let $\gamma,\delta > 0$, $a \in \mathbb{R}$, and let $u_{\gamma,0} \in \overline{\mathrm{Ent}}(\mathbb{T}^{1})$. A \emph{weak solution} of equation \eqref{equgamma} with initial data $u_{\gamma,0}$ is an $L^2(\mathbb{T}^1)$-valued, $\{\mathcal{F}(t)\}_{t\in[0,T]}$-predictable process $u_{\gamma,\delta}$ such that almost surely,
$$
u_{\gamma,\delta} \in L^{\infty}([0,T];L^2(\mathbb{T}^1))\cap L^2([0,T]; H^1(\mathbb{T}^1)), \quad \text{and} \quad -\gamma^{-1/3} \leq u_{\gamma,\delta} \leq \gamma^{-1/3}.
$$
Moreover, it holds almost surely that for every $\psi \in C^{\infty}(\mathbb{T}^1)$ and every $t \in [0,T]$, 
\begin{align}
\int_{\mathbb{T}^1} u_{\gamma,\delta}(t)\psi(x)\,dx &= \int_{\mathbb{T}^1} u_{\gamma,0} \psi\,dx - \gamma^{-2/3} \int_0^t\int_{\mathbb{T}^1} \partial_x u_{\gamma,\delta} \partial_x \psi\,dx\,ds\notag\\
&\quad + \gamma^{-2/3}(1 + a\gamma^{2/3}) \int_0^t \int_{\mathbb{T}^1} \big((1 - \gamma^{2/3} u_{\gamma,\delta}^2) J_{\gamma^{1/3}} \ast \partial_x u_{\gamma,\delta}\big) \partial_x \psi\,dx\,ds\notag\\
&\quad + \sqrt{2} \int_0^t \int_{\mathbb{T}^1} \partial_x \psi\, \sigma(1 - \gamma^{2/3} u_{\gamma,\delta}^2)\,dW_{\delta}\,dx\notag\\
&\quad - 4 \int_0^t \int_{\mathbb{T}^1} F_{1,\delta}\, \sigma'(1 - \gamma^{2/3} u_{\gamma,\delta}^2)^2\, \gamma^{4/3} u_{\gamma,\delta}^2\, \partial_x u_{\gamma,\delta} \partial_x \psi\,dx\,ds.
\end{align}
\end{definition}

\begin{definition}\label{def-2}
Let $u_0 \in H^{-1}(\mathbb{T}^1)$ and $a \in \mathbb{R}$. A \emph{weak solution} of equation \eqref{CahnHilliard} with initial data $u_0$ is an $H^{-1}(\mathbb{T}^1)$-valued, $\{\mathcal{F}(t)\}_{t\in[0,T]}$-predictable process $u$ such that almost surely,
$$
u \in C([0,T]; H^{-1}(\mathbb{T}^1)) \cap L^2([0,T]; H^1(\mathbb{T}^1)) \cap L^4([0,T]; L^4(\mathbb{T}^1)).
$$
Moreover, it holds almost surely that for every $\psi \in C^{\infty}(\mathbb{T}^1)$ and every $t \in [0,T]$, 
\begin{align*}
\int_{\mathbb{T}^1} u(t)\psi\,dx &= \int_{\mathbb{T}^1} u_0 \psi\,dx + \frac{D}{2} \int_0^t \int_{\mathbb{T}^1} \partial_x u\, \partial_{xxx}^3 \psi\,dx\,ds\\
&\quad + \int_0^t \int_{\mathbb{T}^1} V'(u)\, \partial_{xx}^2 \psi\,dx\,ds + \sqrt{2} \int_0^t \int_{\mathbb{T}^1} \partial_x \psi\, dW\,dx.
\end{align*}
\end{definition}
In the following, we establish the global-in-time existence and uniqueness of weak solutions to both \eqref{equgamma} and \eqref{CahnHilliard}.

\begin{theorem}
For every $\gamma\in(0,1]$, let $u_{\gamma,0} \in \overline{\mathrm{Ent}}(\mathbb{T}^1)$. Then there exists a unique weak solution to \eqref{equgamma} in the sense of Definition~\ref{def-1} with initial data $u_{\gamma,0}$.
\end{theorem}
\begin{proof}
The proof can be carried out using the same approach as in \cite{FG24,WWZ22}. In fact, the methodology developed in \cite{FG24,WWZ22} establishes the well-posedness of \eqref{equgamma} even in the case where the diffusion coefficient $\sigma(\cdot)$ takes the square-root form. As a regularized version of \eqref{intro-spde-3}, the equation \eqref{equgamma} with a diffusion coefficient $\sigma(\cdot)$ satisfying Assumption \ref{Assump-sigma} allows for the equivalence between the notions of renormalized kinetic solutions (see \cite[Definition 7.2]{WWZ22} for details) and weak solutions. Consequently, existence can be established via $L^2(\mathbb{T}^d)$-based energy estimates combined with a compactness argument. Uniqueness follows from the kinetic formulation through the method of doubling variables. 
\end{proof}

\begin{theorem}
Let $u_0 \in H^{-1}(\mathbb{T}^1)$ and $a \in \mathbb{R}$. Then there exists a unique weak solution to \eqref{CahnHilliard} with initial data $u_0$ in the sense of Definition~\ref{def-2}.
\end{theorem} 
\begin{proof}
The proof can be found in \cite{DZ92}. 	
\end{proof}

\section{Uniform estimates and compactness}\label{sec-3}
In this section, we establish two uniform estimates in the parameter $\gamma\in(0,1]$ for the weak solution of \eqref{equgamma}, with $\delta > 0$ fixed throughout. For notational convenience, we denote the weak solution of \eqref{equgamma} by $u_{\gamma}$, and henceforth omit the explicit mention that $\delta > 0$ is given and fixed. In Subsection \ref{subsec-3-1}, we first derive a uniform estimate for the dissipation of the mathematical entropy. Subsequently, in Subsection \ref{subsec-3-2}, we provide a uniform estimate on the time regularity. As a consequence of these two estimates, we deduce the tightness of the family of solutions to \eqref{equgamma} in $L^2([0,T];L^2(\mathbb{T}^1))$.

\subsection{Uniform rescaled entropy dissipation estimates}\label{subsec-3-1}
For every $\gamma\in(0,1]$, we recall that the rescaled entropy function is defined by
\begin{align}\label{rescaled-entropy}
\Psi_{\gamma}(\zeta) = \frac{\gamma^{-1/3}}{2} \Big[(1+\gamma^{1/3}\zeta)\log(\gamma^{1/3}\zeta+1) + (1-\gamma^{1/3}\zeta)\log(1-\gamma^{1/3}\zeta) - 2\Big], \quad \zeta \in [-\gamma^{-1/3}, \gamma^{-1/3}], 
\end{align}
where $\Psi_{\gamma}(\gamma^{-1/3})$ and $\Psi_{\gamma}(-\gamma^{-1/3})$ are defined by the continuum limits
$$
\Psi_{\gamma}(\gamma^{-1/3}) := \lim_{\zeta \uparrow \gamma^{-1/3}} \Psi(\zeta) = \gamma^{-1/3} (\log 2 - 1), \quad
\Psi_{\gamma}(-\gamma^{-1/3}) := \lim_{\zeta \downarrow -\gamma^{-1/3}} \Psi(\zeta) = \gamma^{-1/3} (\log 2 - 1).
$$
 Moreover, we denote by $\psi_{\gamma}$ the derivative of $\Psi_{\gamma}$, given explicitly by
\begin{align}\label{entropy-derivative} 
\psi_{\gamma}(\zeta) = \Psi'_{\gamma}(\zeta) = \frac{1}{2} \log\left(\frac{1 + \gamma^{1/3} \zeta}{1 - \gamma^{1/3} \zeta}\right), \quad \zeta \in (-\gamma^{-1/3}, \gamma^{-1/3}).
\end{align}
A straightforward computation shows that the second derivative of $\Psi_{\gamma}$ takes the form
\begin{align*}
\psi_{\gamma}'(\zeta) = \frac{\gamma^{1/3}}{1 - \gamma^{2/3} \zeta^2}, \quad \zeta \in (-\gamma^{-1/3}, \gamma^{-1/3}).
\end{align*}
In the special case $\gamma = 1$, we write $\Psi := \Psi_1$ and $\psi := \psi_1$ to denote the standard entropy function and its derivative, respectively. In the following, we establish several fundamental analytic properties of the rescaled entropy function, which play a key role in the subsequent analysis. 
\begin{lemma}\label{lem-tech-1}
For every $\gamma\in(0,1]$, let $\Psi_{\gamma}$ be defined as in \eqref{rescaled-entropy}. Then the following properties hold. First, for all $\zeta \in [-\gamma^{-1/3}, \gamma^{-1/3}]$, we have the lower bound
\begin{align}\label{property-1}
\Psi_{\gamma}(\zeta) \geq \frac{1}{2} \gamma^{1/3} \zeta^2 - \gamma^{-1/3}.
\end{align}
In addition, the rescaled entropy function is strictly negative throughout its domain:
\begin{align}\label{property-2}
\Psi_{\gamma}(\zeta) < 0, \quad \text{for all } \zeta \in [-\gamma^{-1/3}, \gamma^{-1/3}].
\end{align} 
\end{lemma}
\begin{proof}
Recall that $\psi_{\gamma}$ is defined in \eqref{entropy-derivative}. A direct computation yields
\begin{align*}
\psi_{\gamma}'(\zeta) = \frac{\gamma^{1/3}}{1 - \gamma^{2/3} \zeta^2} \geq \gamma^{1/3}, \quad \text{for all } \zeta \in (-\gamma^{-1/3}, \gamma^{-1/3}).
\end{align*}
Since $\psi_{\gamma}(0) = 0$, it follows that
\begin{align*}
\psi_{\gamma}(\zeta) \geq \gamma^{1/3} \zeta, \quad \text{for all } \zeta \in (0, \gamma^{-1/3}).
\end{align*}
Integrating this inequality and using the symmetry of $\Psi_{\gamma}$, we deduce that
\begin{align*}
\Psi_{\gamma}(\zeta) \geq \frac{1}{2} \gamma^{1/3} \zeta^2 - \gamma^{-1/3}, \quad \text{for all } \zeta \in [-\gamma^{-1/3}, \gamma^{-1/3}],
\end{align*}
which establishes \eqref{property-1}.

Next, observe that $\psi_{\gamma}(\zeta) > 0$ for all $\zeta \in (0, \gamma^{-1/3})$, and $\psi_{\gamma}(\zeta) < 0$ for all $\zeta \in (-\gamma^{-1/3}, 0)$. Hence, the function $\Psi_{\gamma}$ attains its maximum at the boundary points $\zeta = \pm \gamma^{-1/3}$, where
$$
\sup_{\zeta \in [-\gamma^{-1/3}, \gamma^{-1/3}]} \Psi_{\gamma}(\zeta) = \Psi_{\gamma}(\gamma^{-1/3}) = \Psi_{\gamma}(-\gamma^{-1/3}) = \frac{\gamma^{-1/3}}{2} (2 \log 2 - 2) < 0.
$$
This completes the proof of \eqref{property-2}.
\end{proof}

In the following, we show the uniform rescaled entropy dissipation estimates. 
\begin{proposition}\label{prp-entropydissipation}
Assume that the initial data, the coefficient, and the interaction kernel $J$ satisfy Assumptions (A1), \ref{Assump-sigma}, and \ref{Assump-J}, respectively. For every $\gamma\in(0,1]$, let $u_{\gamma}$ denote the weak solution of \eqref{equgamma} with initial data $u_{\gamma,0}$. Let $\Psi_{\gamma}$ be defined as in \eqref{rescaled-entropy}. Then there exists a constant $C > 0$, independent of $\gamma$ and $\delta$, such that the following uniform entropy estimate holds: \begin{align}\label{uniformentropy}
\sup_{t \in [0,T]} \mathbb{E} \left( \int_{\mathbb{T}^1} \Psi_{\gamma}(u_{\gamma}(t))\, dx \right)
&+ \gamma^{-1/3} \mathbb{E} \int_0^T \int_{\mathbb{T}^1} \frac{|\partial_x u_{\gamma}|^2}{1 - \gamma^{2/3} u_{\gamma}^2}\, dx\, dt \notag \\
&+ \gamma^{-1/3}(1 + a\gamma^{2/3}) \mathbb{E} \int_0^T \int_{\mathbb{T}^1} \partial_x u_{\gamma} \, (J_{\gamma^{1/3}} \ast \partial_x u_{\gamma})\, dx\, ds \notag \\
&\lesssim \int_{\mathbb{T}^1} \Psi_{\gamma}(u_{\gamma,0})\, dx + CT \gamma^{1/3} \delta^{-1}, \quad \text{for all } \gamma\in(0,1].
\end{align}

As a consequence, we also obtain the estimate
\begin{align*}
\gamma^{1/3} \sup_{t \in [0,T]} \mathbb{E} \left( \int_{\mathbb{T}^1} \left( \frac{1}{2} |u_{\gamma}(t)|^2 - 1 \right) dx \right)
&+ \gamma^{-1/3} \mathbb{E} \int_0^T \int_{\mathbb{T}^1} \frac{|\partial_x u_{\gamma}|^2}{1 - \gamma^{2/3} u_{\gamma}^2}\, dx\, dt \\
&+ \gamma^{-1/3}(1 + a\gamma^{2/3}) \mathbb{E} \int_0^T \int_{\mathbb{T}^1} \partial_x u_{\gamma} \, (J_{\gamma^{1/3}} \ast \partial_x u_{\gamma})\, dx\, ds \\
&\lesssim \int_{\mathbb{T}^1} \Psi_{\gamma}(u_{\gamma,0})\, dx + CT \gamma^{1/3} \delta^{-1}, \quad \text{for all } \gamma\in(0,1].
\end{align*}

Furthermore, if $a < 0$, then the following uniform energy estimate holds:
\begin{align}\label{uniformestimate}
\sup_{t \in [0,T]} \mathbb{E} \left( \int_{\mathbb{T}^1} \frac{1}{2} |u_{\gamma}(t)|^2\, dx \right)
&+ \mathbb{E} \int_0^T \int_{\mathbb{T}^1} \frac{u_{\gamma}^2 |\partial_x u_{\gamma}|^2}{1 - \gamma^{2/3} u_{\gamma}^2}\, dx\, dt \notag \\
&- a \mathbb{E} \int_0^T \|\partial_x u_{\gamma}\|_{L^2(\mathbb{T}^1)}^2\, ds \lesssim CT \delta^{-1}, \quad \text{for all } \gamma\in(0,1].
\end{align}
\end{proposition}
\begin{proof}
We begin by introducing a family of smooth approximations to the entropy function. For every $\bar{\delta} > 0$, define
\begin{align*}
\Psi_{\gamma,\bar{\delta}}(\zeta) &= \frac{\gamma^{-1/3}}{2(1 + \bar{\delta})} \left[ (1 + \bar{\delta} + \gamma^{1/3} \zeta) \log(1 + \bar{\delta} + \gamma^{1/3} \zeta) + (1 + \bar{\delta} - \gamma^{1/3} \zeta) \log(1 + \bar{\delta} - \gamma^{1/3} \zeta) - 2 \right], \\
\psi_{\gamma,\bar{\delta}}(\zeta) &= \Psi_{\gamma,\bar{\delta}}'(\zeta) = \frac{1}{2(1 + \bar{\delta})} \log\left( \frac{1 + \bar{\delta} + \gamma^{1/3} \zeta}{1 + \bar{\delta} - \gamma^{1/3} \zeta} \right), \quad \text{for all } \zeta \in [-\gamma^{-1/3}, \gamma^{-1/3}].
\end{align*}
A direct computation yields
\begin{align*}
\psi_{\gamma,\bar{\delta}}'(\zeta) = \frac{\gamma^{1/3}}{(1 + \bar{\delta})^2 - \gamma^{2/3} \zeta^2}, \quad \text{for all } \zeta \in [-\gamma^{-1/3}, \gamma^{-1/3}].
\end{align*}

Referring to \cite[Theorem 3.1]{Krylov}, we apply It\^o's formula to the functional $\Psi_{\gamma,\bar{\delta}}(u_{\gamma})$. By employing integration by parts and utilizing the regularity properties of $u_{\gamma}$, we obtain that, almost surely, for every $t \in [0,T]$,
\begin{align*}
& \int_{\mathbb{T}^1} \Psi_{\gamma,\bar{\delta}}(u_{\gamma}(t)) \, dx
+ \gamma^{-2/3} \int_0^t \int_{\mathbb{T}^1} \psi'_{\gamma,\bar{\delta}}(u_{\gamma}) |\partial_x u_{\gamma}|^2 \, dx \, ds \\
=\, & \int_{\mathbb{T}^1} \Psi_{\gamma,\bar{\delta}}(u_{\gamma,0}) \, dx
+ \gamma^{-2/3} (1 + a \gamma^{2/3}) \int_0^t \int_{\mathbb{T}^1} \psi'_{\gamma,\bar{\delta}}(u_{\gamma}) \partial_x u_{\gamma} (1 - \gamma^{2/3} u_{\gamma}^2) J_{\gamma^{1/3}} * \partial_x u_{\gamma} \, dx \, ds \\
& + \sqrt{2} \int_0^t \int_{\mathbb{T}^1} \psi'_{\gamma,\bar{\delta}}(u_{\gamma}) \partial_x u_{\gamma} \sigma(1 - \gamma^{2/3} u_{\gamma}^2) \, dW_{\delta} \, dx
+ \int_0^t \int_{\mathbb{T}^1} \psi'_{\gamma,\bar{\delta}}(u_{\gamma}) \sigma(1 - \gamma^{2/3} u_{\gamma}^2)^2 F_{3,\delta} \, dx \, ds \\
& - 4 \gamma^{4/3} \int_0^t \int_{\mathbb{T}^1} \psi'_{\gamma,\bar{\delta}}(u_{\gamma}) |\partial_x u_{\gamma}|^2 \sigma'(1 - \gamma^{2/3} u_{\gamma}^2)^2 u_{\gamma}^2 F_{1,\delta} \, dx \, ds \\
& + 4 \gamma^{4/3} \int_0^t \int_{\mathbb{T}^1} \psi'_{\gamma,\bar{\delta}}(u_{\gamma}) |\partial_x u_{\gamma}|^2 \sigma'(1 - \gamma^{2/3} u_{\gamma}^2)^2 u_{\gamma}^2 F_{1,\delta} \, dx \, ds.
\end{align*}
Noting that the last two terms cancel each other out. Due to the properties of the function $\sigma(\cdot)$ and the specific form of $\psi'_{\gamma,\bar{\delta}}$,  the above identity reduces to
\begin{align*}
& \int_{\mathbb{T}^1} \Psi_{\gamma,\bar{\delta}}(u_{\gamma}(t)) \, dx
+ \gamma^{-2/3} \int_0^t \int_{\mathbb{T}^1} \psi'_{\gamma,\bar{\delta}}(u_{\gamma}) |\partial_x u_{\gamma}|^2 \, dx \, ds \\
\leq\, & \int_{\mathbb{T}^1} \Psi_{\gamma,\bar{\delta}}(u_{\gamma,0}) \, dx
+ \gamma^{-2/3} (1 + a \gamma^{2/3}) \int_0^t \int_{\mathbb{T}^1} \psi'_{\gamma,\bar{\delta}}(u_{\gamma}) \partial_x u_{\gamma} (1 - \gamma^{2/3} u_{\gamma}^2) J_{\gamma^{1/3}} * \partial_x u_{\gamma} \, dx \, ds \\
& + \sqrt{2} \int_0^t \int_{\mathbb{T}^1} \psi'_{\gamma,\bar{\delta}}(u_{\gamma}) \partial_x u_{\gamma} \sigma(1 - \gamma^{2/3} u_{\gamma}^2) \, dW_{\delta} \, dx
+ \gamma^{1/3} \int_0^t \int_{\mathbb{T}^1} F_{3,\delta} \, dx \, ds.
\end{align*}
In the sequel, we first address the martingale term, which vanishes upon taking expectation due to the properties of stochastic integrals. More precisely, we have
\begin{align*}
\sqrt{2} \, \mathbb{E} \int_0^t \int_{\mathbb{T}^1} \psi'_{\gamma,\bar{\delta}}(u_{\gamma}) \, \partial_x u_{\gamma} \, \sigma\bigl(1 - \gamma^{2/3} u_{\gamma}^2\bigr) \, dW_{\delta} \, dx = 0.
\end{align*}
Regarding the nonlocal interaction term, a direct computation shows that for all $\zeta \in (-\gamma^{-1/3}, \gamma^{-1/3})$, we have
\begin{align} \label{bound-psi-prime}
|\psi_{\gamma,\bar{\delta}}'(\zeta)|(1 - \gamma^{2/3} \zeta^2) \leq \gamma^{1/3}.
\end{align}

Now, assume that $a < 0$ and that $\gamma\in(0,1]$ is sufficiently small so that $1 + a \gamma^{2/3} > 0$. Applying Young's inequality and using the fact that the kernel $J_{\gamma^{1/3}}$ integrates to one, we deduce that 
\begin{align}\label{kernel-term}
&\gamma^{-2/3}(1 + a\gamma^{2/3})\,\mathbb{E} \int_0^t \int_{\mathbb{T}^1} \psi_{\gamma,\bar{\delta}}'(u_{\gamma})\,\partial_x u_{\gamma} (1 - \gamma^{2/3} u_{\gamma}^2)\, J_{\gamma^{1/3}} \ast \partial_x u_{\gamma} \,dx\,ds \notag\\
&\leq \frac{\gamma^{-1/3}}{1 + \bar{\delta}}(1 + a\gamma^{2/3}) \,\mathbb{E} \int_0^T \int_{\mathbb{T}^1} \partial_x u_{\gamma}\, J_{\gamma^{1/3}} \ast \partial_x u_{\gamma} \,dx\,ds \notag\\
&\leq \frac{\gamma^{-1/3}}{1 + \bar{\delta}}(1 + a\gamma^{2/3}) \,\mathbb{E} \int_0^T \int_{\mathbb{T}^1} \int_{\mathbb{T}^1} \left[\frac{1}{2} (\partial_x u_{\gamma}(x))^2 + \frac{1}{2} (\partial_x u_{\gamma}(y))^2 \right] J_{\gamma^{1/3}}(x - y)\,dx\,dy\,ds \notag\\
&= \frac{\gamma^{-1/3}}{1 + \bar{\delta}}(1 + a\gamma^{2/3})\, \mathbb{E} \int_0^T \|\partial_x u_{\gamma}\|_{L^2(\mathbb{T}^1)}^2\,ds \notag\\
&= \frac{\gamma^{-1/3}}{1 + \bar{\delta}}\, \mathbb{E} \int_0^T \|\partial_x u_{\gamma}\|_{L^2(\mathbb{T}^1)}^2\,ds + \frac{\gamma^{1/3}}{1 + \bar{\delta}}\, a\, \mathbb{E} \int_0^T \|\partial_x u_{\gamma}\|_{L^2(\mathbb{T}^1)}^2\,ds.
\end{align}

The first term on the right-hand side above can be absorbed by the dissipation term. Indeed, by chain rule, we can rewrite the dissipation term into:  
\begin{align*}
\gamma^{-2/3}\, \mathbb{E} \int_0^t \int_{\mathbb{T}^1} \psi_{\gamma,\bar{\delta}}'(u_{\gamma})\, |\partial_x u_{\gamma}|^2\,dx\,ds
&= \gamma^{-1/3} \,\mathbb{E} \int_0^t \int_{\mathbb{T}^1} \frac{|\partial_x u_{\gamma}|^2}{1 + \bar{\delta} - \gamma^{2/3} u_{\gamma}^2} \,dx\,ds \\
&= \frac{\gamma^{-1/3}}{1+\bar{\delta}} \,\mathbb{E} \int_0^t \int_{\mathbb{T}^1} \left[ \gamma^{-2/3} \left| \partial_x \sqrt{1 + \bar{\delta} - \gamma^{2/3} u_{\gamma}^2} \right|^2 + |\partial_x u_{\gamma}|^2 \right] dx\,ds \\
&= \frac{\gamma^{1/3}}{1 + \bar{\delta}} \,\mathbb{E} \int_0^t \int_{\mathbb{T}^1} \frac{u_{\gamma}^2 (\partial_x u_{\gamma})^2}{1 + \bar{\delta} - \gamma^{2/3} u_{\gamma}^2}\, dx\,ds \\
&\quad + \frac{\gamma^{-1/3}}{1 + \bar{\delta}}\, \mathbb{E} \int_0^t \|\partial_x u_{\gamma}\|_{L^2(\mathbb{T}^1)}^2\,ds.
\end{align*}
Thus, the last term on the right-hand side above cancels with the first term in \eqref{kernel-term}. Moreover, by the assumptions on the covariance structure of the noise, the definition of $\eta_{\delta} = \delta^{-1/3} \eta(\cdot / \delta^{1/3})$, and applying the change of variables formula, there exists a constant independent of both $\gamma$ and $\delta$ such that 
$$
\gamma^{1/3} \int_0^t \int_{\mathbb{T}^1} F_{3,\delta}\, dx\,ds \leq\gamma^{1/3}\sum_{k\geq0}\langle\nabla\eta_{\delta},e_k\rangle^2\leq CT \gamma^{1/3} \delta^{-1}.
$$

Combining the above bounds, we conclude that
\begin{align*}
&\sup_{t \in [0,T]} \mathbb{E} \left( \int_{\mathbb{T}^1} \Psi_{\gamma,\bar{\delta}}(u_{\gamma}(t))\,dx \right) + \frac{\gamma^{1/3}}{1 + \bar{\delta}} \,\mathbb{E} \int_0^T \int_{\mathbb{T}^1} \frac{u_{\gamma}^2 |\partial_x u_{\gamma}|^2}{1 + \bar{\delta} - \gamma^{2/3} u_{\gamma}^2}\, dx\,dt \\
&\quad - a\, \frac{\gamma^{1/3}}{1 + \bar{\delta}} \,\mathbb{E} \int_0^T \|\partial_x u_{\gamma}\|_{L^2(\mathbb{T}^1)}^2\, ds \lesssim \mathbb{E} \int_{\mathbb{T}^1} \Psi_{\gamma,\bar{\delta}}(u_{\gamma,0})\,dx + CT\gamma^{1/3} \delta^{-1}.
\end{align*}

Now, for any fixed $\gamma\in(0,1]$, since $\Psi_{\gamma,\bar{\delta}}$ is uniformly bounded with respect to $\bar{\delta} \in (0,1)$, we may apply the dominated convergence theorem and the monotone convergence theorem to pass to the limit $\bar{\delta} \to 0$. We obtain
\begin{align*}
&\sup_{t \in [0,T]} \mathbb{E} \left( \int_{\mathbb{T}^1} \Psi_{\gamma}(u_{\gamma}(t))\,dx \right) + \gamma^{1/3} \,\mathbb{E} \int_0^T \int_{\mathbb{T}^1} \frac{u_{\gamma}^2 |\partial_x u_{\gamma}|^2}{1 - \gamma^{2/3} u_{\gamma}^2}\, dx\,dt \\
&\quad - a \gamma^{1/3} \,\mathbb{E} \int_0^T \|\partial_x u_{\gamma}\|_{L^2(\mathbb{T}^1)}^2\, ds \lesssim \mathbb{E} \int_{\mathbb{T}^1} \Psi_{\gamma}(u_{\gamma,0})\,dx + CT\gamma^{1/3} \delta^{-1}.
\end{align*}

Furthermore, combining this estimate with the lower bound of the entropy function established in Lemma \ref{lem-tech-1}, we deduce 
\begin{align*}
\gamma^{1/3} \sup_{t \in [0,T]} \mathbb{E} \left( \int_{\mathbb{T}^1} \left( \frac{1}{2} |u_{\gamma}(t)|^2 - \gamma^{-2/3} \right)\, dx \right) &+ \gamma^{1/3} \,\mathbb{E} \int_0^T \int_{\mathbb{T}^1} \frac{u_{\gamma}^2 |\partial_x u_{\gamma}|^2}{1 - \gamma^{2/3} u_{\gamma}^2}\, dx\,dt \\
&- a \gamma^{1/3} \,\mathbb{E} \int_0^T \|\partial_x u_{\gamma}\|_{L^2(\mathbb{T}^1)}^2\, ds\\
& \lesssim \mathbb{E} \int_{\mathbb{T}^1} \Psi_{\gamma}(u_{\gamma,0})\,dx + CT\gamma^{1/3} \delta^{-1}.
\end{align*}
With the help of the assumptions on $u_{\gamma,0}$, we have
$$
\mathbb{E} \int_{\mathbb{T}^1} \left( \Psi_{\gamma}(u_{\gamma,0}) + \gamma^{-1/3} \right) dx \lesssim \gamma^{1/3}.
$$
After canceling the common factor $\gamma^{1/3}$, this yields, for $a < 0$, 
\begin{align} \label{entropy-gamma-es}
\sup_{t \in [0,T]} \mathbb{E} \left( \int_{\mathbb{T}^1} \frac{1}{2} |u_{\gamma}(t)|^2 \, dx \right) &+  \,\mathbb{E} \int_0^T \int_{\mathbb{T}^1} \frac{u_{\gamma}^2 |\partial_x u_{\gamma}|^2}{1 - \gamma^{2/3} u_{\gamma}^2}\, dx\,dt \notag\\
& - a  \,\mathbb{E} \int_0^T \|\partial_x u_{\gamma}\|_{L^2(\mathbb{T}^1)}^2\, ds \lesssim  1 + CT\delta^{-1}. 
\end{align}
This concludes the proof of \eqref{uniformestimate}.

\end{proof}

\subsection{Uniform time-regularity estimates}\label{subsec-3-2}
We rewrite \eqref{equgamma} in the form of a stochastic Cahn-Hilliard equation supplemented by error terms. Notably, the fourth-order differential term and the potential term characteristic of the stochastic Cahn-Hilliard equation fundamentally originate from the Kac interaction term in \eqref{equgamma}. This connection becomes clear upon applying a Taylor expansion to the Kac interaction term.

Motivated by the analysis presented in the introduction, we thus reformulate \eqref{equgamma} as follows: 
\begin{align}\label{rewriteform}
du_{\gamma}=&\gamma^{-2/3}\partial_{xx}^2 u_{\gamma}dt-\sqrt{2}\partial_{x}(\sigma(1-\gamma^{2/3}u_{\gamma}^2)dW_{\delta})+4\partial_{x}\Big(F_{1,\delta}\sigma'(1-\gamma^{2/3}u_{\gamma}^2)^2\gamma^{4/3}u_{\gamma}^2\partial_{x}u_{\gamma}\Big)dt\notag\\
&-\gamma^{-2/3}(1+a\gamma^{2/3})\partial_{x}\Big[J_{\gamma^{1/3}}\ast\partial_{x}u_{\gamma}\Big]dt+(1+a\gamma^{2/3})\partial_x\Big[u_{\gamma}^2J_{\gamma^{1/3}}\ast\partial_xu_{\gamma}\Big]dt\notag\\
=&\gamma^{-2/3}\partial_{xx}^2 u_{\gamma}dt-\sqrt{2}\partial_{x}(\sigma(1-\gamma^{2/3}u_{\gamma}^2)dW_{\delta})+4\partial_{x}\Big(F_{1,\delta}\sigma'(1-\gamma^{2/3}u_{\gamma}^2)^2\gamma^{4/3}u_{\gamma}^2\partial_{x}u_{\gamma}\Big)dt\notag\\
&-\gamma^{-2/3}(1+a\gamma^{2/3})\partial_{x}\Big[J_{\gamma^{1/3}}\ast\partial_{x}u_{\gamma}-\partial_xu_{\gamma}-\gamma^{2/3}\frac{D}{2}\partial_{xxx}^3u_{\gamma}\Big]dt\notag\\
&+(1+a\gamma^{2/3})\partial_x\Big[u_{\gamma}^2\Big(J_{\gamma^{1/3}}\ast\partial_xu_{\gamma}-\partial_xu_{\gamma}\Big)\Big]dt\notag\\
&-\gamma^{-2/3}(1+a\gamma^{2/3})\Big[\partial_{xx}^2u_{\gamma}+\gamma^{2/3}\frac{D}{2}\partial_{xxxx}^4u_{\gamma}\Big]dt+(1+a\gamma^{2/3})\partial_x(u_{\gamma}^2\partial_xu_{\gamma})dt\notag\\
=&-\frac{D}{2}\partial_{xxxx}^4u_{\gamma}dt+\partial_{xx}^2\Big[\frac{1}{3}u_{\gamma}^3-au_{\gamma}\Big]dt+\sqrt{2}\partial_{x}(\sigma(1-\gamma^{2/3}u_{\gamma}^2)dW_{\delta})\notag\\
&+(R_{1,\gamma}+R_{2,\gamma}+R_{3,\gamma}+R_{4,\gamma})dt,
\end{align}
where $D=\int_{\mathbb{T}^1}J(x)|x|^2dx$, and the remainder terms are  defined by 
\begin{align*}
&R_{1,\gamma}=-\frac{D}{2}a\gamma^{2/3}\partial_{xxxx}^4u_{\gamma}+\frac{1}{3}a\gamma^{2/3}\partial_{xx}^2u_{\gamma}^3,\\
&R_{2,\gamma}=(1+a\gamma^{2/3})\partial_x\Big[u_{\gamma}^2(J_{\gamma^{1/3}}\ast\partial_xu_{\gamma}-\partial_xu_{\gamma})\Big],\\
&R_{3,\gamma}=-\gamma^{-2/3}(1+a\gamma^{2/3})\partial_{x}\Big[J_{\gamma^{1/3}}\ast\partial_{x} u_{\gamma}-\partial_{x} u_{\gamma}-\gamma^{2/3}\frac{D}{2}\partial_{xxx}^3u_{\gamma}\Big],\\
&R_{4,\gamma}=4\partial_{x}\Big(F_{1,\delta}\sigma'(1-\gamma^{2/3}u_{\gamma}^2)^2\gamma^{4/3}u_{\gamma}^2\partial_{x}u_{\gamma}\Big).
\end{align*}

In the following, we will provide a $W^{\alpha,2}([0,T];H^{-\beta}(\mathbb{T}^1))$-uniform estimate to $\{u_{\gamma}\}_{\gamma\in(0,1]}$ for every $\alpha\in(0,1/2)$ and $\beta>\frac{13}{2}$. 
\begin{proposition}\label{second-uniform-es}
	Assume that the initial data, the coefficient, and the interaction kernel $J$ satisfy Assumptions (A1), \ref{Assump-sigma}, and \ref{Assump-J}, respectively. For every $\gamma\in(0,1]$, let $u_{\gamma}$ be the weak solution of \eqref{equgamma} with initial data $u_{\gamma,0}$. Suppose that $a<0$. For every $\alpha\in(0,1/2)$, $\beta>\frac{13}{2}$, there exists a constant $C=C(u_0,T,\delta)>0$ such that  
	\begin{equation}
		\mathbb{E}\|u_{\gamma}\|_{W^{\alpha,2}([0,T];H^{-\beta}(\mathbb{T}^1))}\leq C. 
	\end{equation}
\end{proposition}
\begin{proof}
	Taking the $W^{\alpha,2}([0,T];H^{-\beta}(\mathbb{T}^1))$-norm to the both sides of the equation in the rewrite form \eqref{rewriteform}, using the Sobolev's embedding $W^{1,1}([0,T];H^{-\beta}(\mathbb{T}^1))\subset W^{\alpha,2}([0,T];H^{-\beta}(\mathbb{T}^1))$, we get 
	\begin{align*}
	\mathbb{E}\|u_{\gamma}\|_{W^{\alpha,2}([0,T];H^{-\beta}(\mathbb{T}^1))}\leq&\|u_{\gamma,0}\|_{L^1(\mathbb{T}^1)}T+\mathbb{E}\Big\|\int^{\cdot}_0\frac{D}{2}\partial_{xxxx}^4u_{\gamma}ds\Big\|_{W^{1,1}([0,T];H^{-\beta}(\mathbb{T}^1))}\\
	&+\mathbb{E}\Big\|\int^{\cdot}_0\partial_{xx}^2[\frac{1}{3}u_{\gamma}^3-au_{\gamma}]ds\Big\|_{W^{1,1}([0,T];H^{-\beta}(\mathbb{T}^1))}\\
	&+\mathbb{E}\Big\|\int^{\cdot}_0\sqrt{2}\partial_x\Big(\sigma(1-\gamma^{2/3}u_{\gamma}^2)dW_{\delta}\Big)\Big\|_{W^{\alpha,2}([0,T];H^{-\beta}(\mathbb{T}^1))}\\
	&+\mathbb{E}\Big\|\int^{\cdot}_0R_{1,\gamma}ds\Big\|_{W^{1,1}([0,T];H^{-\beta}(\mathbb{T}^1))}+\mathbb{E}\Big\|\int^{\cdot}_0R_{2,\gamma}ds\Big\|_{W^{1,1}([0,T];H^{-\beta}(\mathbb{T}^1))}\\
	&+\mathbb{E}\Big\|\int^{\cdot}_0R_{3,\gamma}ds\Big\|_{W^{1,1}([0,T];H^{-\beta}(\mathbb{T}^1))}+\mathbb{E}\Big\|\int^{\cdot}_0R_{4,\gamma}ds\Big\|_{W^{1,1}([0,T];H^{-\beta}(\mathbb{T}^1))}\\
	=:&\|u_{\gamma,0}\|_{L^1(\mathbb{T}^1)}T+I_1+I_2+I_3+I_4+I_5+I_6+I_7.  
	\end{align*}
For the initial data, thanks to the Assumption (A1), we have that 
\begin{align*}
\|u_{\gamma,0}\|_{L^1(\mathbb{T}^1)}\leq C. 
\end{align*}
For the term $I_1$, we employ the Sobolev embedding $L^2(\mathbb{T}^1) \subset H^{-\beta + 3}(\mathbb{T}^1)$. In conjunction with the uniform entropy dissipation estimate \eqref{uniformestimate}, this yields the bound 
\begin{align*}
	I_1\leq\Big(\mathbb{E}\int^T_0\frac{D}{2}\|\partial_xu_{\gamma}\|_{L^2(\mathbb{T}^1)}^2ds\Big)^{1/2}\leq C, 
\end{align*}
for some $C>0$, independent of $\gamma\in(0,1]$ but depending on $T$ and $\delta$. For the term $I_2$, we utilize the Sobolev embedding $L^1(\mathbb{T}^1) \subset H^{-\beta + 1}(\mathbb{T}^1)$. Applying H\"older's inequality together with the uniform entropy dissipation estimate \eqref{uniformestimate}, we can see that there exists a constant $C>0$, independent of $\gamma\in(0,1]$ but depending on $\delta$, such that 
\begin{align*}
I_2\leq&\frac{1}{3}\mathbb{E}\int^T_0\int_{\mathbb{T}^1}|\partial_x(u_{\gamma}^3)|dxds+\left(\mathbb{E}\int^T_0\int_{\mathbb{T}^1}|\partial_xu_{\gamma}|^2dxds\right)^{1/2}\\
\leq&\mathbb{E}\int^T_0\int_{\mathbb{T}^1}|u_{\gamma}^2\partial_xu_{\gamma}|dxds+C\\
\leq&\Big(\mathbb{E}\int^T_0\int_{\mathbb{T}^1}u_{\gamma}^2dxds\Big)^{1/2}\Big(\mathbb{E}\int^T_0\int_{\mathbb{T}^1}|u_{\gamma}\partial_xu_{\gamma}|^2dxds\Big)^{1/2}+C\\
\leq&\Big(\mathbb{E}\int^T_0\int_{\mathbb{T}^1}u_{\gamma}^2dxds\Big)^{1/2}\Big(\mathbb{E}\int^T_0\int_{\mathbb{T}^1}\frac{u_{\gamma}^2|\partial_xu_{\gamma}|^2}{1-\gamma^{2/3}u_{\gamma}^2}dxds\Big)^{1/2}+C\\
\leq&C. 	
\end{align*}
For the stochastic noise term, by applying the argument presented in \cite[Lemma 2.1]{FG95}, we deduce that
\begin{align*}
I_3 \leq \left( \mathbb{E} \int_0^T \int_{\mathbb{T}^1} \sigma \big(1 - \gamma^{2/3} u_{\gamma} \big)^2 \, dx \, ds \right)^{1/2} \leq C,
\end{align*}
for some constant $C > 0$ independent of $\gamma\in(0,1]$. Similarly to the estimates for $I_1$ and $I_2$, the term $I_4$ can be controlled by applying H\"older's inequality together with the uniform entropy dissipation estimate \eqref{uniformestimate}. This leads to the conclusion that 
\begin{align}\label{I4}
I_4\leq C\gamma^{2/3}. 	
\end{align}
For the term $I_5$, employing the same approach as in the estimation of $I_2$, we obtain
\begin{align*}
I_5 \leq{}& \mathbb{E} \left\| \int_0^\cdot \partial_x \big( u_{\gamma}^2 J_{\gamma^{1/3}} \ast \partial_x u_{\gamma} \big) ds \right\|_{W^{1,1}([0,T]; H^{-\beta}(\mathbb{T}^1))} + \mathbb{E} \left\| \int_0^\cdot \partial_x \big( u_{\gamma}^2 \partial_x u_{\gamma} \big) ds \right\|_{W^{1,1}([0,T]; H^{-\beta}(\mathbb{T}^1))} \\
\leq{}& \mathbb{E} \left\| \int_0^\cdot \partial_x \big( u_{\gamma}^2 J_{\gamma^{1/3}} \ast \partial_x u_{\gamma} \big) ds \right\|_{W^{1,1}([0,T]; H^{-\beta}(\mathbb{T}^1))} + C.
\end{align*}

By applying the chain rule, we observe that
\begin{align*}
\partial_x \big( u_{\gamma}^2 J_{\gamma^{1/3}} \ast \partial_x u_{\gamma} \big) = \partial_x \left( \partial_x (u_{\gamma}^2 J_{\gamma^{1/3}} \ast u_{\gamma}) - \partial_x (u_{\gamma}^2) J_{\gamma^{1/3}} \ast u_{\gamma} \right).
\end{align*}

Taking this expression into account, and employing the Sobolev embedding $L^1(\mathbb{T}^1) \subset H^{-\beta + 2}(\mathbb{T}^1)$, together with H\"older's inequality and Young's convolution inequality, we can see that there exists a constant $C>0$, independent of $\gamma\in(0,1]$ but depending on $T$ and $\delta$, such that 
\begin{align*}
I_5\lesssim&\mathbb{E}\int^T_0\int_{\mathbb{T}^1}u_{\gamma}^2|J_{\gamma^{1/3}}\ast u_{\gamma}|dxds\\
&+\Big(\mathbb{E}\int^T_0\|\partial_xu_{\gamma}u_{\gamma}\|_{L^2(\mathbb{T}^1)}^2ds\Big)^{1/2}\Big(\mathbb{E}\int^T_0\|J_{\gamma^{1/3}}\|_{L^1(\mathbb{T}^1)}^2\|u_{\gamma}\|_{L^2(\mathbb{T}^1)}^2ds\Big)^{1/2}+C\\
\leq&\Big(\mathbb{E}\int^T_0\int_{\mathbb{T}^1}u_{\gamma}^4dxds\Big)^{1/2}\Big(\mathbb{E}\int^T_0\|J_{\gamma^{1/3}}\|_{L^1(\mathbb{T}^1)}^2\|u_{\gamma}\|_{L^2(\mathbb{T}^1)}^2ds\Big)^{1/2}+C\\
\leq& C\Big(\mathbb{E}\int^T_0\int_{\mathbb{T}^1}|u_{\gamma}^2|^2dxds\Big)^{1/2}+C. 
\end{align*}
By Poincar\'e's inequality and again invoking the uniform entropy dissipation estimate \eqref{uniformentropy}, we obtain that there exists a constant $C>0$, independent of $\gamma\in(0,1]$ but depending on $T$ and $\delta$, such that
\begin{align*}
I_5 \leq{}& C \left( \mathbb{E} \int_0^T \int_{\mathbb{T}^1} |u_{\gamma}^2|^2 \, dx \, ds \right)^{1/2} + C \\
\leq{}& C \left( \mathbb{E} \int_0^T \int_{\mathbb{T}^1} |\partial_x u_{\gamma}^2|^2 \, dx \, ds + \mathbb{E} \int_0^T \int_{\mathbb{T}^1} u_{\gamma}^2 \, dx \, ds \right)^{1/2} + C \\
\leq{}& C.
\end{align*}

In the following, we define
\begin{align*}
U_{\gamma} := -(-\partial_{xx}^2)^{-\frac{3}{2}} u_{\gamma}.
\end{align*}
For the term $I_6$, we apply the Sobolev embedding $L^1(\mathbb{T}^1) \subset H^{-\beta + 5}(\mathbb{T}^1)$, and, using a Taylor expansion of $U_{\gamma}$, one can see that there exists a constant $C>0$, independent of $\gamma\in(0,1]$ but depending on $T$ and $\delta$, such that
\begin{align}\label{I6}
I_6\leq&\mathbb{E}\Big\|\int^{\cdot}_0R_{3,\gamma}ds\Big\|_{W^{1,1}([0,T];H^{-\beta}(\mathbb{T}^1))}\notag\\
\leq&\mathbb{E}\int^T_0\Big\|\gamma^{-2/3}\partial_x\Big(J_{\gamma^{1/3}}\ast\partial_xu_{\gamma}-\partial_xu_{\gamma}-\gamma^{2/3}\frac{D}{2}\partial_{xxx}^3u_{\gamma}\Big)\Big\|_{H^{-\beta}(\mathbb{T}^1)}ds\notag\\
\leq&\mathbb{E}\int^T_0\Big\|\gamma^{-2/3}\partial_{xxxxx}^5\Big(J_{\gamma^{1/3}}\ast U_{\gamma}-U_{\gamma}-\gamma^{2/3}\frac{D}{2}\partial_{xx}^2U_{\gamma}\Big)\Big\|_{H^{-\beta}(\mathbb{T}^1)}ds\notag	\\
\leq&\mathbb{E}\int^T_0\int_{\mathbb{T}^1}\gamma^{-2/3}|J_{\gamma^{1/3}}\ast U_{\gamma}-U_{\gamma}-\gamma^{2/3}\frac{D}{2}\partial_{xx}^2U_{\gamma}|dxds\notag\\
\leq&\gamma^{-2/3+1}\Big(\frac{1}{6}\int_{\mathbb{T}^1}|y|^3J(y)dy\Big)\mathbb{E}\int^T_0\int_{\mathbb{T}^1}\|\partial_{xxx}^3U_{\gamma}\|_{L^{\infty}(\mathbb{T}^1)}dxds\notag\\
\leq&C\gamma^{1/3}\mathbb{E}\int^T_0\|u_{\gamma}\|_{L^{\infty}(\mathbb{T}^1)}ds\leq C\gamma^{1/3}\mathbb{E}\int^T_0\|u_{\gamma}\|_{H^1(\mathbb{T}^1)}ds\notag\\
\leq&C\gamma^{1/3}. 
\end{align}
Arguing similarly to the estimate for $I_2$ and applying H\"older's inequality, we exploit the boundedness of $\sigma'(\cdot)$ to obtain the existence of a constant $C > 0$ such that
\begin{align*}
I_7 ={}& \mathbb{E} \left\| \int_0^\cdot 4 \partial_x \left( F_{1,\delta} \sigma'\left(1 - \gamma^{2/3} u_\gamma^2 \right)^2 \gamma^{4/3} u_\gamma^2 \partial_x u_\gamma \right) ds \right\|_{W^{1,1}([0,T]; H^{-\beta}(\mathbb{T}^1))} \\
\leq{}& 4\gamma^{4/3} \|\sigma'(\cdot)\|_{L^\infty(\mathbb{R})}^2 \, \mathbb{E} \int_0^T \int_{\mathbb{T}^1} \left| u_\gamma^2 \partial_x u_\gamma \, F_{1,\delta} \right| \, dx \, ds \\
\leq{}& 4\gamma^{4/3} \|\sigma'(\cdot)\|_{L^\infty(\mathbb{R})}^2 \|F_{1,\delta}\|_{L^\infty(\mathbb{T}^1)} \, \mathbb{E} \int_0^T \int_{\mathbb{T}^1} \left| u_\gamma^2 \partial_x u_\gamma \right| \, dx \, ds \\
\leq{}& C \gamma^{4/3} \|\sigma'(\cdot)\|_{L^\infty(\mathbb{R})}^2 \|F_{1,\delta}\|_{L^\infty(\mathbb{T}^1)}.
\end{align*}

To estimate $\|F_{1,\delta}\|_{L^\infty(\mathbb{T}^1)}$, we recall the definition of $F_{1,\delta}$ and apply a change of variables:
\begin{align*}
\|F_{1,\delta}\|_{L^\infty(\mathbb{T}^1)} ={}& \left\| \sum_{k \geq 0} |e_k \ast \eta_\delta|^2 \right\|_{L^\infty(\mathbb{T}^1)} \\
={}& \operatorname*{ess\,sup}_{x \in \mathbb{T}^1} \| \eta_\delta(x - \cdot) \|_{L^2(\mathbb{T}^1)}^2 \\
={}& \int_{\mathbb{T}^1} |\delta^{-1/3} \eta(\delta^{-1/3} y)|^2 \, dy \leq C\delta^{-1/3}.
\end{align*}

Consequently, we conclude that
\begin{align}\label{I7}
I_7 \leq C \|\sigma'(\cdot)\|_{L^\infty(\mathbb{R})}^2 \gamma^{4/3} \delta^{-1/3}, 
\end{align}
for some $C>0$, independent of $\gamma\in(0,1]$ but depending on $T$ and $\delta$. This completes the proof.
\end{proof}

As a consequence, we obtain the following tightness result.

\begin{proposition}\label{prp-tight}
	Assume that the initial data, the coefficient, and the interaction kernel $J$ satisfy Assumptions (A1), \ref{Assump-sigma}, and \ref{Assump-J}, respectively. For each $\gamma\in(0,1]$, let $u_{\gamma}$ be the weak solution of \eqref{equgamma} with initial data $u_{\gamma,0}$. Suppose that $a < 0$. Then the family $(u_{\gamma}, \partial_x u_{\gamma}^2)_{\gamma\in(0,1]}$ is tight in
	$$
	L^2([0,T];L^2(\mathbb{T}^1)) \cap \left(L^2([0,T];H^1(\mathbb{T}^1)), w\right) \times \left(L^2([0,T];L^2(\mathbb{T}^1)), w\right).
	$$
\end{proposition}
\begin{proof}
This result is a direct consequence of the Aubin-Lions lemma, which states that for every $\alpha \in (0,1/2)$ and $\beta > \frac{13}{2}$, the embedding
\begin{align*}
L^2([0,T];H^1(\mathbb{T}^1)) \cap W^{\alpha,2}([0,T];H^{-\beta}(\mathbb{T}^1)) \subset L^2([0,T];L^2(\mathbb{T}^1)) \cap \left(L^2([0,T];H^1(\mathbb{T}^1)), w\right)
\end{align*}
is compact.	
\end{proof}

\section{Convergence of $\{u_{\gamma}\}_{\gamma\in(0,1]}$}\label{sec-4}
Before proceeding to the analysis of the convergence of the family $\{u_{\gamma}\}_{\gamma\in(0,1]}$, we first present the following lemma, which serves as an auxiliary tool in our argument. 
\begin{lemma}\label{lem-approx-CH}
Suppose that $a < 0$. For every $\delta > 0$, let $u_0 \in H^{-1}(\mathbb{T}^1)$, and denote by $u_{\delta}$ the weak solution to 
\begin{align}\label{CahnHilliard-delta}
	du_{\delta} = \partial_{xx}^2\left[V'(u_{\delta}) - \frac{D}{2} \partial_{xx}^2 u_{\delta} \right]dt - \sqrt{2} \, \partial_x dW_{\delta}, \quad u_{\delta}(0) = u_0.
\end{align}	
Let $u$ be the weak solution to
\begin{align*}
	du = \partial_{xx}^2\left[V'(u) - \frac{D}{2} \partial_{xx}^2 u \right]dt - \sqrt{2} \, \partial_x dW, \quad u(0) = u_0.
\end{align*}
with the same initial condition $u_0$ and the same noise $W$. Then, as $\delta \to 0$, the following convergence holds:
\begin{align*}
	\mathbb{E}\|u_{\delta} - u\|_{L^2([0,T]; L^2(\mathbb{T}^1))}^2 \rightarrow 0.
\end{align*}
\end{lemma}
\begin{proof}
	We defer the proof to Appendix \ref{sec-app-B}. 
\end{proof}

To deduce the convergence in probability of the family $\{u_{\gamma}\}_{\gamma\in(0,1]}$ from the established tightness, we invoke the following version of Krylov's diagonal argument. 
\begin{lemma}\label{lem-diagonal} 
Let $(E, d)$ be a Polish space. Let $(Z_n)_{n \geq 1}$ be a sequence of $E$-valued random variables. Then, $Z_n$ converges in probability to an $E$-valued random variable if and only if for every pair of subsequences $(Z_{l})$ and $(Z_{m})$, there exists a further subsequence $\big(v_k\big) = \big(Z_{l(k)}, Z_{m(k)}\big)$ which converges weakly to a random variable $v$ supported on the diagonal $\{(x,y) \in E \times E : x = y\}$.

Moreover, the same result holds when $E$ is replaced by a Hilbert space $H$ equipped with its weak topology $(H, w)$, provided that the sequence satisfies the uniform bound
\begin{equation}\label{Hilbert-cond}
\sup_{n \geq 1} \mathbb{E} \|Z_n\|_H^2 < \infty.
\end{equation}
\end{lemma}
\begin{proof}
The version of the result stated for Polish spaces has been established in \cite{GK21}. In the Hilbert space setting endowed with the weak topology, the same conclusion follows by combining the argument in \cite{GK21} with the additional condition \eqref{Hilbert-cond}. As an alternative, we also refer the reader to \cite[Theorem 1.1]{holden2022global} for a more general formulation of this result, where the convergence criterion is proven in the context of quasi-Polish spaces. 

\end{proof}

For the reader's convenience, we reformulate equation \eqref{equgamma} in the following form:
\begin{align}\label{eq-remainder}
du_{\gamma} =\ & -\frac{D}{2} \partial_{xxxx}^4 u_{\gamma} \, dt + \partial_{xx}^2 \left[ \frac{1}{3} u_{\gamma}^3 - a u_{\gamma} \right] \, dt - \sqrt{2} \, \partial_x \left( \sigma(1 - \gamma^{2/3} u_{\gamma}^2) \, dW_{\delta} \right) \notag \\
& + (R_{1,\gamma} + R_{2,\gamma} + R_{3,\gamma} + R_{4,\gamma}) \, dt,
\end{align}
where $ D = \int_{\mathbb{T}^1} J(x) |x|^2 \, dx $, and the remainder terms are given by
\begin{align}\label{remainder}
R_{1,\gamma} &= -\frac{D}{2} a \gamma^{2/3} \partial_{xxxx}^4 u_{\gamma} + \frac{1}{3} a \gamma^{2/3} \partial_{xx}^2 u_{\gamma}^3, \notag \\
R_{2,\gamma} &= (1 + a \gamma^{2/3}) \partial_x \left[ u_{\gamma}^2 \left( J_{\gamma^{1/3}} * \partial_x u_{\gamma} - \partial_x u_{\gamma} \right) \right], \notag \\
R_{3,\gamma} &= -\gamma^{-2/3} (1 + a \gamma^{2/3}) \, \partial_x \left[ J_{\gamma^{1/3}} * \partial_x u_{\gamma} - \partial_x u_{\gamma} - \gamma^{2/3} \frac{D}{2} \partial_{xxx}^3 u_{\gamma} \right], \notag \\
R_{4,\gamma} &= 4 \, \partial_x \left( F_{1,\delta} \, \sigma'(1 - \gamma^{2/3} u_{\gamma}^2)^2 \, \gamma^{4/3} u_{\gamma}^2 \, \partial_x u_{\gamma} \right).
\end{align}

To emphasize the dependence on the parameter $\delta > 0$, we henceforth denote the solution $u_{\gamma}$ of \eqref{equgamma} by $u_{\gamma,\delta}$, and the remainder terms $R_{i,\gamma}$, $i = 1,2,3,4$, by $R_{i,\gamma,\delta}$ accordingly. We now state the main result.

\begin{proposition}\label{prp-convergence}
Assume that the initial data, the coefficient, and the interaction kernel $J$ satisfy Assumptions (A1), \ref{Assump-sigma}, and \ref{Assump-J}, respectively. For every $ \gamma\in(0,1] $ and $ \delta > 0 $, let $ u_{\gamma,\delta} $ denote the weak solution to \eqref{equgamma} with initial data $ u_{\gamma,0} $. Suppose further that $ a < 0 $. Then, for every fixed $ \delta > 0 $, the following convergence holds in probability:
\begin{align*}
	\|u_{\gamma,\delta} - u_{\delta}\|_{L^2([0,T]; L^2(\mathbb{T}^1))} \rightarrow 0,
\end{align*}
as $ \gamma \rightarrow 0 $, where $ u_{\delta} $ denotes the weak solution to the Cahn-Hilliard equation \eqref{CahnHilliard-delta} with initial data $ u_0 $.
 \end{proposition}
\begin{proof}
Let $\delta > 0$ be fixed. By applying Proposition~\ref{prp-tight}, it follows that the family %Recall that the sequence $(u_{\gamma,\delta}, \partial_x u_{\gamma,\delta}^2, u_{\gamma,0})_{\gamma\in(0,1]}$ is tight in the product space 
%$$
%L^2([0,T]; L^2(\mathbb{T}^1)) \times \left( L^2([0,T]; L^2(\mathbb{T}^1)), w \right) \times H^{-1}(\mathbb{T}^1).
%$$
$$
X_{\gamma,\delta} := (u_{\gamma,\delta}, \partial_x u_{\gamma,\delta}^2, u_{\gamma,0})_{\gamma\in(0,1]}
$$
is tight in the space
\begin{align*}
	\mathbb{X} := \left( L^2([0,T]; L^2(\mathbb{T}^1)) \cap \left( L^2([0,T]; H^1(\mathbb{T}^1)), w \right) \right) \times \left( L^2([0,T]; L^2(\mathbb{T}^1)), w \right) \times H^{-1}(\mathbb{T}^1).
\end{align*}

Let $(B^k)_{k \geq 0}$ denote the Brownian motions appearing in \eqref{brownian}. To apply Lemma \ref{lem-diagonal}, let $\gamma_j$ and $\gamma_j'$ be two subsequences of $\gamma$. We now consider the pair
\begin{align*}
(X_{\gamma_j,\delta},X_{\gamma_j',\delta},(B^k)_{k\geq0})\ \text{on }\mathbb{Y}=\mathbb{X}\times\mathbb{X}\times C([0,T];\mathbb{R}^{\mathbb{N}}). 	
\end{align*}
Let $(u_{\delta}, v_{\delta}, u_0)$ be a limit point of the sequence $(u_{\gamma_j,\delta}, \partial_x u_{\gamma_j,\delta}^2, u_{\gamma_j,0})_{\gamma_j > 0}$, and let $(u_{\delta}', v_{\delta}', u_0)$ be a limit point of $(u_{\gamma_j',\delta}, \partial_x u_{\gamma_j',\delta}^2, u_{\gamma_j',0})_{\gamma_j' > 0}$. 

By the Skorokhod-Jakubowski representation theorem \cite{Jak97}, there exists a new probability space, denoted by $(\bar{\Omega}, \bar{\mathcal{F}}, \bar{\mathbb{P}})$, and a new sequence of random variables,
\begin{align*}
	\bar{X}_{\gamma_j,\delta} := (\bar{u}_{\gamma_j,\delta}, \bar{v}_{\gamma_j,\delta}, \bar{u}_{\gamma_j,0,\delta}), \quad 
	\bar{X}_{\gamma_j',\delta} := (\bar{u}_{\gamma_j',\delta}, \bar{v}_{\gamma_j',\delta}, \bar{u}_{\gamma_j',0,\delta}), \quad 
	(\bar{B}^k_{j,\delta})_{k \geq 0},
\end{align*}
and
\begin{align*}
	\bar{X}_{\delta} := (\bar{u}_\delta, \bar{v}_\delta, \bar{u}_0), \quad 
	\bar{X}_{\delta}' := (\bar{u}_{\delta}', \bar{v}_{\delta}', \bar{u}_0), \quad 
	(\bar{B}^k_{\delta})_{k \geq 0},
\end{align*}
such that all random variables on the new probability space have the same laws as their counterparts on the original space. More precisely, for every $j > 0$,
\begin{align*}
	\bar{X}_{\gamma_j,\delta} \overset{d}{=} X_{\gamma_j,\delta}, \quad 
	\bar{X}_{\gamma_j',\delta} \overset{d}{=} X_{\gamma_j',\delta}, \quad 
	(\bar{B}^k_{j,\delta})_{k \geq 0} \overset{d}{=} (B^k)_{k \geq 0},
\end{align*}
and
\begin{align*}
	\bar{X}_{\delta} \overset{d}{=} (u_\delta, v_\delta, u_0), \quad 
	\bar{X}_{\delta}' \overset{d}{=} (u_{\delta}', v_{\delta}', u_0), \quad 
	(\bar{B}^k_{\delta})_{k \geq 0} \overset{d}{=} (B^k)_{k \geq 0}.
\end{align*}

Furthermore, we have the following almost sure convergences as $j \to \infty$: 
\begin{itemize}
	\item $\bar{u}_{\gamma_j,\delta} \to \bar{u}_\delta$ strongly in $L^2([0,T]; L^2(\mathbb{T}^1))$ and weakly in $L^2([0,T];H^1(\mathbb{T}^1))$;
	\item $\bar{v}_{\gamma_j,\delta} \rightharpoonup \bar{v}_\delta$ weakly in $L^2([0,T]; L^2(\mathbb{T}^1))$;
	\item for every $k \geq 0$, $\sup_{t \in [0,T]} |\bar{B}^k_{j,\delta}(t) - \bar{B}^k_\delta(t)| \to 0$;
	\item $\bar{u}_{\gamma_j,0,\delta} \to \bar{u}_0$ in $H^{-1}(\mathbb{T}^1)$.
\end{itemize}
The same convergences also hold for the subsequence $(\gamma_j')$.

We next verify that, for every $\gamma_j > 0$, the identity
\begin{align}\label{identity-derivativeterm}
	\bar{v}_{\gamma_j,\delta} = \partial_x \bar{u}_{\gamma_j,\delta}^2, \quad \text{for almost every } (t,x) \in [0,T] \times \mathbb{T}^1,
\end{align}
holds almost surely. This follows from the fact that the joint laws of $(u_{\gamma_j,\delta}, \partial_x u_{\gamma_j,\delta}^2)$ and $(\bar{u}_{\gamma_j,\delta}, \bar{v}_{\gamma_j,\delta})$ are identical. Therefore, for every test function $\psi \in C^\infty([0,T] \times \mathbb{T}^1)$, we have
\begin{align*}
	\bar{\mathbb{E}}\Big|\langle \bar{v}_{\gamma_j,\delta}, \psi \rangle + \langle \bar{u}_{\gamma_j,\delta}^2, \partial_x \psi \rangle\Big| 
	= \mathbb{E}\Big|\langle \partial_x u_{\gamma_j,\delta}^2, \psi \rangle + \langle u_{\gamma_j,\delta}^2, \partial_x \psi \rangle\Big| = 0.
\end{align*}
This implies \eqref{identity-derivativeterm}.

Let $\bar{W}_{j,\delta}$ and $\bar{W}_\delta$ denote the cylindrical Wiener processes generated by $(\bar{B}^k_{j,\delta})_{k \geq 0}$ and $(\bar{B}^k_\delta)_{k \geq 0}$, respectively.
In the following, we show that $(\bar{B}^k_{\delta})_{k\geq1}$ is a sequence of Brownian motions. For every $t\in[0,T]$, let $\mathcal{G}_t:=\sigma\Big(\bar{X}_{\delta}|_{[0,t]},\bar{X}'_{\delta}|_{[0,t]},(\bar{B}^k)_{k\geq1}|_{[0,t]}\Big)$, and let $\bar{\mathcal{G}}_t$ be the augmented filtration of $\mathcal{G}_t$. Let $F:\mathbb{X}\times\mathbb{X}\times C([0,T])\rightarrow\mathbb{R}$ be a bounded continuous function. For every $0\leq s\leq t\leq T$ and every $j,k\geq1$, by the identical of distributions, we find that 
\begin{align*}
&\bar{\mathbb{E}}\Big(F(\bar{X}_{\gamma_j,\delta}|_{[0,s]},\bar{X}_{\gamma_j',\delta}|_{[0,s]},\bar{B}^k_{j,\delta}|_{[0,s]})(\bar{B}^k_{j,\delta}(t)-\bar{B}^k_{j,\delta}(s))\Big)\\
=&\mathbb{E}\Big(F(X_{\gamma_j,\delta}|_{[0,s]},X_{\gamma_j',\delta}|_{[0,s]},B^k|_{[0,s]})(B^k(t)-B^k(s))\Big)=0. 
\end{align*}
By the continuity of the functional $F$, we can pass to the limit as $j \to \infty$ in the corresponding identities. Consequently, for every $0 \leq s \leq t \leq T$ and every $k \geq 1$, we obtain
\begin{align*}
\bar{\mathbb{E}}\Big(F(\bar{X}_{\delta}|_{[0,s]}, \bar{X}'_{\delta}|_{[0,s]}, \bar{B}^k_{\delta}|_{[0,s]})(\bar{B}^k_{\delta}(t) - \bar{B}^k_{\delta}(s))\Big) = 0.
\end{align*}
Similarly, for every $k, l \geq 1$, we have
\begin{align*}
\bar{\mathbb{E}}\Big(F(\bar{X}_{\delta}|_{[0,s]}, \bar{X}'_{\delta}|_{[0,s]}, \bar{B}^k_{\delta}|_{[0,s]})
\big[\bar{B}^k_{\delta}(t) \bar{B}^l_{\delta}(t) - \bar{B}^k_{\delta}(s) \bar{B}^l_{\delta}(s) - \delta_{kl}(t - s)\big]\Big) = 0.
\end{align*}
These identities imply that the sequence $(\bar{B}^k_{\delta})_{k \geq 1}$ forms a family of standard Brownian motions with respect to the filtration $(\mathcal{G}_t)_{t \in [0,T]}$. Moreover, due to the continuity in time of each $\bar{B}^k_{\delta}$, it follows that $(\bar{B}^k_{\delta})_{k \geq 1}$ remains a family of Brownian motions with respect to the completed, right-continuous filtration $(\bar{\mathcal{G}}_t)_{t \in [0,T]}$.

In what follows, we aim to establish that $\bar{u}_\delta$ satisfies the Cahn-Hilliard equation driven by the correlated noise $\bar{W}_\delta \ast \eta_{\delta}$ in the weak sense:
\begin{align}\label{CahnHilliard-n}
	\partial_t\bar{u}_\delta = \partial_{xx}^2\Big[V'(\bar{u}_\delta) - \frac{D}{2}\partial_{xx}^2 \bar{u}_\delta\Big] - \sqrt{2} \partial_x\bar{dW}_\delta \ast \eta_{\delta}, \quad \bar{u}_\delta(0) = \bar{u}_{0}.
\end{align}
To this end, we invoke the definition of weak solutions for the approximating equation \eqref{equgamma}. It suffices to verify that, almost surely, for every test function $\varphi \in C^{\infty}(\mathbb{T}^1)$ and for every $t \in [0,T]$, 
\begin{align}\label{passagetothelimits}
	\langle (\bar{u}_{\gamma_j,\delta}-\bar{u}_\delta)(t),\varphi\rangle-&(\langle \bar{u}_{\gamma_j,0,\delta}(0),\varphi\rangle+\langle \bar{u}_{0},\varphi\rangle)-\frac{D}{2}\int^t_0\langle \bar{u}_{\gamma_j,\delta}-\bar{u}_\delta,\partial_{xxxx}^4\varphi\rangle ds\notag\\
	&+\int^t_0\langle\Big[\frac{1}{3}(\bar{u}_{\gamma_j,\delta}^3-\bar{u}_\delta^3)-a(\bar{u}_{\gamma_j,\delta}-\bar{u}_\delta)\Big],\partial_{xx}^2\varphi\rangle ds\notag\\
	&+\sqrt{2}\int^t_0\langle \partial_{x}\varphi,(\sigma(1-\gamma_j^{2/3}\bar{u}_{\gamma_j,\delta}^2)-1)d\bar{W}_{j,\delta}\ast\eta_{\delta}\rangle\notag\\
	&+\sqrt{2}\int^t_0\langle \partial_{x}\varphi,(d\bar{W}_{j,\delta}-d\bar{W}_{\delta})\ast\eta_{\delta}\rangle\notag\\
&+\int^t_0\langle R_{1,\gamma_j,\delta}+R_{2,\gamma_j,\delta}+R_{3,\gamma_j,\delta}+R_{4,\gamma_j,\delta},\varphi\rangle ds\notag\\
&\rightarrow0,
\end{align}
as $j\rightarrow\infty$. By the $L^2([0,T];L^2(\mathbb{T}^1))$-convergence of $\bar{u}_{\gamma_j,\delta}$ and the $H^{-1}(\mathbb{T}^1)$-convergence of the initial data $\bar{u}_{\gamma_j,0,\delta}$, we can readily conclude the convergence of all terms exhibiting a linear structure. In particular, we have
\begin{align*}
&\langle (\bar{u}_{\gamma_j,\delta}-\bar{u}_\delta)(t),\varphi\rangle-(\langle \bar{u}_{\gamma_j,0,\delta},\varphi\rangle-\langle \bar{u}_{0,\delta},\varphi\rangle)+\frac{D}{2}\int^t_0\langle \bar{u}_{\gamma_j,\delta}-\bar{u}_\delta,\partial_{xxxx}^4\varphi\rangle ds\\
\leq&(\|\bar{u}_{\gamma_j,\delta}(t)-\bar{u}_\delta(t)\|_{L^2(\mathbb{T}^1)}+\|\bar{u}_{\gamma_j,\delta}-\bar{u}_\delta\|_{L^2([0,T];L^2(\mathbb{T}^1))}+\|\bar{u}_{\gamma_j,0,\delta}-\bar{u}_{0,\delta}\|_{H^{-1}(\mathbb{T}^1)})\|\varphi\|_{H^4(\mathbb{T}^1)}\rightarrow0,
\end{align*} 
as $j \to \infty$. For the error terms, we proceed using the same techniques as those employed in the estimates \eqref{I4}, \eqref{I6}, and \eqref{I7} of Lemma \ref{second-uniform-es}. These yield the following bounds:
\begin{align*}
\int_0^t \langle R_{1,\gamma_j,\delta}, \varphi \rangle \, ds &\leq \|\varphi\|_{H^4(\mathbb{T}^1)} \gamma_j^{2/3},\\
\int_0^t \langle R_{3,\gamma_j,\delta}, \varphi \rangle \, ds &\leq \|\varphi\|_{H^4(\mathbb{T}^1)} \gamma_j^{1/3},\\
\int_0^t \langle R_{4,\gamma_j,\delta}, \varphi \rangle \, ds &\leq \|\varphi\|_{H^4(\mathbb{T}^1)} \gamma_j^{4/3} \delta^{-1/3}.
\end{align*}

To estimate $R_{2,\gamma_j,\delta}$, we decompose it into two parts using integration by parts:
\begin{align*}
\int_0^t \langle R_{2,\gamma_j,\delta}, \varphi \rangle \, ds 
&= -(1 + a \gamma_j^{2/3}) \int_0^t \langle \bar{u}_{\gamma_j,\delta}^2 (J_{\gamma_j^{1/3}} \ast \partial_x \bar{u}_{\gamma_j,\delta} - \partial_x \bar{u}_{\gamma_j,\delta}), \partial_x \varphi \rangle \, ds \\
&=  (1 + a \gamma_j^{2/3}) \int_0^t \langle \partial_x \bar{u}_{\gamma_j,\delta}^2 (J_{\gamma_j^{1/3}} \ast \bar{u}_{\gamma_j,\delta} - \bar{u}_{\gamma_j,\delta}), \partial_x \varphi \rangle \, ds \\
&\quad  (1 + a \gamma_j^{2/3}) \int_0^t \langle \bar{u}_{\gamma_j,\delta}^2 (J_{\gamma_j^{1/3}} \ast \bar{u}_{\gamma_j,\delta} - \bar{u}_{\gamma_j,\delta}), \partial_{xx}^2 \varphi \rangle \, ds \\
&=: I_1 + I_2.
\end{align*}

For $I_1$, by applying H\"older's inequality and Young's inequality for convolutions, we obtain
\begin{align*}
I_1 
&\lesssim \|\varphi\|_{H^4(\mathbb{T}^1)} \int_0^t \|\partial_x \bar{u}_{\gamma_j,\delta} \bar{u}_{\gamma_j,\delta}\|_{L^2(\mathbb{T}^1)} \|J_{\gamma_j^{1/3}} \ast \bar{u}_{\gamma_j,\delta} - \bar{u}_{\gamma_j,\delta}\|_{L^2(\mathbb{T}^1)} \, ds \\
&\leq \|\varphi\|_{H^4(\mathbb{T}^1)} \left( \int_0^t \|\partial_x \bar{u}_{\gamma_j,\delta} \bar{u}_{\gamma_j,\delta}\|_{L^2(\mathbb{T}^1)}^2 \, ds \right)^{1/2} \\
&\quad \cdot \left( \int_0^t \|J_{\gamma_j^{1/3}} \ast (\bar{u}_{\gamma_j,\delta} - \bar{u}_\delta)\|_{L^2(\mathbb{T}^1)}^2 + \|J_{\gamma_j^{1/3}} \ast \bar{u}_\delta - \bar{u}_\delta\|_{L^2(\mathbb{T}^1)}^2 + \|\bar{u}_\delta - \bar{u}_{\gamma_j,\delta}\|_{L^2(\mathbb{T}^1)}^2 \, ds \right)^{1/2} \\
&\leq \|\varphi\|_{H^4(\mathbb{T}^1)} \left( \int_0^t \|\partial_x \bar{u}_{\gamma_j,\delta} \bar{u}_{\gamma_j,\delta}\|_{L^2(\mathbb{T}^1)}^2 \, ds \right)^{1/2} \\
&\quad \cdot \left( \int_0^t \|J_{\gamma_j^{1/3}}\|_{L^1(\mathbb{T}^1)}^2 \|\bar{u}_{\gamma_j,\delta} - \bar{u}_\delta\|_{L^2(\mathbb{T}^1)}^2 + \|J_{\gamma_j^{1/3}} \ast \bar{u}_\delta - \bar{u}_\delta\|_{L^2(\mathbb{T}^1)}^2 + \|\bar{u}_\delta - \bar{u}_{\gamma_j,\delta}\|_{L^2(\mathbb{T}^1)}^2 \, ds \right)^{1/2} \rightarrow 0,
\end{align*}
as $j \to \infty$.

Similarly, for $I_2$, applying H\"older's inequality and Poincar\'e's inequality yields
\begin{align*}
I_2 
&\leq \|\varphi\|_{H^4(\mathbb{T}^1)} \int_0^t \|\bar{u}_{\gamma_j,\delta}^2\|_{L^2(\mathbb{T}^1)} \big( \|J_{\gamma_j^{1/3}} \ast (\bar{u}_{\gamma_j,\delta} - \bar{u}_\delta)\|_{L^2(\mathbb{T}^1)} \\
&\quad + \|J_{\gamma_j^{1/3}} \ast \bar{u}_\delta - \bar{u}_\delta\|_{L^2(\mathbb{T}^1)} + \|\bar{u}_\delta - \bar{u}_{\gamma_j,\delta}\|_{L^2(\mathbb{T}^1)} \big) \, ds \\
&\leq \|\varphi\|_{H^4(\mathbb{T}^1)} \left( \int_0^t \|\bar{u}_{\gamma_j,\delta}^2\|_{L^2(\mathbb{T}^1)}^2 \, ds \right)^{1/2} \\
&\quad \cdot \left( \int_0^t \|J_{\gamma_j^{1/3}} \ast (\bar{u}_{\gamma_j,\delta} - \bar{u}_\delta)\|_{L^2(\mathbb{T}^1)}^2 + \|J_{\gamma_j^{1/3}} \ast \bar{u}_\delta - \bar{u}_\delta\|_{L^2(\mathbb{T}^1)}^2 + \|\bar{u}_\delta - \bar{u}_{\gamma_j,\delta}\|_{L^2(\mathbb{T}^1)}^2 \, ds \right)^{1/2} \\
&\leq \|\varphi\|_{H^4(\mathbb{T}^1)} \left( \int_0^t \|\partial_x \bar{u}_{\gamma_j,\delta}^2\|_{L^2(\mathbb{T}^1)}^2 + \|\bar{u}_{\gamma_j,\delta}\|_{L^2(\mathbb{T}^1)}^2 \, ds \right)^{1/2} \\
&\quad \cdot \left( \int_0^t \|J_{\gamma_j^{1/3}} \ast (\bar{u}_{\gamma_j,\delta} - \bar{u}_\delta)\|_{L^2(\mathbb{T}^1)}^2 + \|J_{\gamma_j^{1/3}} \ast \bar{u}_\delta - \bar{u}_\delta\|_{L^2(\mathbb{T}^1)}^2 + \|\bar{u}_\delta - \bar{u}_{\gamma_j,\delta}\|_{L^2(\mathbb{T}^1)}^2 \, ds \right)^{1/2} \rightarrow 0,
\end{align*}
as $j \to \infty$. In the final step of the previous estimate, we observe that
$$
\limsup_{j\rightarrow\infty}\left(\int^t_0\|\partial_x\bar{u}_{\gamma_j,\delta}^2\|_{L^2(\mathbb{T}^1)}^2\,ds + \int^t_0\|\bar{u}_{\gamma_j,\delta}\|_{L^2(\mathbb{T}^1)}^2\,ds\right)^{1/2} < \infty,
$$
as a consequence of the weak convergence of $\partial_x\bar{u}_{\gamma_j,\delta}^2$ in $L^2([0,T];L^2(\mathbb{T}^1))$ and the strong convergence of $\bar{u}_{\gamma_j,\delta}$ in the same space. We now proceed to analyze the convergence of the cubic nonlinear term. By applying Poincar\'e's inequality once more, we obtain
\begin{align*}
\frac{1}{3}\int^t_0\langle\bar{u}_{\gamma_j,\delta}^3 - \bar{u}_\delta^3, \partial_{xx}^2\varphi\rangle\,ds
&= \frac{1}{3}\int^t_0\langle (\bar{u}_{\gamma_j,\delta} - \bar{u}_\delta)(\bar{u}_{\gamma_j,\delta}^2 + \bar{u}_{\gamma_j,\delta}\bar{u}_\delta + \bar{u}_\delta^2), \partial_{xx}^2\varphi \rangle\,ds \\
&\leq \|\varphi\|_{H^4(\mathbb{T}^1)} \left( \int^t_0\|\bar{u}_{\gamma_j,\delta} - \bar{u}_\delta\|_{L^2(\mathbb{T}^1)}^2\,ds \right)^{1/2} \\
& \times \left( \int^t_0 \|\bar{u}_{\gamma_j,\delta}^2 + \bar{u}_{\gamma_j,\delta}\bar{u}_\delta + \bar{u}_\delta^2\|_{L^2(\mathbb{T}^1)}^2\,ds \right)^{1/2} \\
&\lesssim \|\varphi\|_{H^4(\mathbb{T}^1)} \left( \int^t_0\|\bar{u}_{\gamma_j,\delta} - \bar{u}_\delta\|_{L^2(\mathbb{T}^1)}^2\,ds \right)^{1/2} \\
& \times \left( \int^t_0 \|\partial_x\bar{u}_{\gamma_j,\delta}^2\|_{L^2(\mathbb{T}^1)}^2 + \|\bar{u}_{\gamma_j,\delta}\|_{L^2(\mathbb{T}^1)}^2 + \|\partial_x\bar{u}_\delta^2\|_{L^2(\mathbb{T}^1)}^2 + \|\bar{u}_\delta\|_{L^2(\mathbb{T}^1)}^2\,ds \right)^{1/2} \to 0,
\end{align*}
as $j\rightarrow\infty$.

For the stochastic noise term, we apply It\^o's isometry and Young's convolution inequality. For any fixed $t\in[0,T]$, we estimate
\begin{align*}
\mathbb{E}\int^t_0 \langle \partial_x\varphi, \left(\sigma(1 - \gamma_j^{2/3} \bar{u}_{\gamma_j,\delta}^2) - 1\right) d\bar{W} \ast \eta_\delta \rangle
&\leq \mathbb{E}\int^T_0 \left\| \eta_\delta \ast \left( \partial_x\varphi \left( \sigma(1 - \gamma_j^{2/3} \bar{u}_{\gamma_j,\delta}^2) - 1 \right) \right) \right\|_{L^2(\mathbb{T}^1)}^2\,ds \\
&\leq \|\varphi\|_{H^4(\mathbb{T}^1)}^2 \mathbb{E}\int^T_0 \left\| \sigma(1 - \gamma_j^{2/3} \bar{u}_{\gamma_j,\delta}^2) - 1 \right\|_{L^2(\mathbb{T}^1)}^2\,ds.
\end{align*}
By the assumptions on $\sigma(\cdot)$ and the fact that $\gamma_j^{2/3} \bar{u}_{\gamma_j,\delta}^2 \leq 1$ almost surely, we have
$$
\left| \sigma(1 - \gamma_j^{2/3} \bar{u}_{\gamma_j,\delta}^2) - 1 \right| \leq \sqrt{2} + 1.
$$
Hence, with the help of the fact that $\sigma(1)=1$, applying the dominated convergence theorem yields
$$
\lim_{j\to\infty} \mathbb{E} \int^T_0 \left\| \sigma(1 - \gamma_j^{2/3} \bar{u}_{\gamma_j,\delta}^2) - 1 \right\|_{L^2(\mathbb{T}^1)}^2\,ds = \mathbb{E} \int^T_0 \int_{\mathbb{T}^1} \lim_{j\to\infty} \left| \sigma(1 - \gamma_j^{2/3} \bar{u}_{\gamma_j,\delta}^2) - 1 \right|^2\,dx\,ds = 0.
$$
This implies that, up to a subsequence (still denoted by $(\bar{u}_{\gamma_j,\delta})$), we have
$$
\int^t_0 \langle \partial_x\varphi, \left( \sigma(1 - \gamma_j^{2/3} \bar{u}_{\gamma_j,\delta}^2) - 1 \right) d\bar{W}_\delta \rangle \to 0,
$$
as $j \to \infty$, almost surely.

Finally, since for every $k \geq 0$,
$$
\sup_{t\in[0,T]} \left| \bar{B}_{j,\delta}^k(t) - \bar{B}_\delta^k(t) \right| \to 0 \quad \text{almost surely,}
$$
we conclude that
$$
\sqrt{2} \int^t_0 \langle \partial_x\varphi, (d\bar{W}_{j,\delta} - d\bar{W}_\delta) \ast \eta_\delta \rangle \to 0,
$$
as $j \to \infty$, almost surely. These results imply that $\bar{u}_\delta$ is a weak solution of \eqref{CahnHilliard-n} with initial data $\bar{u}_{0}$. Moreover, by the uniqueness of weak solutions to \eqref{CahnHilliard-n}, we conclude that $\bar{u}_\delta$ is the unique weak solution with initial data $\bar{u}_{0}$. An identical passage to the limit can be performed for the subsequence $(\gamma_k')$, yielding a limit $\bar{u}_\delta'$ which is likewise the unique weak solution of \eqref{CahnHilliard-n} with initial data $\bar{u}_{0}$.

By the uniqueness of solutions to \eqref{CahnHilliard-delta}, it follows that $\bar{u}_\delta = \bar{u}_\delta'$ almost surely for almost every $(x,t) \in \mathbb{T}^1 \times [0,T]$. Consequently, when returning to the original probability space, the joint laws of $(X_{\gamma_j}, X_{\gamma_j'})$ converge weakly to a probability measure on $\mathbb{X}^2$ supported on the diagonal set $\{(x,y) \in \mathbb{X}^2 : x = y\}$. Applying Lemma~\ref{lem-diagonal}, we deduce that the family $(X_\gamma)_{\gamma\in(0,1]}$ converges in probability. In particular, by extracting a subsequence, there exists a sequence $(\gamma_j)$ such that $u_{\gamma_j,\delta}$ converges to $u_\delta$ in $L^2([0,T];L^2(\mathbb{T}^1))$ almost surely. Moreover, following the above analysis, $u_\delta$ is the weak solution of \eqref{CahnHilliard-delta}. This concludes the proof.  

\end{proof} 

\begin{theorem}\label{thm-convergence}
Assume that the initial data, the coefficient, and the interaction kernel $J$ satisfy Assumptions (A1), \ref{Assump-sigma}, and \ref{Assump-J}, respectively. For every $\gamma\in(0,1]$ and $\delta > 0$, let $u_{\gamma,\delta}$ denote the weak solution of \eqref{equgamma} with initial condition $u_{\gamma,0}$. Suppose that $a < 0$. Then there exists a subsequence of the family $(u_{\gamma,\delta})_{\gamma,\delta > 0}$, denoted by $(u_n := u_{\gamma(n), \delta(n)})_{n \geq 1}$, such that $\lim_{n \to \infty} (\gamma(n), \delta(n)) = (0,0)$, and
\begin{align*}
\|u_n - u\|_{L^2([0,T]; L^2(\mathbb{T}^1))} \rightarrow 0,
\end{align*}
as $n \to \infty$ in probability, where $u$ is the weak solution of the Cahn-Hilliard equation \eqref{CahnHilliard} with initial data $u_0$.
\end{theorem}
\begin{proof}
	Let $u$ be the weak solution of \eqref{CahnHilliard} with initial data $u_0$. By Lemma~\ref{lem-approx-CH}, for any $\delta > 0$, let $u_\delta$ denote the weak solution of \eqref{CahnHilliard-delta} with initial data $u_0$. Then we have
\begin{align*}
\mathbb{E}\|u_\delta - u\|_{L^2([0,T]; L^2(\mathbb{T}^1))}^2 \rightarrow 0,
\end{align*}
as $\delta \to 0$. Consequently, for every $n \in \mathbb{N}_+$, there exists $\delta(n) > 0$ such that
\begin{align*}
\mathbb{E}\|u_{\delta(n)} - u\|_{L^2([0,T]; L^2(\mathbb{T}^1))}^2 < \frac{1}{8n^3}.
\end{align*}

Applying the triangle inequality, we obtain that for every $\gamma\in(0,1]$,
\begin{align*}
&\mathbb{P}\Big(\|u_{\gamma,\delta(n)} - u\|_{L^2([0,T]; L^2(\mathbb{T}^1))} > \frac{1}{n} \Big)\\
&\leq \mathbb{P}\Big(\|u_{\gamma,\delta(n)} - u_{\delta(n)}\|_{L^2([0,T]; L^2(\mathbb{T}^1))} + \|u_{\delta(n)} - u\|_{L^2([0,T]; L^2(\mathbb{T}^1))} > \frac{1}{n} \Big) \\
&\leq \mathbb{P}\Big(\|u_{\gamma,\delta(n)} - u_{\delta(n)}\|_{L^2([0,T]; L^2(\mathbb{T}^1))} > \frac{1}{2n} \Big) + \mathbb{P}\Big(\|u_{\delta(n)} - u\|_{L^2([0,T]; L^2(\mathbb{T}^1))} > \frac{1}{2n} \Big) \\
&\leq \mathbb{P}\Big(\|u_{\gamma,\delta(n)} - u_{\delta(n)}\|_{L^2([0,T]; L^2(\mathbb{T}^1))} > \frac{1}{2n} \Big) + 4n^2 \, \mathbb{E}\|u_{\delta(n)} - u\|_{L^2([0,T]; L^2(\mathbb{T}^1))}^2 \\
&\leq \mathbb{P}\Big(\|u_{\gamma,\delta(n)} - u_{\delta(n)}\|_{L^2([0,T]; L^2(\mathbb{T}^1))} > \frac{1}{2n} \Big) + \frac{1}{2n},
\end{align*}
where in the penultimate inequality we have used Chebyshev's inequality.

By Proposition \ref{prp-convergence}, there exists $\gamma = \gamma(\delta(n))$ such that
\begin{align*}
\mathbb{P}\Big(\|u_{\gamma,\delta(n)} - u_{\delta(n)}\|_{L^2([0,T]; L^2(\mathbb{T}^1))} > \frac{1}{2n} \Big) \leq \frac{1}{2n}.
\end{align*}
This implies that there exists a pair $(\gamma(n), \delta(n))$ such that
\begin{align*}
\mathbb{P}\Big(\|u_{\gamma(n), \delta(n)} - u\|_{L^2([0,T]; L^2(\mathbb{T}^1))} > \frac{1}{n} \Big) < \frac{1}{n}.
\end{align*}
This completes the proof.

\end{proof}

\begin{remark}
The above analysis establishes a two-step convergence of the nonlinear fluctuations. However, it remains unclear whether there exists an explicit scaling regime $(\gamma, \delta(\gamma))$ under which the same nonlinear fluctuation phenomenon can be observed. In particular, considering \eqref{entropy-gamma-es}, once all occurrences of the parameter $\gamma^{1/3}$ are eliminated, no remaining scaling parameter is available to control the divergence $\delta^{-2/3} \to \infty$. This observation implies that the uniform entropy dissipation estimate derived in Proposition~\ref{prp-entropydissipation} is valid only when $\delta$ is fixed, and does not extend to any joint scaling regime $(\gamma, \delta(\gamma)) \to (0,0)$. 
\end{remark}

\section{Multi-scale large deviations}\label{sec-5}
In this section, we consider a sequence of smooth approximations of the square-root diffusion coefficient and introduce an additional noise intensity parameter into equation \eqref{equgamma}. Our goal is to establish a multi-scale large deviation principle for this equation in the small noise regime. 
\begin{definition}
Let $E$ be a Polish space, and let $(X_{\varepsilon})_{\varepsilon>0}$ be a family of $E$-valued random variables defined on a fixed probability space $(\Omega,\mathcal{F},\mathbb{P})$. A function $I:E \to [0,+\infty]$ is called a \emph{rate function} if it is lower semi-continuous. We say that the family $(X_{\varepsilon})_{\varepsilon>0}$ satisfies a \emph{large deviation principle} (LDP) with speed $\varepsilon$ and rate function $I$ if for every closed set $F \subset E$ and every open set $U \subset E$, the following bounds hold:
\begin{align*}
	\limsup_{\varepsilon \to 0} \varepsilon \log \mathbb{P}(X_{\varepsilon} \in F) &\leq -\inf_{x \in F} I(x), \\
	\liminf_{\varepsilon \to 0} \varepsilon \log \mathbb{P}(X_{\varepsilon} \in U) &\geq -\inf_{x \in U} I(x).
\end{align*} 
\end{definition}

In the following, we introduce a sequence of smooth approximations of the square-root coefficient. 
\begin{lemma}\label{lem-approximation}
There exists a sequence of functions $\{\sigma_{n}\}_{n\in\mathbb{N}}$ such that $\sigma_{n}(\cdot) \to \sqrt{\cdot}$ in $\mathrm{C}_{\mathrm{loc}}^{1}((0,\infty))$ as $n \to \infty$. Moreover, each function $\sigma_n$ satisfies the following properties:
\begin{enumerate}
	\item[(1)] $\sigma_n \in C([0,\infty)) \cap C^{\infty}((0,\infty))$ with $\sigma_n(0) = 0$, and $\sigma_n' \in C_c^{\infty}([0,\infty))$ for every $n \in \mathbb{N}$;
	\item[(2)] there exists a constant $c \in (0,\infty)$ such that for every $\zeta \in [0,\infty)$,
	\begin{align}\label{kk-5.2}
		|\sigma_n(\zeta)| \leq c\sqrt{\zeta} \quad \text{uniformly in } n \in \mathbb{N};
	\end{align}
	\item[(3)] for every $\delta \in (0,1)$, there exists a constant $c_{\delta} \in (0,\infty)$ such that
	\begin{align}\label{eq-5.2}
		[\sigma_n'(\zeta)]^4 \mathbf{1}_{\{\zeta \geq \delta\}} + |\sigma_n(\zeta)\sigma_n'(\zeta)|^2 \mathbf{1}_{\{\zeta \geq \delta\}} \leq c_{\delta} \quad \text{uniformly in } n \in \mathbb{N}.
	\end{align}
\end{enumerate}
	
\end{lemma}

In this section, we investigate the small-noise large deviation behavior of the stochastic PDE
\begin{align}\label{eqgamma-1}
\partial_t u_{\varepsilon} =\ & \gamma^{-2/3} \partial_{xx}^2 u_{\varepsilon} - \gamma^{-2/3}(1 + a \gamma^{2/3}) \partial_x \left[ (1 - \gamma^{2/3} u_{\varepsilon}^2) J_{\gamma^{1/3}} \ast \partial_x u_{\varepsilon} \right] \notag \\
& - \sqrt{2} \varepsilon^{1/2} \partial_x \left( \sigma_n(1 - \gamma^{2/3} u_{\varepsilon}^2) \, dW_{\delta} \right) + 4\varepsilon \partial_x \left( F_{1,\delta} \, \sigma_n'(1 - \gamma^{2/3} u_{\varepsilon}^2)^2 \gamma^{4/3} u_{\varepsilon}^2 \partial_x u_{\varepsilon} \right),
\end{align}
with initial condition $u_{\varepsilon}(0) = u_{\gamma,0}$, in the asymptotic regime $(\varepsilon, \gamma(\varepsilon), \delta(\varepsilon), n(\varepsilon)) \to (0, 0, 0, +\infty)$.

Our analysis of the dynamical large deviation principle is based on the weak convergence framework developed in \cite{DE}, which relies on the equivalence between the Laplace principle and the large deviation principle under the assumption that the rate function is good, i.e., it possesses compact sublevel sets. We also refer to \cite{BD} and \cite{BDM11} for further details and foundational results.

To apply this approach, it is necessary to verify two key conditions. The first is the weak-strong stability of the associated skeleton equation. The second is the convergence in distribution of the controlled stochastic system (the so-called stochastic control equation) to the solution of the skeleton equation. These two aspects will be treated separately in the subsequent subsections. 

\subsection{The skeleton Cahn-Hilliard equation}
Let $g\in L^2([0,T];L^2(\mathbb{T}^1))$. Consider the following skeleton Cahn-Hilliard equation: 
\begin{equation}\label{skeleton-CahnHilliard}
\partial_tu=\partial_{xx}^2\Big[V'(u)-\frac{D}{2}\partial_{xx}^2 u\Big]-\sqrt{2}\partial_xg. 
\end{equation}

\begin{theorem}\label{stability-skeleton}
Let $g \in L^2([0,T]; L^2(\mathbb{T}^1))$, $a<0$ and $u_0 \in H^{-1}(\mathbb{T}^1)$. Then there exists a unique weak solution of \eqref{skeleton-CahnHilliard} with initial data $u_0$ and control $g$. 

Furthermore, let $(g_m)_{m \geq 1} \subset L^2([0,T]; L^2(\mathbb{T}^1))$ be a sequence such that
\begin{align*}
	g_m \rightharpoonup g \quad \text{weakly in } L^2([0,T]; L^2(\mathbb{T}^1)),
\end{align*}
as $m \to \infty$. Let $u_m$ denote the weak solution of \eqref{skeleton-CahnHilliard} with initial data $u_0$ and control $g_m$, and let $u$ be the solution corresponding to the control $g$. Then
\begin{align*}
	\|u_m - u\|_{L^2([0,T]; L^2(\mathbb{T}^1))} \to 0,
\end{align*}
as $m \to \infty$. 
\end{theorem}
\begin{proof}
\textbf{Existence.} We consider the following decomposition:
\begin{align*}
\partial_t z &= -\frac{D}{2}(\partial_{xx}^2)^2 z - \sqrt{2} \partial_x g,\quad z(0) = z_0, \\
\partial_t w &= \partial_{xx}^2 \big[(w + z)^3 - a(w + z) - \tfrac{D}{2} \partial_{xx}^2 w\big],\quad w(0) = w_0,
\end{align*}
where $z_0, w_0 \in H^{-1}(\mathbb{T}^1)$ with $z_0 + w_0 = u_0$. Let $(S(t))_{t \geq 0}$ denote the semigroup generated by the bi-Laplacian operator $-\frac{D}{2}(\partial_{xx}^2)^2$. Then the linear part $z$ can be represented via the mild formulation:
\begin{align*}
z(t) = S(t) z_0 - \sqrt{2} \int_0^t S(t - s) \partial_x g(s)\, ds.
\end{align*}
According to \cite{RYZ21}, there exists a weak solution $w$ to the nonlinear part of the system. Consequently, the function $u := w + z$ defines a weak solution of \eqref{skeleton-CahnHilliard}. 

\textbf{Uniqueness.} 
Let $u_1, u_2$ be two weak solutions of \eqref{skeleton-CahnHilliard} corresponding to the same control $g$, with initial data $u_{1,0}$ and $u_{2,0}$, respectively. By the definition of weak solutions to \eqref{skeleton-CahnHilliard}, the difference $u_1 - u_2$ satisfies, in the weak sense,
\begin{equation}\label{u1-u2}
\partial_t(u_1 - u_2) = -\partial_{xx}^2\big[u_1^3 - u_2^3 - a(u_1 - u_2) - \tfrac{D}{2} \partial_{xx}^2(u_1 - u_2)\big],
\end{equation}
with initial data $u_{1,0} - u_{2,0}$. 

Due to the regularity of $u_1$ and $u_2$, and since $C^\infty([0,T] \times \mathbb{T}^1)$ is dense in $L^2([0,T]; \dot{H}^3(\mathbb{T}^1)) \cap C([0,T]; H^{-l}(\mathbb{T}^1))$ for some $l > \tfrac{5}{2}$, we may use $-(-\partial_{xx}^2)^{-1}(u_1 - u_2)$ as a test function. Taking the $L^2(\mathbb{T}^1)$-inner product of \eqref{u1-u2} with this function, and applying the chain rule, we obtain:
\begin{align*}
&\frac{1}{2}\sup_{t \in [0,T]} \|u_1(t) - u_2(t)\|_{\dot{H}^{-1}(\mathbb{T}^1)}^2 
+ D \int_0^T \|\partial_x(u_1(s) - u_2(s))\|_{L^2(\mathbb{T}^1)}^2 \, ds \\
&= \frac{1}{2}\|u_{1,0} - u_{2,0}\|_{\dot{H}^{-1}(\mathbb{T}^1)}^2 
- \int_0^T \langle u_1 - u_2, (u_1^3 - a u_1) - (u_2^3 - a u_2) \rangle \, ds.
\end{align*}

Since the function $u \mapsto u^3 - a u$ is monotone increasing for every $a < 0$, the last term is non-positive:
\begin{align*}
- \int_0^T \langle u_1 - u_2, (u_1^3 - a u_1) - (u_2^3 - a u_2) \rangle \, ds \leq 0.
\end{align*}

In addition, since $u_1 - u_2$ satisfies \eqref{u1-u2} weakly, we may test it with the constant function $\varphi \equiv 1$, which yields
$$
\langle (u_1 - u_2)(t), 1 \rangle = \langle u_{1,0} - u_{2,0}, 1 \rangle \quad \text{for all } t \in [0,T].
$$

Now, applying Parseval's identity, we obtain:
\begin{align}\label{Parseval}
\|u_1 - u_2\|_{L^2([0,T]; L^2(\mathbb{T}^1))}^2 
&= \int_0^T \sum_{k \in \mathbb{Z}} \langle u_1 - u_2, e_k \rangle^2 \, dt \notag\\
&= \int_0^T \sum_{k \in \mathbb{Z} \setminus \{0\}} \langle u_1 - u_2, e_k \rangle^2 \, dt + \langle u_{1,0} - u_{2,0}, 1 \rangle^2 T \notag\\
&\leq \int_0^T \sum_{k \in \mathbb{Z} \setminus \{0\}} |k|^2 \langle u_1 - u_2, e_k \rangle^2 \, dt + \langle u_{1,0} - u_{2,0}, 1 \rangle^2 T \notag\\
&\leq \int_0^T \|\partial_x(u_1(t) - u_2(t))\|_{L^2(\mathbb{T}^1)}^2 \, dt + \langle u_{1,0} - u_{2,0}, 1 \rangle^2 T \notag\\
&\leq C(T) \|u_{1,0} - u_{2,0}\|_{\dot{H}^{-1}(\mathbb{T}^1)}^2+\langle u_{1,0}-u_{2,0},1\rangle^2T. 
\end{align}

This implies that $u_1 = u_2$ in $L^2([0,T]; L^2(\mathbb{T}^1))$ as long as $u_{1,0}=u_{2,0}$ in $H^{-1}(\mathbb{T}^1)$, and thus the weak solution to \eqref{skeleton-CahnHilliard} is unique. 

\textbf{Weak-strong stability.} 
Let $u_m$ be the solution to \eqref{skeleton-CahnHilliard} with initial data $u_0$ and control $g_m$. Applying the $H^{-1}(\mathbb{T}^1)$-energy estimate and using the chain rule, we obtain
\begin{align*}
\frac{1}{2} \sup_{t \in [0,T]} \|u_m(t)\|_{H^{-1}(\mathbb{T}^1)}^2 
+ D \int_0^T \|\partial_x u_m(s)\|_{L^2(\mathbb{T}^1)}^2 \, ds&+\int^T_0\|u_m(s)\|_{L^4(\mathbb{T}^1)}^4-a\|u_m(s)\|_{L^2(\mathbb{T}^1)}^2ds\\
\leq& \frac{1}{2}\|u_0\|_{H^{-1}(\mathbb{T}^1)}^2 
+ \sqrt{2} \int_0^T \langle (-\partial_{xx}^2)^{-1} u_m, \partial_x g_m \rangle \, ds.
\end{align*}

To estimate the right-hand side, we integrate by parts and apply Young's convolution inequality:
\begin{align*}
\left| \sqrt{2} \int_0^T \langle (-\partial_{xx}^2)^{-1} u_m, \partial_x g_m \rangle \, ds \right|
&= \left| \sqrt{2} \int_0^T \langle \partial_x (-\partial_{xx}^2)^{-1} u_m, g_m \rangle \, ds \right| \\
&\leq \sqrt{2} \int_0^T \|\partial_x (-\partial_{xx}^2)^{-1} u_m\|_{L^2(\mathbb{T}^1)} \|g_m\|_{L^2(\mathbb{T}^1)} \, ds \\
&\leq \sqrt{2} \int_0^T \left[ \frac{1}{2} \left\| \partial_{xx}^2 (-\partial_{xx}^2)^{-1} \int_{-1/2}^{\cdot} u_m(y) \, dy \right\|_{L^2(\mathbb{T}^1)}^2 
+ \frac{1}{2} \|g_m\|_{L^2(\mathbb{T}^1)}^2 \right] ds \\
&\leq C \int_0^T \|u_m(s)\|_{L^2(\mathbb{T}^1)}^2 \, ds 
+ C \sup_{m \geq 1} \|g_m\|_{L^2([0,T]; L^2(\mathbb{T}^1))}^2.
\end{align*}

Applying Gronwall's inequality then yields the uniform estimate
\begin{align*}
\frac{1}{2} \sup_{t \in [0,T]} \|u_m(t)\|_{H^{-1}(\mathbb{T}^1)}^2 
+ D \int_0^T \|\partial_x u_m(s)\|_{L^2(\mathbb{T}^1)}^2 \, ds 
\leq C(T) \left( \|u_0\|_{H^{-1}(\mathbb{T}^1)}^2 
+ \sup_{m \geq 1} \|g_m\|_{L^2([0,T]; L^2(\mathbb{T}^1))}^2 \right).
\end{align*}

As in \eqref{Parseval}, we further deduce that
\begin{align*}
\int_0^T \|u_m(s)\|_{L^2(\mathbb{T}^1)}^2 \, ds 
+ \int_0^T \|\partial_x u_m(s)\|_{L^2(\mathbb{T}^1)}^2 \, ds 
\leq C(T) \left( \|u_0\|_{H^{-1}(\mathbb{T}^1)}^2 
+ \sup_{m \geq 1} \|g_m\|_{L^2([0,T]; L^2(\mathbb{T}^1))}^2 \right).
\end{align*}

Moreover, for every $l > \tfrac{9}{2}$, a direct computation gives
$$
\|\partial_t u_m\|_{L^1([0,T]; H^{-l}(\mathbb{T}^1))} 
\leq C(T) \left( \|u_0\|_{H^{-1}(\mathbb{T}^1)}^2 
+ \sup_{m \geq 1} \|g_m\|_{L^2([0,T]; L^2(\mathbb{T}^1))}^2 \right).
$$

Thus, by the Aubin-Lions lemma, the sequence $(u_m)_{m \geq 1}$ is relatively compact in $L^2([0,T]; L^2(\mathbb{T}^1))$. 

Let $g$ be the weak limit of $(g_m)_{m \geq 1}$ in $L^2([0,T]; L^2(\mathbb{T}^1))$. A standard passage-to-the-limit argument shows that any limit point $u$ of $(u_m)_{m \geq 1}$ is a weak solution to \eqref{skeleton-CahnHilliard} with control $g$. This concludes the proof.

\end{proof}

\subsection{Large deviations from the deterministic Cahn-Hilliard limit} 
Let
$$
\mathcal{A}_T=\left\{g\ \text{is progressively measurable}:\ g\in L^2\big(\Omega; L^2([0,T]; L^2(\mathbb{T}^1))\big)\right\}.
$$
For a given sequence $(g_{\varepsilon})_{\varepsilon>0} \subset \mathcal{A}_T$, we consider the stochastic control equation
\begin{align}\label{eqgamma-2}
du_{\varepsilon} 
=&\, \gamma^{-2/3} \partial_{xx}^2 u_{\varepsilon} \, dt 
- \gamma^{-2/3}(1 + a\gamma^{2/3}) \partial_x \big[(1 - \gamma^{2/3} u_{\varepsilon}^2) J_{\gamma^{1/3}} \ast \partial_x u_{\varepsilon}\big] \, dt \notag \\
&- \sqrt{2} \varepsilon^{1/2} \partial_x \big(\sigma_n (1 - \gamma^{2/3} u_{\varepsilon}^2) dW_{\delta}\big) 
+ 4 \varepsilon \partial_x \Big( F_{1,\delta} \sigma_n'(1 - \gamma^{2/3} u_{\varepsilon}^2)^2 \gamma^{4/3} u_{\varepsilon}^2 \partial_x u_{\varepsilon} \Big) dt \notag \\
&- \sqrt{2} \partial_x \big(\sigma_n (1 - \gamma^{2/3} u_{\gamma}^2) g_{\varepsilon}\big) \, dt, \quad u_{\varepsilon}(0) = u_{\gamma,0}.
\end{align}

We remark that, for notational simplicity, the dependence of the solution on the parameters $\gamma$, $\delta$, and $n$ is suppressed throughout, and the weak solution to equation \eqref{eqgamma-2} will be denoted simply by $u_{\varepsilon}$. In the subsequent analysis, we shall consider a scaling regime 
$$
(\varepsilon, \gamma(\varepsilon), \delta(\varepsilon), n(\varepsilon)) \to (0, 0, 0, \infty)
$$
under which a large deviation principle will be established. The concept of weak solutions for equation \eqref{eqgamma-2} can be formulated analogously to Definition \ref{def-1}. By leveraging the analytical framework developed in \cite{WWZ22, WZ22}, we obtain the following well-posedness result for equation \eqref{eqgamma-2}.

\begin{theorem}
For every fixed parameter $\gamma\in(0,1]$, $\varepsilon,\delta>0$ and $n\geq1$, let the initial data $u_{\gamma,0}$ belong to the space $\overline{{\rm Ent}}(\mathbb{T}^1)$. Then there exists a unique weak solution $u_{\varepsilon}$ to the stochastic control equation \eqref{eqgamma-2} corresponding to the initial condition $u_{\varepsilon,0}$. 
\end{theorem}

In the following, we show a uniform entropy dissipation estimate for the stochastic control equation. 
\begin{proposition}\label{entropydissipation-stochasticcontrol}
Assume that the initial data and the interaction kernel $J$ satisfy Assumptions (A2) and \ref{Assump-J}, respectively. For each $\gamma\in(0,1]$, $\varepsilon, \delta > 0$ and $n\geq1$, let the control $g_{\varepsilon} \subset \mathcal{A}_T$ satisfy the uniform bound
$$
\sup_{\varepsilon > 0} \|g_{\varepsilon}\|_{L^2(\Omega; L^2([0,T]; L^2(\mathbb{T}^1)))}^2 < \infty.
$$
Let $u_{\varepsilon}$ denote the weak solution to equation \eqref{eqgamma-2} with initial data $u_{\gamma,0}$ and control $g_{\varepsilon}$. Suppose that $a < 0$. Then we have the uniform energy estimate
\begin{align}\label{uniformestimate-gep}
\sup_{t \in [0,T]} \mathbb{E} \left( \int_{\mathbb{T}^1} \frac{1}{2} |u_{\varepsilon}(t)|^2 \, dx \right) 
&+ \mathbb{E} \int_0^T \int_{\mathbb{T}^1} \frac{u_{\varepsilon}^2 |\partial_x u_{\varepsilon}|^2}{1 - \gamma^{2/3} u_{\varepsilon}^2} \, dx \, dt 
- a \mathbb{E} \int_0^T \|\partial_x u_{\varepsilon}\|_{L^2(\mathbb{T}^1)}^2 \, ds \notag\\
&\lesssim 1 + \varepsilon \delta^{-2/3} + \mathbb{E} \|g_{\varepsilon}\|_{L^2([0,T]; L^2(\mathbb{T}^1))}^2, 
\end{align}
for every $\gamma\in(0,1]$, $\varepsilon, \delta > 0$ and $n\geq1$.

\end{proposition}
\begin{proof}
	Similar to the proof of Proposition \ref{prp-entropydissipation}, we first regularized the entropy function. For every $\bar{\delta}>0$, define
\begin{align*}
\Psi_{\gamma,\bar{\delta}}(\zeta) &= \frac{\gamma^{-1/3}}{2(1+\bar{\delta})} \left[ (1+\bar{\delta}+\gamma^{1/3}\zeta)\log(1+\bar{\delta}+\gamma^{1/3}\zeta) + (1+\bar{\delta}-\gamma^{1/3}\zeta)\log(1+\bar{\delta}-\gamma^{1/3}\zeta) - 2 \right], \\
\psi_{\gamma,\bar{\delta}}(\zeta) &= \Psi'_{\gamma,\bar{\delta}}(\zeta) = \frac{1}{2(1+\bar{\delta})} \log\left( \frac{1+\bar{\delta}+\gamma^{1/3}\zeta}{1+\bar{\delta}-\gamma^{1/3}\zeta} \right), \quad \text{for all } \zeta \in [-\gamma^{-1/3}, \gamma^{1/3}].
\end{align*}
We also recall that 
$$
\psi'_{\gamma,\bar{\delta}}(\zeta) = \frac{\gamma^{1/3}}{(1+\bar{\delta})^2 - \gamma^{2/3}\zeta^2}, \quad \text{for all } \zeta \in [-\gamma^{-1/3}, \gamma^{1/3}].
$$

Applying It\^o's formula to $\Psi_{\gamma,\bar{\delta}}(u_{\varepsilon})$ (\cite[Theorem 3.1]{Krylov}), we obtain
\begin{align*}
& \int_{\mathbb{T}^1} \Psi_{\gamma,\bar{\delta}}(u_{\varepsilon}(t))\,dx
+ \gamma^{-2/3} \int_0^t \int_{\mathbb{T}^1} \psi'_{\gamma,\bar{\delta}}(u_{\varepsilon}) |\partial_x u_{\varepsilon}|^2 \, dx \, ds \\
=& \int_{\mathbb{T}^1} \Psi_{\gamma}(u_{\gamma,0})\,dx
+ \gamma^{-2/3}(1 + a\gamma^{2/3}) \int_0^t \int_{\mathbb{T}^1} \psi'_{\gamma,\bar{\delta}}(u_{\varepsilon}) \partial_x u_{\varepsilon}(1 - \gamma^{2/3} u_{\varepsilon}^2) J_{\gamma^{1/3}} \ast \partial_x u_{\varepsilon} \, dx \, ds \\
& + \sqrt{2} \varepsilon^{1/2} \int_0^t \int_{\mathbb{T}^1} \psi'_{\gamma,\bar{\delta}}(u_{\varepsilon}) \partial_x u_{\varepsilon} \sigma_n(1 - \gamma^{2/3} u_{\varepsilon}^2) \, dW_{\delta} \, dx \\
& + \varepsilon \int_0^t \int_{\mathbb{T}^1} \psi'_{\gamma,\bar{\delta}}(u_{\varepsilon}) \sigma_n(1 - \gamma^{2/3} u_{\varepsilon}^2)^2 F_{3,\delta} \, dx \, ds \\
& - 4 \varepsilon \gamma^{4/3} \int_0^t \int_{\mathbb{T}^1} \psi'_{\gamma,\bar{\delta}}(u_{\varepsilon}) |\partial_x u_{\varepsilon}|^2 \sigma_n'(1 - \gamma^{2/3} u_{\varepsilon}^2)^2 u_{\varepsilon}^2 F_{1,\delta} \, dx \, ds \\
& + 4 \varepsilon \gamma^{4/3} \int_0^t \int_{\mathbb{T}^1} \psi'_{\gamma,\bar{\delta}}(u_{\varepsilon}) |\partial_x u_{\varepsilon}|^2 \sigma_n'(1 - \gamma^{2/3} u_{\varepsilon}^2)^2 u_{\varepsilon}^2 F_{1,\delta} \, dx \, ds \\
& + \sqrt{2} \int_0^t \int_{\mathbb{T}^1} \psi'_{\gamma,\bar{\delta}}(u_{\varepsilon}) \partial_x u_{\varepsilon} \sigma_n(1 - \gamma^{2/3} u_{\varepsilon}^2) g_{\varepsilon} \, dx \, ds. 
\end{align*}
In comparison with the proof of Proposition \ref{prp-entropydissipation}, the main additional difficulty lies in the additional stochastic control term. To this end, we apply Young's inequality to the stochastic control term, which yields the estimate
\begin{align*}
&\int_{\mathbb{T}^1} \Psi_{\gamma,\bar{\delta}}(u_{\varepsilon}(t))\,dx
+ \gamma^{-2/3} \int_0^t \int_{\mathbb{T}^1} \psi'_{\gamma,\bar{\delta}}(u_{\varepsilon}) |\partial_x u_{\varepsilon}|^2 \, dx \, ds \\
\leq& \int_{\mathbb{T}^1} \Psi_{\gamma}(u_{\gamma,0})\,dx
+ \gamma^{-2/3}(1 + a\gamma^{2/3}) \int_0^t \int_{\mathbb{T}^1} \psi'_{\gamma,\bar{\delta}}(u_{\varepsilon}) \partial_x u_{\varepsilon}(1 - \gamma^{2/3} u_{\varepsilon}^2) J_{\gamma^{1/3}} \ast \partial_x u_{\varepsilon} \, dx \, ds \\
& + \sqrt{2} \varepsilon^{1/2} \int_0^t \int_{\mathbb{T}^1} \psi'_{\gamma,\bar{\delta}}(u_{\varepsilon}) \partial_x u_{\varepsilon} \sigma_n(1 - \gamma^{2/3} u_{\varepsilon}^2) \, dW_{\delta} \, dx \\
& + \varepsilon \gamma^{1/3} \int_0^t \int_{\mathbb{T}^1} F_{3,\delta} \, dx \, ds 
+ \frac{1}{4} \int_0^t \int_{\mathbb{T}^1} \psi'_{\gamma,\bar{\delta}}(u_{\varepsilon}) |\partial_x u_{\varepsilon}|^2 \, dx \, ds \\
& + C \int_0^t \int_{\mathbb{T}^1} \psi'_{\gamma,\bar{\delta}}(u_{\varepsilon}) \sigma_n(1 - \gamma^{2/3} u_{\varepsilon}^2)^2 |g_{\varepsilon}|^2 \, dx \, ds.
\end{align*}

By invoking the same reasoning as in the proof of Proposition \ref{prp-entropydissipation}, and exploiting the fact that $\psi'_{\gamma,\bar{\delta}}(u_{\varepsilon})\sigma_n(1 - \gamma^{2/3} u_{\varepsilon}^2)^2$ remains positive and uniformly bounded by a constant of order $\gamma^{1/3}$, we derive the desired entropy dissipation estimate. 
\begin{align*}
\gamma^{1/3} \sup_{t \in [0,T]} \mathbb{E} \left( \int_{\mathbb{T}^1} \left( \frac{1}{2} |u_{\varepsilon}(t)|^2 - \gamma^{-1/3} \right) dx \right)
&+ \gamma^{1/3} \mathbb{E} \int_0^T \int_{\mathbb{T}^1} \frac{u_{\varepsilon}^2 |\partial_x u_{\varepsilon}|^2}{1 - \gamma^{2/3} u_{\varepsilon}^2} \, dx \, dt \\
& - a \gamma^{1/3} \mathbb{E} \int_0^T \|\partial_x u_{\varepsilon}\|_{L^2(\mathbb{T}^1)}^2 \, ds \\
&\lesssim \int_{\mathbb{T}^1} \Psi_{\gamma}(u_{\gamma,0}) \, dx 
+ \gamma^{1/3} \varepsilon \delta^{-2/3} 
+ \gamma^{1/3} \mathbb{E} \|g_{\varepsilon}\|_{L^2([0,T]; L^2(\mathbb{T}^1))}^2.
\end{align*}
This confirms that the presence of the stochastic control does not alter the scaling structure of the bound and allows us to maintain uniform control over the solution $u_\varepsilon$ in the relevant energy norms. In particular, when $a < 0$, it follows that
\begin{align*}
\gamma^{1/3} \sup_{t \in [0,T]} \mathbb{E} \left( \int_{\mathbb{T}^1} \frac{1}{2} |u_{\gamma}(t)|^2 \, dx \right)
&+ \gamma^{1/3} \mathbb{E} \int_0^T \int_{\mathbb{T}^1} \frac{u_{\gamma}^2 |\partial_x u_{\gamma}|^2}{1 - \gamma^{2/3} u_{\gamma}^2} \, dx \, dt \\
& - a \gamma^{1/3} \mathbb{E} \int_0^T \|\partial_x u_{\gamma}\|_{L^2(\mathbb{T}^1)}^2 \, ds \\
&\lesssim \gamma^{1/3} \left( 1 + \varepsilon \delta^{-2/3} + \mathbb{E} \|g_{\varepsilon}\|_{L^2([0,T]; L^2(\mathbb{T}^1))}^2 \right).
\end{align*}

This completes the proof.

\end{proof}

\begin{proposition}\label{second-uniform-es-stochasticcontrol}
	Assume that the initial data and the interaction kernel $J$ satisfy Assumptions (A2) and \ref{Assump-J}, respectively. For every $\gamma\in(0,1]$, $\varepsilon,\delta>0$ and $n\geq1$, let $g_{\varepsilon}\subset \mathcal{A}_T$ satisfy the uniform bound
$$
\sup_{\varepsilon>0} \|g_{\varepsilon}\|_{L^2(\Omega;L^2([0,T];L^2(\mathbb{T}^1)))}^2 < \infty.
$$
Let $u_{\varepsilon}$ denote the weak solution to \eqref{eqgamma-2} with initial data $u_{\gamma,0}$ and stochastic control $g_{\varepsilon}$. Suppose moreover that $a < 0$. Then, for every $\alpha \in (0,1/2)$ and $\beta > \frac{13}{2}$, there exists a constant $C = C(u_0,T) > 0$ such that the following time-regularity estimate holds:
\begin{equation}
	\mathbb{E} \|u_{\varepsilon}\|_{W^{\alpha,2}([0,T];H^{-\beta}(\mathbb{T}^1))} \leq C(u_0,T) + C\varepsilon\gamma^{4/3}\delta^{-1/3} \|\sigma_{n}'(\cdot)\|_{L^{\infty}(\mathbb{R})}^2.
\end{equation}

\end{proposition}
\begin{proof}
	Taking the $W^{\alpha,2}([0,T];H^{-\beta}(\mathbb{T}^1))$-norm on both sides of the equation in the reformulated version \eqref{rewriteform}, and using the Sobolev embedding 
$$
W^{1,1}([0,T];H^{-\beta}(\mathbb{T}^1)) \subset W^{\alpha,2}([0,T];H^{-\beta}(\mathbb{T}^1)),
$$
we obtain the estimate
\begin{align*}
\mathbb{E}\|u_{\varepsilon}\|_{W^{\alpha,2}([0,T];H^{-\beta}(\mathbb{T}^1))}
\leq\,& \|u_{\gamma,0}\|_{L^1(\mathbb{T}^1)} T 
+ \mathbb{E} \Big\|\int_0^{\cdot} \tfrac{D}{2} \partial_x^4 u_{\varepsilon} \, ds \Big\|_{W^{1,1}([0,T];H^{-\beta}(\mathbb{T}^1))} \\
& + \mathbb{E} \Big\|\int_0^{\cdot} \partial_{xx}^2 \left[\tfrac{1}{3} u_{\varepsilon}^3 - a u_{\varepsilon}\right] \, ds \Big\|_{W^{1,1}([0,T];H^{-\beta}(\mathbb{T}^1))} \\
& + \mathbb{E} \Big\|\int_0^{\cdot} \sqrt{2} \partial_x\Big(\sigma_n(1 - \gamma^{2/3} u_{\varepsilon}^2) \, dW_{\delta} \Big) \Big\|_{W^{\alpha,2}([0,T];H^{-\beta}(\mathbb{T}^1))} \\
& + \sum_{j=1}^4 \mathbb{E} \Big\|\int_0^{\cdot} R_{j,\varepsilon} \, ds \Big\|_{W^{1,1}([0,T];H^{-\beta}(\mathbb{T}^1))} \\
& + \mathbb{E} \Big\|\int_0^{\cdot} \sqrt{2} \partial_x \Big( \sigma_n(1 - \gamma^{2/3} u_{\varepsilon}^2) g_{\varepsilon} \ast \eta_{\delta} \Big) \, ds \Big\|_{W^{1,1}([0,T];H^{-\beta}(\mathbb{T}^1))} \\
=:&\ \|u_{\gamma,0}\|_{L^1(\mathbb{T}^1)} T + I_1 + I_2 + I_3 + I_4 + I_5 + I_6 + I_7 + I_8.
\end{align*}

For the estimates of $I_1$ through $I_7$, we follow the same strategy as in Proposition \ref{second-uniform-es}. The term $I_8$ is additional, and to better understand the dependency on parameter $n$, we reconsider $I_7$ and note that
\begin{align*}
I_7 =\mathbb{E} \Big\|\int_0^{\cdot} R_{4,\varepsilon} \, ds \Big\|_{W^{1,1}([0,T];H^{-\beta}(\mathbb{T}^1))}\leq C \varepsilon \|\sigma_n'(\cdot)\|_{L^{\infty}(\mathbb{R})}^2 \gamma^{4/3} \delta^{-1/3}.
\end{align*}
It thus remains to estimate $I_8$. Using the boundedness of the prefactor $1 - \gamma^{2/3} u_{\varepsilon}^2$ and applying Young's inequality, we find
\begin{align*}
I_8 \lesssim \sup_{\varepsilon > 0} \|g_{\varepsilon}\|_{L^2(\Omega;L^2([0,T];L^2(\mathbb{T}^1)))}^2.
\end{align*}
This concludes the proof.

\end{proof}

\begin{theorem}\label{converg-stochasticcontrol}
Assume that the initial data and the interaction kernel $J$ satisfy Assumptions (A2) and \ref{Assump-J}, respectively. Let $a < 0$. For every $\gamma\in(0,1]$, $\varepsilon,\delta>0$ and $n\geq1$, let $g_{\varepsilon},g \in \mathcal{A}_T$. Let $u_{\varepsilon}$ denote the weak solution to \eqref{eqgamma-2} with initial data $u_{\gamma,0}$ and control $g_{\varepsilon}$, and let $u$ denote the weak solution to \eqref{skeleton-CahnHilliard} with initial data $u_{\gamma,0}$ and control $g$. Suppose that $g_{\varepsilon} \rightharpoonup g$ weakly in $L^2(\Omega;L^2([0,T];L^2(\mathbb{T}^1)))$. 

Then, under a scaling regime 
$$
(\varepsilon,\gamma(\varepsilon),\delta(\varepsilon),n(\varepsilon)) \rightarrow (0,0,0,+\infty)
$$
such that
\begin{align} \label{scaling}
\varepsilon\Big(\delta(\varepsilon)^{-2/3} + \gamma(\varepsilon)^{4/3} \delta(\varepsilon)^{-1/3} \|\sigma_{n(\varepsilon)}'(\cdot)\|_{L^{\infty}(\mathbb{R})}^2\Big) \rightarrow 0,
\end{align}
we have the convergence
\begin{align*}
\|u_{\varepsilon} - u\|_{L^2([0,T];L^2(\mathbb{T}^1))} \rightarrow 0
\end{align*}
in distribution as $\varepsilon \rightarrow 0$.
\end{theorem}
\begin{proof}
	Since $g_{\varepsilon} \rightharpoonup g$ weakly in $L^2(\Omega;L^2([0,T];L^2(\mathbb{T}^1)))$, there exists a constant $C > 0$ such that
\begin{align*}
\sup_{\varepsilon > 0} \|g_{\varepsilon}\|_{L^2(\Omega;L^2([0,T];L^2(\mathbb{T}^1)))} + \|g\|_{L^2(\Omega;L^2([0,T];L^2(\mathbb{T}^1)))} \leq C.
\end{align*}
As a consequence, under the scaling regime \eqref{scaling}, we may invoke Propositions \ref{entropydissipation-stochasticcontrol} and \ref{second-uniform-es-stochasticcontrol} to obtain the uniform estimate
\begin{align*}
\sup_{t \in [0,T]} \mathbb{E} \left( \int_{\mathbb{T}^1} \frac{1}{2} |u_{\varepsilon}(t)|^2 \, dx \right)
&+ \mathbb{E} \int_0^T \int_{\mathbb{T}^1} \frac{u_{\varepsilon}^2 |\partial_x u_{\varepsilon}|^2}{1 - \gamma^{2/3} u_{\varepsilon}^2} \, dx \, dt \\
&- a \, \mathbb{E} \int_0^T \|\partial_x u_{\varepsilon}\|_{L^2(\mathbb{T}^1)}^2 \, ds
+ \mathbb{E} \|u_{\varepsilon}\|_{W^{\alpha,2}([0,T];H^{-\beta}(\mathbb{T}^1))} \leq C.
\end{align*}

By applying the Aubin-Lions compactness criterion, it follows that the family 
$$
(u_{\varepsilon},\, \partial_x u_{\varepsilon}^2,\, (B^k)_{k \geq 0},\, (g_{\varepsilon})_{\varepsilon})_{\varepsilon > 0}
$$
is tight in the space
$$
L^2([0,T];L^2(\mathbb{T}^1)) \cap (L^2([0,T];H^1(\mathbb{T}^1)), w)
\times (L^2([0,T];L^2(\mathbb{T}^1)), w)
\times C([0,T];\mathbb{R}^{\mathbb{N}})
\times (L^2([0,T];L^2(\mathbb{T}^1)), w).
$$

Using the Skorokhod-Jakubowski representation theorem \cite{Jak97}, there exists a new probability space $(\bar{\Omega}, \bar{\mathcal{F}}, \bar{\mathbb{P}})$ and a sequence of random variables $(\bar{u}_{\varepsilon}, \bar{v}_{\varepsilon}, (\bar{B}^k_{\varepsilon})_{k \geq 0}, \bar{g}_{\varepsilon})$ and a limit $(\bar{u}, \bar{v}, (\bar{B}^k)_{k \geq 0}, \bar{g})$ such that for every $\varepsilon > 0$,
\begin{align*}
(u_{\varepsilon}, \partial_x u_{\varepsilon}^2, (B^k)_{k \geq 0}, g_{\varepsilon}) 
\overset{d}{=} 
(\bar{u}_{\varepsilon}, \bar{v}_{\varepsilon}, (\bar{B}^k_{\varepsilon})_{k \geq 0}, \bar{g}_{\varepsilon}), \\
(u, \partial_x u^2, (B^k)_{k \geq 0}, g) 
\overset{d}{=} 
(\bar{u}, \bar{v}, (\bar{B}^k)_{k \geq 0}, \bar{g}).
\end{align*}

Furthermore, the following convergences hold almost surely as $\varepsilon \to 0$: 
\begin{itemize}
    \item $\bar{u}_{\varepsilon} \to \bar{u}$ in $L^2([0,T];L^2(\mathbb{T}^1))\cap (L^2([0,T];H^1(\mathbb{T}^1)), w)$,
    \item $\bar{v}_{\varepsilon}  \rightharpoonup \bar{v}$ weakly in $L^2([0,T];L^2(\mathbb{T}^1))$,
    \item $\bar{g}_{\varepsilon} \rightharpoonup \bar{g}$ weakly in $L^2([0,T];L^2(\mathbb{T}^1))$,
    \item for every $k \geq 0$, $\sup_{t \in [0,T]} |\bar{B}^k_{\varepsilon}(t) - \bar{B}^k(t)| \to 0$.
\end{itemize}

Moreover, by repeating the identification argument used in Theorem \ref{thm-convergence}, we conclude that for every $\varepsilon > 0$, almost surely,
\begin{align*}
\bar{v}_{\varepsilon} = \partial_x \bar{u}_{\varepsilon}^2\  \text{and } \bar{v} = \partial_x \bar{u}^2, \quad \text{for almost every } (t,x) \in [0,T] \times \mathbb{T}^1.
\end{align*}
Let $\bar{W}_{\varepsilon}$ and $\bar{W}$ denote the noises generated by $(\bar{B}^k_{\varepsilon})_{k\geq0}$ and $(\bar{B}^k)_{k\geq0}$, respectively. 

By applying the definitions of weak solutions to equations \eqref{eqgamma-2} and \eqref{skeleton-CahnHilliard}, we obtain the following identity, valid almost surely for every $\varphi \in C^{\infty}(\mathbb{T}^1)$ and every $t \in [0,T]$:
\begin{align}\label{passagetothelimits-2}
	\langle (\bar{u}_{\varepsilon}-\bar{u})(t),\varphi\rangle=&(\langle u_{\gamma(\varepsilon),0},\varphi\rangle-\langle u_0,\varphi\rangle)+\frac{D}{2}\int^t_0\langle \bar{u}_{\varepsilon}-\bar{u},\partial_{xxxx}^4\varphi\rangle ds\notag\\
	&-\int^t_0\langle\Big[\frac{1}{3}(\bar{u}_{\varepsilon}^3-\bar{u}^3)-a(\bar{u}_{\varepsilon}-\bar{u})\Big],\partial_{xx}^2\varphi\rangle ds\notag\\
	&-\sqrt{2}\varepsilon^{1/2}\int^t_0\langle \partial_{x}\varphi,(\sigma_{n(\varepsilon)}(1-\gamma(\varepsilon)^{2/3}\bar{u}_{\varepsilon}^2)-1)d\bar{W}_{\varepsilon}\ast\eta_{\delta(\varepsilon)}\rangle\notag\\
	&-\sqrt{2}\varepsilon^{1/2}\int^t_0\langle\partial_x\varphi,(d\bar{W}_{\varepsilon}\ast{\eta_{\delta(\varepsilon)}}-d\bar{W})\rangle\notag\\
	&-\int^t_0\langle R_{1,\varepsilon}+R_{2,\varepsilon}+R_{3,\varepsilon}+R_{4,\varepsilon},\varphi\rangle ds\notag\\
	&+\int^t_0\langle\partial_x\varphi,\sigma_{n(\varepsilon)}(1-\gamma(\varepsilon)^{2/3}\bar{u}_{\varepsilon}^2)\bar{g}_{\varepsilon}\ast\eta_{\delta(\varepsilon)}-\bar{g}\rangle ds.  
\end{align}

Compared to the proof of Theorem \ref{thm-convergence}, the only additional term is the one involving the stochastic control. All other terms can be handled using estimates identical to those presented in the proof of Theorem \ref{thm-convergence}. Therefore, it suffices to address the convergence of the remaining term:
$$
\int^t_0\langle\partial_x\varphi,\sigma_{n(\varepsilon)}(1-\gamma(\varepsilon)^{2/3}\bar{u}_{\varepsilon}^2)\bar{g}_{\varepsilon}\ast\eta_{\delta(\varepsilon)}-\bar{g}\rangle ds \rightarrow 0, \quad \text{almost surely.}
$$
To this end, we decompose the integrand as follows:
\begin{align*}
\int^t_0\langle\partial_x\varphi,\sigma_{n(\varepsilon)}(1-\gamma(\varepsilon)^{2/3}\bar{u}_{\varepsilon}^2)\bar{g}_{\varepsilon}\ast\eta_{\delta(\varepsilon)}-\bar{g}\rangle ds &= \int^t_0\langle\partial_x\varphi,[\sigma_{n(\varepsilon)}(1-\gamma(\varepsilon)^{2/3}\bar{u}_{\varepsilon}^2)-1]\bar{g}_{\varepsilon}\ast\eta_{\delta(\varepsilon)}\rangle ds\\
&\quad + \int^t_0\langle\partial_x\varphi,\bar{g}_{\varepsilon}\ast\eta_{\delta(\varepsilon)}-\bar{g}_{\varepsilon}\rangle ds\\
&\quad + \int^t_0\langle\partial_x\varphi,\bar{g}_{\varepsilon}-\bar{g}\rangle ds\\
&=: I_1 + I_2 + I_3.
\end{align*}

To estimate $I_1$, we use H\"older's inequality, the Young convolution inequality, and the dominated convergence theorem to obtain:
$$
I_1 \leq C(\varphi)\|\sigma_{n(\varepsilon)}(1-\gamma(\varepsilon)\bar{u}_{\gamma(\varepsilon)}^2)-1\|_{L^2([0,T];L^2(\mathbb{T}^1))} \sup_{\varepsilon>0}\|\bar{g}_{\varepsilon}\|_{L^2(\bar{\Omega};L^2([0,T];L^2(\mathbb{T}^1)))} \rightarrow 0
$$
as $\varepsilon \to 0$.  
For $I_2$, by exploiting the symmetry of the convolution kernel $\eta_{\delta(\varepsilon)}$ and applying H\"older's inequality, we deduce
$$
I_2 = \int^t_0\langle\partial_x\varphi\ast\eta_{\delta(\varepsilon)} - \partial_x\varphi, \bar{g}_{\varepsilon}\rangle ds \leq \|\partial_x\varphi\ast\eta_{\delta(\varepsilon)} - \partial_x\varphi\|_{L^2([0,T];L^2(\mathbb{T}^1))} \sup_{\varepsilon>0}\|\bar{g}_{\varepsilon}\|_{L^2(\bar{\Omega};L^2([0,T];L^2(\mathbb{T}^1)))} \rightarrow 0,
$$
as $\varepsilon \to 0$.  

Finally, the convergence $I_3 \to 0$ follows from the weak convergence $\bar{g}_{\varepsilon} \rightharpoonup \bar{g}$ in $L^2(\bar{\Omega};L^2([0,T];L^2(\mathbb{T}^1)))$. Hence, we conclude that the limit $u$ satisfies the skeleton equation \eqref{skeleton-CahnHilliard} in the weak sense.

It remains to show that the entire sequence $(u_{\varepsilon})_{\varepsilon>0}$ on the original probability space converges in distribution. Since $(\bar{u}_{\varepsilon})_{\varepsilon>0}$ and its version $(u_{\varepsilon})_{\varepsilon>0}$ on the original probability space are identically distributed, the version on the original space also admits a subsequence converging in distribution to $u$ in $L^2([0,T];L^2(\mathbb{T}^1))$. For an arbitrary subsequence $(u_{\varepsilon_k})_{k>0}$, the compactness argument and the preceding passage-to-the-limit analysis yield the existence of a further subsequence $(u_{\varepsilon_{k_l}})_{l>0}$ that converges in distribution to $u$ in $L^2([0,T];L^2(\mathbb{T}^1))$. This shows that the whole sequence $(u_{\varepsilon})_{\varepsilon>0}$ converges in distribution to $u$ in $L^2([0,T];L^2(\mathbb{T}^1))$, thereby completing the proof.

\end{proof}

In the sequel, we present the main results concerning the rescaled large deviation principle.

\begin{theorem}
Assume that the initial data and the interaction kernel $J$ satisfy Assumptions (A2) and \ref{Assump-J}, respectively. For every $\gamma\in(0,1]$, $\varepsilon,\delta>0$ and $n\geq1$, let $u_{\varepsilon}$ denote the weak solution to \eqref{eqgamma-1} with initial data $0$. Suppose that $a < 0$. Then under the scaling regime 
$$
(\varepsilon,\gamma(\varepsilon),\delta(\varepsilon),n(\varepsilon)) \to (0,0,0,+\infty)
$$
such that
$$
\varepsilon\Big(\delta(\varepsilon)^{-2/3} + \gamma(\varepsilon)^{4/3} \delta(\varepsilon)^{-1/3} \|\sigma_{n(\varepsilon)}'(\cdot)\|_{L^{\infty}(\mathbb{R})}^2\Big) \to 0,
$$
the family of laws of $(u_{\varepsilon})_{\varepsilon>0}$ satisfies a large deviation principle with the good rate function $\mathcal{I}_{CH}: L^2([0,T];L^2(\mathbb{T}^1)) \to [0,+\infty]$ given by
\begin{align}\label{rate-CH}
\mathcal{I}_{CH}(u) := \inf \left\{ \frac{1}{2} \|g\|_{L^2([0,T];L^2(\mathbb{T}^1))}^2 : \partial_t u = \partial_{xx}^2\left[ V'(u) - \frac{D}{2} \partial_{xx}^2 u \right] - \sqrt{2}\,\partial_x g,\quad u(0) = u_0 \right\}.
\end{align}
\end{theorem}
\begin{proof}
The proof proceeds by applying the weak convergence framework developed in \cite{BDM11}. By combining Theorems \ref{stability-skeleton} and \ref{converg-stochasticcontrol}, we establish the stated large deviation principle and identify the associated rate function. This concludes the proof. 
\end{proof}

\section{$\Gamma$-convergence of the Ising-Kac-Kawasaki rate function}\label{sec-6}

In this section, we assume that the initial data and the interaction kernel $J$ satisfy Assumptions (A2) and \ref{Assump-J}, respectively, and we establish a $\Gamma$-convergence result for the rate function associated with the Ising-Kac-Kawasaki (IKK) dynamics.

Following the large deviation framework developed in \cite{WZ22}, we consider the scaling regime $(\varepsilon, \delta(\varepsilon), n(\varepsilon)) \to (0, 0, +\infty)$ and fix a parameter $\gamma \in (0,1)$. Under this regime, the family of stochastic processes defined by equation \eqref{eqgamma-1} satisfies a large deviation principle with speed $\varepsilon$ and rate function given by
\begin{align}\label{rate-gamma}
\mathcal{I}_{IKK,\gamma}(u)=\frac{1}{2}\inf\Big\{\|g\|_{L^2([0,T];L^2(\mathbb{T}^1))}^2:\partial_tu=&\gamma^{-2/3}\partial_{xx}^2u-\gamma^{-2/3}(1+a\gamma^{2/3})\partial_x[(1-\gamma^{2/3}u^2)J_{\gamma^{1/3}}\ast\partial_xu]\notag\\
&-\sqrt{2}\partial_x(\sqrt{1-\gamma^{2/3}u^2}g),\ u(0)=u_{\gamma,0}\Big\},
\end{align}
for all $u \in L^2([0,T];H^1(\mathbb{T}^1))$ satisfying 
$$
\int_0^T\int_{\mathbb{T}^1} \frac{|\partial_x u|^2}{1 - u^2} \, dx \, dt < \infty.
$$
Otherwise, we set $\mathcal{I}_{IKK,\gamma}(u) = +\infty$. Let $\mathcal{I}_{CH}$ denote the Cahn--Hilliard rate function defined in \eqref{rate-CH}. Our main result in this section is the following $\Gamma$-convergence statement:
\begin{align}
\mathcal{I}_{IKK,\gamma} \overset{\Gamma}{\longrightarrow} \mathcal{I}_{CH},
\end{align}
as $\gamma \to 0$.

\subsection{Convergence of the skeleton equation}
For every $\gamma \in (0,1)$, let $g \in L^2([0,T];L^2(\mathbb{T}^1))$ and $u_{\gamma,0} \in L^2(\mathbb{T}^1)$. We consider the skeleton equation associated with the rate function \eqref{rate-gamma}, given by
\begin{align}\label{skeleton-gamma}
	\partial_t u_{\gamma} &= \gamma^{-2/3} \partial_{xx}^2 u_{\gamma} - \gamma^{-2/3}(1 + a\gamma^{2/3}) \partial_x \left[(1 - \gamma^{2/3} u_{\gamma}^2) J_{\gamma^{1/3}} \ast \partial_x u_{\gamma} \right] \notag \\
	&\quad - \sqrt{2} \, \partial_x \left( \sqrt{1 - \gamma^{2/3} u_{\gamma}^2} \, g \right), \quad u_{\gamma}(0) = u_{\gamma,0}.
\end{align}

The well-posedness theory for the Dean-Kawasaki skeleton equation, which shares a similar nonlinear and nonlocal structure with \eqref{skeleton-gamma}, has been developed in \cite{FG23}. Building upon this framework, \cite{WZ24} adapted the techniques introduced in \cite{FG23} to establish the well-posedness of equation \eqref{skeleton-gamma}. For completeness, we briefly summarize the main results of \cite{FG23}. The authors established the uniqueness of the renormalized kinetic solution (see \cite[Theorem 8]{FG23}) and the existence of a weak solution (see \cite[Proposition 20]{FG23}) for the Dean-Kawasaki skeleton equation. Moreover, they showed that these two notions of solution coincide, leveraging the additional regularity provided by entropy dissipation estimates.

In the present work, we will employ only the notion of weak solutions to \eqref{skeleton-gamma}. For the reader's convenience, we recall the definition below.

\begin{definition}\label{dfn-skeleton-ising}
	For every $\gamma\in(0,1]$, let $g \in L^2([0,T];L^2(\mathbb{T}^1))$ and $u_{\gamma,0} \in L^2(\mathbb{T}^1; [-\gamma^{-1/3}, \gamma^{-1/3}])$. A \emph{weak solution} to \eqref{skeleton-gamma} is a function
	$$
	u_{\gamma} \in C([0,T];L^2(\mathbb{T}^1;[-\gamma^{-1/3}, \gamma^{-1/3}])) \cap L^2([0,T];H^1(\mathbb{T}^1))
	$$
	that satisfies, for every $\psi \in C^{\infty}(\mathbb{T}^1)$ and every $t \in [0,T]$,
	\begin{align*}
	\int_{\mathbb{T}^1} u_{\gamma}(t) \psi \, dx &= \int_{\mathbb{T}^1} u_{\gamma,0} \psi \, dx - \gamma^{-2/3} \int_0^t \int_{\mathbb{T}^1} \partial_x u_{\gamma} \, \partial_x \psi \, dx \, ds \\
	&\quad + \gamma^{-2/3}(1 + a\gamma^{2/3}) \int_0^t \int_{\mathbb{T}^1} \partial_x \psi(x) \cdot \left[ (1 - \gamma^{2/3} u_{\gamma}^2) J_{\gamma^{1/3}} \ast \partial_x u_{\gamma} \right] dx \, ds \\
	&\quad + \sqrt{2} \int_0^t \int_{\mathbb{T}^1} \sqrt{1 - \gamma^{2/3} u_{\gamma}^2} \, g \, \partial_x \psi \, dx \, ds.
	\end{align*}
\end{definition}

In this section, we prove the convergence of solutions to the skeleton equation \eqref{skeleton-gamma} towards solutions of the limiting skeleton equation \eqref{skeleton-CahnHilliard}. 

As a preliminary step, we establish a uniform entropy dissipation estimate and a time-regularity estimate that will be instrumental in the compactness arguments used later. 

\begin{lemma}[Uniform entropy dissipation estimates]\label{entropydissipation-skeleton}
Assume that the initial data and the interaction kernel $J$ satisfy Assumptions (A2) and \ref{Assump-J}, respectively. Let $\gamma\in(0,1]$ and let $g \in L^2([0,T];L^2(\mathbb{T}^1))$. Let $u_{\gamma}$ denote the weak solution to \eqref{skeleton-gamma} with initial condition $u_{\gamma}(0) = u_{\gamma,0}$ and control $g$. Let $\Psi_{\gamma}$ be the entropy functional defined by \eqref{rescaled-entropy}. Then the following estimate holds 
\begin{align}\label{uniformentropy-gep}
&\sup_{t \in [0,T]} \int_{\mathbb{T}^1} \Psi_{\gamma}(u_{\gamma}(t)) \, dx 
+ \gamma^{-1/3} \int_0^T \int_{\mathbb{T}^1} \frac{|\partial_x u_{\gamma}|^2}{1 - \gamma^{2/3} u_{\gamma}^2} \, dx \, dt \notag \\
&\quad + \gamma^{-1/3}(1 + a \gamma^{2/3}) \int_0^T \int_{\mathbb{T}^1} \partial_x u_{\gamma} \cdot (J_{\gamma^{1/3}} \ast \partial_x u_{\gamma}) \, dx \, ds 
\lesssim 1 + \|g\|_{L^2([0,T];L^2(\mathbb{T}^1))}^2, 
\end{align}
uniformly for every $\gamma\in(0,1]$.

Moreover, if $a < 0$, then we also have the uniform estimate
\begin{align}\label{uniformestimate-gep}
\sup_{t \in [0,T]} \int_{\mathbb{T}^1} \frac{1}{2} |u_{\gamma}(t)|^2 \, dx 
+ \int_0^T \int_{\mathbb{T}^1} \frac{u_{\gamma}^2 |\partial_x u_{\gamma}|^2}{1 - \gamma^{2/3} u_{\gamma}^2} \, dx \, dt 
- a \int_0^T \|\partial_x u_{\gamma}\|_{L^2(\mathbb{T}^1)}^2 \, ds 
\lesssim 1 + \|g\|_{L^2([0,T];L^2(\mathbb{T}^1))}^2,
\end{align}
for every $\gamma\in(0,1]$. 
\end{lemma}
\begin{proof}
The proof follows by combining the arguments used in Proposition \ref{prp-entropydissipation} and Proposition \ref{entropydissipation-stochasticcontrol}.  
\end{proof}

\begin{lemma}[Time regularity of the skeleton equation]\label{time-regularity-skeleton}
Assume that the initial data and the interaction kernel $J$ satisfy Assumptions (A2) and \ref{Assump-J}, respectively. Let $\gamma\in(0,1]$ and let $g \in L^2([0,T];L^2(\mathbb{T}^1))$. Let $u_{\gamma}$ denote the weak solution to \eqref{skeleton-gamma} with initial data $u_{\gamma,0}$ and control $g$. Suppose that $a < 0$. Then, for every $\alpha \in (0, \tfrac{1}{2})$ and $\beta > \tfrac{13}{2}$, there exists a constant $C = C(T) > 0$ such that
\begin{equation}
\|\partial_t u_{\gamma}\|_{L^1([0,T];H^{-\beta}(\mathbb{T}^1))} \leq C \left( 1 + \|g\|_{L^2([0,T];L^2(\mathbb{T}^1))}^2 \right),
\end{equation}
for every $\gamma\in(0,1]$. 
\end{lemma}

\begin{proof}
The estimate follows by applying the same argument as in Proposition~\ref{second-uniform-es-stochasticcontrol}, adapted to the deterministic skeleton setting.
\end{proof}

We are now in a position to establish the convergence of the solutions to the skeleton equation \eqref{skeleton-gamma} toward the corresponding solution of the limiting skeleton equation \eqref{skeleton-CahnHilliard}.

\begin{proposition}[Convergence of skeleton equations]\label{convergence-skeleton}
Assume that the initial data and the interaction kernel $J$ satisfy Assumptions (A2) and \ref{Assump-J}, respectively. Let $a < 0$, and for each $\gamma\in(0,1]$, let $g \in L^2([0,T];L^2(\mathbb{T}^1))$. Let $u_{\gamma}$ denote the weak solution of \eqref{skeleton-gamma} with initial data $u_{\gamma,0}$ and control $g$, and let $u$ denote the weak solution of \eqref{skeleton-CahnHilliard} with initial condition $u(0) = u_0$ and the same control $g$. Then,
\begin{align*}
\|u_{\gamma} - u\|_{L^2([0,T];L^2(\mathbb{T}^1))} \longrightarrow 0 \quad \text{as } \gamma \to 0.
\end{align*}
\end{proposition}
\begin{proof}
By Lemmas \ref{entropydissipation-skeleton} and \ref{time-regularity-skeleton}, together with the Sobolev embedding $W^{1,1}([0,T]) \subset W^{\alpha,2}([0,T])$ for every $\alpha \in (0,1/2)$, and the Aubin-Lions compactness criterion, the family $(u_{\gamma})_{\gamma\in(0,1]}$ is relatively compact in $L^2([0,T];L^2(\mathbb{T}^1))$. 

Applying a passage-to-the-limit argument analogous to that used in the proofs of Proposition~\ref{prp-convergence} and Theorem~\ref{converg-stochasticcontrol}, we conclude that any limit point of $(u_{\gamma})_{\gamma\in(0,1]}$ must coincide with the weak solution of \eqref{skeleton-CahnHilliard} with initial condition $u_0$ and control $g$. This proves the desired convergence. 
\end{proof}

\subsection{The $\Gamma$-convergence} 
Let 
\begin{align*}
\mathbb{X} := L^2([0,T];L^2(\mathbb{T}^1))\cap (L^2([0,T];H^1(\mathbb{T}^1)),w) &\cap (L^4([0,T];L^4(\mathbb{T}^1)),w^*),
\end{align*}
equipped with the subspace strong topology inherited from $L^2([0,T];L^2(\mathbb{T}^1))$, weak topology of $L^2([0,T];H^1(\mathbb{T}^1))$ and weak* topology of $L^4([0,T];L^4(\mathbb{T}^1))$. Refer to \cite{M12}, we now introduce the notion of $\Gamma$-convergence in this context. 	
\begin{definition}[$\Gamma$-convergence]\label{def-gammaconvergence} We say that $\mathcal{I}_{IKK,\gamma}$ $\Gamma$-converges to $\mathcal{I}_{CH}$ as $\gamma \to 0$ if: 

\noindent (i) For every sequence $(u_{\gamma_k})_{k\in\mathbb{N}}\subset \mathbb{X}$ satisfying
$$
u_{\gamma_k} \to u \quad \text{in } \mathbb{X}, \quad \text{as } k \to \infty, \ \gamma_k \to 0,
$$
we have the lower bound
$$
\liminf_{k \to \infty} \mathcal{I}_{IKK,\gamma_k}(u_{\gamma_k}) \geq \mathcal{I}_{CH}(u).
$$

\medskip

\noindent (ii) For every $u \in \mathbb{X}$, there exists a sequence $(u_{\gamma_k})_{k \in \mathbb{N}} \subset \mathbb{X}$ such that
$$
u_{\gamma_k} \to u \quad \text{in } \mathbb{X}, \quad \text{as } k \to \infty, \ \gamma_k \to 0,
$$
and
$$
\limsup_{k \to \infty} \mathcal{I}_{IKK,\gamma_k}(u_{\gamma_k}) \leq \mathcal{I}_{CH}(u).
$$

\end{definition}

We require the following characterization of the rate function $\mathcal{I}_{IKK,\gamma}$. Fix $\gamma \in (0,1)$, and let $u \in L^{\infty}([0,T];L^{\infty}(\mathbb{T}^1))$ be such that $-\gamma^{-1/3} \leq u \leq \gamma^{-1/3}$. Define an equivalence relation $\thicksim$ on $C^{\infty}(\mathbb{T}^1 \times [0,T])$ by declaring
$$
\psi \thicksim \phi \quad \Longleftrightarrow \quad \int_0^T \int_{\mathbb{T}^1} (1 - \gamma^{2/3} u^2) |\partial_x \psi - \partial_x \phi|^2 \,dx\,dt = 0.
$$
Let $H^1_u$ denote the Hilbert space given by the completion of the set of equivalence classes $C^{\infty}(\mathbb{T}^1 \times [0,T])/\thicksim$ under the inner product
$$
\langle \phi, \psi \rangle_{H^1_u} := \int_0^T \int_{\mathbb{T}^1} (1 - \gamma^{2/3} u^2) \, \partial_x \phi \, \partial_x \psi \, dx\,dt, \quad \text{with norm} \quad \|\phi\|_{H^1_u} := \langle \phi, \phi \rangle_{H^1_u}^{1/2}.
$$

We now state a characterization of the rate function $\mathcal{I}_{IKK,\gamma}$ for every fixed $\gamma \in(0,1)$.

\begin{lemma}\label{risez}
Assume that the initial data and the interaction kernel $J$ satisfy Assumptions (A2) and \ref{Assump-J}, respectively. Fix $\gamma \in(0,1)$, and suppose $u$ satisfies $\mathcal{I}_{IKK,\gamma}(u) < \infty$. Then there exists a function $\Psi_u \in H^1_u$ such that the rate function $\mathcal{I}_{IKK,\gamma}$ admits the following equivalent representation:
\begin{align}\label{RF low}
\mathcal{I}_{IKK,\gamma}(u) = \frac{1}{2} \int_0^T \int_{\mathbb{T}^1} (1 - \gamma^{2/3} u^2) |\partial_x \Psi_u|^2 \, dx \, dt = \frac{1}{2} \| \Psi_u \|^2_{H^1_u}.
\end{align}
Furthermore, such a $\Psi_u$ is unique in $H^1_u$. 
\end{lemma}
\begin{proof}
The proof can be obtained by applying a Riesz representation argument in $H^1_u$. A similar proof can be found in \cite[Proposition 37]{FG23}, thus we omit it. 
\end{proof} 
Similarly, we have the following characterization of the rate function $\mathcal{I}_{CH}$. 
\begin{lemma}\label{risez-2}
Suppose $u$ satisfies $\mathcal{I}_{CH}(u) < \infty$. Then there exists a unique function $g_u \in L^2([0,T];L^2(\mathbb{T}^1))$ such that the rate function $\mathcal{I}_{CH}$ admits the following equivalent representation:
\begin{align}\label{RF low-2}
\mathcal{I}_{CH}(u) = \frac{1}{2} \int_0^T \int_{\mathbb{T}^1}|g_u|^2 \, dx \, dt = \frac{1}{2} \| g_u \|^2_{L^2([0,T];L^2(\mathbb{T}^1))}. 
\end{align} 
\end{lemma}

We are now ready to state the main $\Gamma$-convergence result.

\begin{theorem}
Assume that the initial data and the interaction kernel $J$ satisfy Assumptions (A2) and \ref{Assump-J}, respectively. For every $\gamma \in (0,1)$, let $\mathcal{I}_{IKK,\gamma}$ be the rate function defined in \eqref{rate-gamma}, and let $\mathcal{I}_{CH}$ be the rate function defined in \eqref{rate-CH}. Then, as $\gamma \to 0$, the family $(\mathcal{I}_{IKK,\gamma})_{\gamma\in(0,1]}$ $\Gamma$-converges to $\mathcal{I}_{CH}$ in the space $\mathbb{X}$:
\begin{align}
\mathcal{I}_{IKK,\gamma} \xrightarrow{\Gamma} \mathcal{I}_{CH}, \quad \text{as } \gamma \to 0.
\end{align}
\end{theorem}
\begin{proof}

\noindent \textbf{Proof of (i) in Definition \ref{def-gammaconvergence}.}  
Let $(u_{\gamma_k})_{k \in \mathbb{N}}$ be a sequence such that $u_{\gamma_k} \to u$ in $\mathbb{X}$ as $k \to \infty$. If 
$$
\liminf_{k \to \infty} \mathcal{I}_{IKK,\gamma_k}(u_{\gamma_k}) = +\infty,
$$
then the inequality holds trivially. Otherwise, assume
$$
\liminf_{k \to \infty} \mathcal{I}_{IKK,\gamma_k}(u_{\gamma_k}) < +\infty.
$$
By passing to a subsequence, denoted by $(\gamma_l)$, we may assume
$$
\liminf_{k \to \infty} \mathcal{I}_{IKK,\gamma_k}(u_{\gamma_k}) = \lim_{l \to \infty} \mathcal{I}_{IKK,\gamma_l}(u_{\gamma_l}).
$$

By Lemma \ref{risez}, for each fixed $l$, there exists a control $g_l \in L^2([0,T];L^2(\mathbb{T}^1))$ such that
$$
\frac{1}{2} \|g_l\|_{L^2([0,T];L^2(\mathbb{T}^1))}^2 = \mathcal{I}_{IKK,\gamma_l}(u_{\gamma_l}).
$$
Since the right-hand side is bounded uniformly in $l$, we deduce the existence of $g \in L^2([0,T];L^2(\mathbb{T}^1))$ with
$$
g_l \rightharpoonup g \quad \text{weakly in } L^2([0,T];L^2(\mathbb{T}^1)).
$$

By Lemma \ref{risez} again, $u_{\gamma_l}$ solves the skeleton equation \eqref{skeleton-gamma} with control $g_l$. Using the convergence $u_{\gamma_l} \to u$ in $\mathbb{X}$ and arguing as in the proofs of Proposition \ref{convergence-skeleton} and Theorem \ref{converg-stochasticcontrol}, we pass to the limit in \eqref{skeleton-gamma} and conclude that $u$ is a weak solution of \eqref{skeleton-CahnHilliard} with initial data $u_0$ and control $g$. Hence, by the lower semi-continuity of the weak convergence of $L^2([0,T];L^2(\mathbb{T}^1))$ and the definition of the rate function $\mathcal{I}_{CH}$, 
$$
\liminf_{k \to \infty} \mathcal{I}_{IKK,\gamma_k}(u_{\gamma_k}) \geq \frac{1}{2} \|g\|_{L^2([0,T];L^2(\mathbb{T}^1))}^2 \geq \mathcal{I}_{CH}(u).
$$
This completes the proof of (i).

\medskip

\noindent \textbf{Proof of (ii) in Definition \ref{def-gammaconvergence}.}  
Let $u \in \mathbb{X}$. If $\mathcal{I}_{CH}(u) = +\infty$, the statement holds trivially. Otherwise, thanks to Lemma \ref{risez-2}, there exists a control $g_u \in L^2([0,T];L^2(\mathbb{T}^1))$ such that $u$ is the weak solution of \eqref{skeleton-CahnHilliard} with control $g_u$ and
$$
\frac{1}{2} \|g_u\|_{L^2([0,T];L^2(\mathbb{T}^1))}^2 = \mathcal{I}_{CH}(u).
$$

For each $\gamma \in (0,1)$, let $u_{\gamma}$ be the weak solution of \eqref{skeleton-gamma} with control $g$ and initial data $u_{\gamma,0}$ satisfying Assumption (A2). By Proposition \ref{convergence-skeleton}, we have
$$
u_{\gamma} \to u \quad \text{in } \mathbb{X} \quad \text{as } \gamma \to 0.
$$
By the definition of $\mathcal{I}_{IKK,\gamma}$, it follows that
$$
\limsup_{\gamma \to 0} \mathcal{I}_{IKK,\gamma}(u_{\gamma}) \leq \frac{1}{2} \|g_u\|_{L^2([0,T];L^2(\mathbb{T}^1))}^2 = \mathcal{I}_{CH}(u).
$$
This completes the proof of (ii).

\end{proof}

\noindent{\bf  Acknowledgements}\quad This work is supported by the US Army Research Office, grant W911NF2310230.

\appendix
\renewcommand{\appendixname}{Appendix~\Alph{section}}
\renewcommand{\theequation}{A.\arabic{equation}}
\section{Approximation of the stochastic Cahn-Hilliard equation}\label{sec-app-A}
In this appendix, we present a rigorous proof of Lemma \ref{lem-approx-CH}. For the reader's convenience, we restate the lemma here.

\begin{lemma}
Suppose that $a < 0$. For every $\delta > 0$, let $u_0 \in H^{-1}(\mathbb{T}^1)$ be given, and denote by $u_{\delta}$ the weak solution to the approximated equation
\begin{align}\label{CahnHilliard-delta-appendix}
    du_{\delta} = \partial_{xx}^2 \Big[ V'(u_{\delta}) - \frac{D}{2} \partial_{xx}^2 u_{\delta} \Big] dt - \sqrt{2} \partial_x dW_{\delta}, \quad u_{\delta}(0) = u_0.
\end{align}
Let $u$ be the weak solution to
\begin{align*}
	du = \partial_{xx}^2\left[V'(u) - \frac{D}{2} \partial_{xx}^2 u \right]dt - \sqrt{2} \, \partial_x dW, \quad u(0) = u_0.
\end{align*}
with the same initial condition $u_0$ and the same noise $W$.  Then the following convergence holds:
$$
\mathbb{E} \| u_{\delta} - u \|_{L^2([0,T]; L^2(\mathbb{T}^1))}^2 \to 0 \quad \text{as } \delta \to 0.
$$
\end{lemma}

\begin{proof}
	Consider the following decomposition: let $z$ be the solution to
\begin{align*}
    dz = -\frac{D}{2} (\partial_{xx}^2)^2 z\, dt - \sqrt{2}\, \partial_x dW, \quad z(0) = 0,
\end{align*}
and let $w$ be the solution to
\begin{align*}
    dw = \partial_{xx}^2 \Big[ V'(w + z) - \frac{D}{2} \partial_{xx}^2 w \Big] dt, \quad w(0) = u_0.
\end{align*}
Similarly, we consider the analogous decomposition for equation \eqref{CahnHilliard-delta-appendix}. Let $z_{\delta}$ be the solution to
\begin{align*}
    dz_{\delta} = -\frac{D}{2} (\partial_{xx}^2)^2 z_{\delta} \, dt - \sqrt{2}\, \partial_x dW_{\delta}, \quad z_{\delta}(0) = 0,
\end{align*}
and let $w_{\delta}$ be the solution to
\begin{align*}
    dw_{\delta} = \partial_{xx}^2 \Big[ V'(w_{\delta} + z_{\delta}) - \frac{D}{2} \partial_{xx}^2 w_{\delta} \Big] dt, \quad w_{\delta}(0) = u_0.
\end{align*}

\textbf{Step 1. Uniform estimates.}

We begin by establishing several uniform estimates for the processes $z, z_{\delta}, w,$ and $w_{\delta}$. 

Let $(e^{-(\partial_{xx}^2)^2 t})_{t \in [0,T]}$ denote the semigroup generated by $-(\partial_{xx}^2)^2$, and let $(K_t(x))_{(t,x) \in [0,T] \times \mathbb{T}^1}$ be the associated convolution kernel on the torus; see \cite{RYZ21} for further details. Applying the $L^p$-isometry for stochastic integrals with $p > 1$ \cite[Corollary 3.11]{NW}, together with integration by parts and Young's convolution inequality, we deduce that 
\begin{align*}
    \mathbb{E}\|z_{\delta}\|_{L^4([0,T]; L^4(\mathbb{T}^1))}^4 
    &= 2 \mathbb{E} \int_0^T \int_{\mathbb{T}^1} \Big| \int_0^t \int_{\mathbb{T}^1} \partial_x K_{t-s}(x - y) \, dW_{\delta}(s,y) \Big|^4 dx\, dt \\
    &\lesssim \int_0^T \int_{\mathbb{T}^1} \Big( \mathbb{E} \int_0^t \| \partial_x K_{t-s} \ast \eta_{\delta}(x - \cdot) \|_{L^2(\mathbb{T}^1)}^2 ds \Big)^2 dx\, dt \\
    &\leq T \cdot \Big( \int_0^T \| \partial_x K_s \ast \eta_{\delta} \|_{L^2(\mathbb{T}^1)}^2 ds \Big)^2 \\
    &\leq T \cdot \Big( \int_0^T \| \partial_x K_s \|_{L^2(\mathbb{T}^1)}^2 ds \Big)^2.
\end{align*}
By the definition of the semigroup $(e^{-(\partial_{xx}^2)^2 t})_{t \in [0,T]}$ and Parseval's identity, we have
\begin{align}\label{finite-L2}
    \int_0^T \| \partial_x K_s \|_{L^2(\mathbb{T}^1)}^2 ds 
    &= \int_0^T \sum_{k \geq 0} \langle \partial_x K_s, e_k \rangle^2 ds \notag \\
    &= \int_0^T \sum_{k \geq 0} \big(\partial_x e^{-(\partial_{xx}^2)^2 s} e_k\big)(0)^2 ds \notag \\
    &= \int_0^T \sum_{k \geq 0} |k|^2 e^{-2|k|^4 s} ds \notag \\
    &= \sum_{k \geq 0} |k|^2 \left( \frac{1 - e^{-2|k|^4 T}}{2 |k|^4} \right) < \infty.
\end{align}
Using the same argument for $z$, we conclude that there exists a constant $C > 0$ such that
\begin{align}\label{uniformes-z}
    \sup_{\delta > 0} \mathbb{E} \| z_{\delta} \|_{L^4([0,T]; L^4(\mathbb{T}^1))}^4 + \mathbb{E} \| z \|_{L^4([0,T]; L^4(\mathbb{T}^1))}^4 \leq C.
\end{align}

Next, applying It\^o's formula to the functional $\frac{1}{2} \langle (-\partial_{xx}^2)^{-1} w_{\delta}, w_{\delta} \rangle$, we obtain
\begin{align*}
    \frac{1}{2} \partial_t \langle (-\partial_{xx}^2)^{-1} w_{\delta}, w_{\delta} \rangle
    &= - \| \partial_x w_{\delta} \|_{L^2(\mathbb{T}^1)}^2 - \int_{\mathbb{T}^1} w_{\delta} (w_{\delta} + z_{\delta})^3 - a w_{\delta} (w_{\delta} + z_{\delta}) \, dx \\
    &= - \| \partial_x w_{\delta} \|_{L^2(\mathbb{T}^1)}^2 - \int_{\mathbb{T}^1} w_{\delta} \big( w_{\delta}^3 + z_{\delta}^3 + 3 w_{\delta}^2 z_{\delta} + 3 w_{\delta} z_{\delta}^2 \big) - a w_{\delta} (w_{\delta} + z_{\delta}) \, dx.
\end{align*}
Taking expectations, using Young's inequality, and invoking \eqref{uniformes-z}, we deduce
\begin{align*}
    &\sup_{t \in [0,T]} \frac{1}{2} \mathbb{E} \| w_{\delta}(t) \|_{\dot{H}^{-1}(\mathbb{T}^1)}^2 + \mathbb{E} \int_0^T \| \partial_x w_{\delta} \|_{L^2(\mathbb{T}^1)}^2 ds-\frac{a}{2}\mathbb{E}\int^T_0\|w_{\delta}\|_{L^2(\mathbb{T}^1)}^2ds \\
    \lesssim\; & \frac{1}{2} \mathbb{E} \| u_0 \|_{\dot{H}^{-1}(\mathbb{T}^1)}^2 - \frac{1}{2} \mathbb{E} \| w_{\delta} \|_{L^4([0,T]; L^4(\mathbb{T}^1))}^4 + \mathbb{E} \| z_{\delta} \|_{L^4([0,T]; L^4(\mathbb{T}^1))}^4 + \mathbb{E} \| z_{\delta} \|_{L^2([0,T]; L^2(\mathbb{T}^1))}^2 \\
    \leq\; & \frac{1}{2} \mathbb{E} \| u_0 \|_{\dot{H}^{-1}(\mathbb{T}^1)}^2 - \frac{1}{2} \mathbb{E} \| w_{\delta} \|_{L^4([0,T]; L^4(\mathbb{T}^1))}^4 + C.
\end{align*}
Applying the same reasoning to $w$, we obtain the uniform bound
\begin{align}\label{uniformes-w}
    \sup_{\delta > 0} \mathbb{E} \| w_{\delta} \|_{L^4([0,T]; L^4(\mathbb{T}^1))}^4 + \mathbb{E} \| w \|_{L^4([0,T]; L^4(\mathbb{T}^1))}^4 \leq C.
\end{align}

\textbf{Step 2. Estimates of the differences.} 

Let us define the differences $U_{\delta} := z - z_{\delta}$ and $V_{\delta} := w - w_{\delta}$. Our goal in this step is to derive suitable estimates for these differences, beginning with $U_{\delta}$. 

By applying the $L^p$-isometry for stochastic integrals with $p > 1$ (\cite{NW}) and utilizing the integration by parts formula, we observe that the following estimate holds: 
\begin{align*}
	\mathbb{E}\|U_{\delta}\|_{L^4([0,T];L^4(\mathbb{T}^1))}^4
	&= 2 \, \mathbb{E}\int_0^T \int_{\mathbb{T}^1} \Big|\int_0^t \int_{\mathbb{T}^1} \partial_x K_{t-s}(x-y) \big(dW\,dy - dW_{\delta}\,dy \big) \Big|^4 dx\, dt \\
	&\lesssim \int_0^T \int_{\mathbb{T}^1} \left( \mathbb{E} \int_0^t \left\| \partial_x K_{t-s}(x-\cdot) - \partial_x K_{t-s} \ast \eta_{\delta}(x-\cdot) \right\|_{L^2(\mathbb{T}^1)}^2 ds \right)^2 dx\, dt \\
	&\leq T \left(\int_0^T \left\| \partial_x K_s - \partial_x K_s \ast \eta_{\delta} \right\|_{L^2(\mathbb{T}^1)}^2 ds\right)^2.
\end{align*}
Utilizing the result in \eqref{finite-L2}, it follows that
\begin{equation}\label{convergence-Udelta}
	\mathbb{E}\|U_{\delta}\|_{L^4([0,T];L^4(\mathbb{T}^1))}^4 \leq T \left(\int_0^T \left\| \partial_x K_s - \partial_x K_s \ast \eta_{\delta} \right\|_{L^2(\mathbb{T}^1)}^2 ds\right)^2 \rightarrow 0, \quad \text{as } \delta \to 0.
\end{equation}

Next, we turn to estimating $V_{\delta}$. A direct computation reveals that $V_{\delta}$ satisfies the following partial differential equation:
\begin{align*}
\partial_t V_{\delta} 
= & -\frac{D}{2} (\partial_{xx}^2)^2 V_{\delta} + \partial_{xx}^2 \Big[(w+z)^3 - (w_{\delta} + z_{\delta})^3 - a(V_{\delta} + U_{\delta}) \Big] \\
= & -\frac{D}{2} (\partial_{xx}^2)^2 V_{\delta} + \partial_{xx}^2 \Big[ (V_{\delta} + U_{\delta}) \big((w+z)^2 + (w+z)(w_{\delta}+z_{\delta}) + (w_{\delta}+z_{\delta})^2 \big) \Big] \\
& \quad - a \partial_{xx}^2 (V_{\delta} + U_{\delta}).
\end{align*}

Applying It\^o's formula to the functional $\frac{1}{2} \langle (-\partial_{xx}^2)^{-1} V_{\delta}, V_{\delta} \rangle$, we obtain
\begin{align}\label{H-1Ito}
\frac{1}{2} \partial_t \langle (-\partial_{xx}^2)^{-1} V_{\delta}, V_{\delta} \rangle
= & -\frac{D}{2} \| \partial_x V_{\delta} \|_{L^2(\mathbb{T}^1)}^2 + a \int_{\mathbb{T}^1} V_{\delta} (V_{\delta} + U_{\delta}) \, dx \notag \\
& - \int_{\mathbb{T}^1} V_{\delta} \Big[ (V_{\delta} + U_{\delta}) \big((w+z)^2 + (w+z)(w_{\delta}+z_{\delta}) + (w_{\delta}+z_{\delta})^2 \big) \Big] dx.
\end{align}

Since $a < 0$, by Young's inequality we have
\begin{align*}
	a \int_{\mathbb{T}^1} V_{\delta} (V_{\delta} + U_{\delta}) dx 
	&\leq a \int_{\mathbb{T}^1} V_{\delta}^2 dx + \frac{|a|}{2} \int_{\mathbb{T}^1} \big( V_{\delta}^2 + U_{\delta}^2 \big) dx \\
	&= \frac{a}{2} \int_{\mathbb{T}^1} V_{\delta}^2 dx + \frac{|a|}{2} \int_{\mathbb{T}^1} U_{\delta}^2 dx.
\end{align*}

We further decompose the last term in \eqref{H-1Ito} as
\begin{align*}
	& -\int_{\mathbb{T}^1} V_{\delta} \big[(V_{\delta} + U_{\delta}) ((w+z)^2 + (w+z)(w_{\delta} + z_{\delta}) + (w_{\delta} + z_{\delta})^2) \big] dx \\
	= & -\int_{\mathbb{T}^1} V_{\delta}^2 \big((w+z)^2 + (w+z)(w_{\delta}+z_{\delta}) + (w_{\delta}+z_{\delta})^2 \big) dx \\
	& - \int_{\mathbb{T}^1} V_{\delta} U_{\delta} \big((w+z)^2 + (w+z)(w_{\delta}+z_{\delta}) + (w_{\delta}+z_{\delta})^2 \big) dx \\
	=: & I_1 + I_2.
\end{align*}

For the term $I_1$, applying Young's inequality to the mixed product, we deduce
\begin{align*}
	-\int_{\mathbb{T}^1} V_{\delta}^2 (w+z)(w_{\delta} + z_{\delta}) dx 
	\leq \int_{\mathbb{T}^1} \frac{1}{2} V_{\delta}^2 (w+z)^2 + \frac{1}{2} V_{\delta}^2 (w_{\delta} + z_{\delta})^2 dx.
\end{align*}
Hence, it follows that
\begin{align*}
	I_1 \leq -\int_{\mathbb{T}^1} \frac{1}{2} V_{\delta}^2 (w+z)^2 + \frac{1}{2} V_{\delta}^2 (w_{\delta} + z_{\delta})^2 dx.
\end{align*}

Similarly, the term $I_2$ is estimated as
\begin{align*}
	I_2 
	&\leq \frac{1}{2} \int_{\mathbb{T}^1} |V_{\delta} U_{\delta}| \big( (w+z)^2 + (w_{\delta} + z_{\delta})^2 \big) dx \\
	&\leq \frac{1}{2} \int_{\mathbb{T}^1} |V_{\delta}| \sqrt{(w+z)^2 + (w_{\delta} + z_{\delta})^2} |U_{\delta}| \sqrt{(w+z)^2 + (w_{\delta} + z_{\delta})^2} dx \\
	&\leq \frac{1}{4} \int_{\mathbb{T}^1} V_{\delta}^2 \big( (w+z)^2 + (w_{\delta} + z_{\delta})^2 \big) dx + C \int_{\mathbb{T}^1} U_{\delta}^2 \big( (w+z)^2 + (w_{\delta} + z_{\delta})^2 \big) dx.
\end{align*}

Taking expectations and applying H\"older's inequality, we obtain
\begin{align*}
	\mathbb{E}[I_1 + I_2] 
	&\leq C \, \mathbb{E} \int_{\mathbb{T}^1} U_{\delta}^2 \big( (w+z)^2 + (w_{\delta} + z_{\delta})^2 \big) dx \\
	&\leq C \, \mathbb{E} \| U_{\delta} \|_{L^4(\mathbb{T}^1)}^2 \left( \mathbb{E} \| w \|_{L^4(\mathbb{T}^1)}^2 + \mathbb{E} \| z \|_{L^4(\mathbb{T}^1)}^2 + \mathbb{E} \| w_{\delta} \|_{L^4(\mathbb{T}^1)}^2 + \mathbb{E} \| z_{\delta} \|_{L^4(\mathbb{T}^1)}^2 \right).
\end{align*}

Combining the above estimates and employing H\"older's inequality once more, we deduce the key inequality
\begin{align*}
&\sup_{t \in [0,T]} \mathbb{E} \| V_{\delta}(t) \|_{\dot{H}^{-1}(\mathbb{T}^1)}^2 + \mathbb{E} \int_0^T \| \partial_x V_{\delta} \|_{L^2(\mathbb{T}^1)}^2 ds - a \, \mathbb{E} \int_0^T \| V_{\delta} \|_{L^2(\mathbb{T}^1)}^2 ds \\
&\quad \leq C \int_0^T \mathbb{E} \| U_{\delta} \|_{L^4(\mathbb{T}^1)}^2 \left( \mathbb{E} \| w \|_{L^4(\mathbb{T}^1)}^2 + \mathbb{E} \| z \|_{L^4(\mathbb{T}^1)}^2 + \mathbb{E} \| w_{\delta} \|_{L^4(\mathbb{T}^1)}^2 + \mathbb{E} \| z_{\delta} \|_{L^4(\mathbb{T}^1)}^2 \right) ds \\
&\quad \quad + \frac{|a|}{2} \, \mathbb{E} \int_0^T \| U_{\delta} \|_{L^2(\mathbb{T}^1)}^2 ds \\
&\quad \leq C \, \mathbb{E} \| U_{\delta} \|_{L^4([0,T];L^4(\mathbb{T}^1))}^2 \Big( \mathbb{E} \| w \|_{L^4([0,T];L^4(\mathbb{T}^1))}^2 + \mathbb{E} \| z \|_{L^4([0,T];L^4(\mathbb{T}^1))}^2 + \mathbb{E} \| w_{\delta} \|_{L^4([0,T];L^4(\mathbb{T}^1))}^2\\
&\quad \quad + \mathbb{E} \| z_{\delta} \|_{L^4([0,T];L^4(\mathbb{T}^1))}^2 \Big) + \frac{|a|}{2} \mathbb{E} \int_0^T \| U_{\delta} \|_{L^2(\mathbb{T}^1)}^2 ds.
\end{align*}

Thanks to the uniform bounds established in \eqref{uniformes-z}, \eqref{uniformes-w} together with the convergence result \eqref{convergence-Udelta}, we conclude that
\begin{align*}
	& \sup_{t \in [0,T]} \mathbb{E} \| V_{\delta}(t) \|_{\dot{H}^{-1}(\mathbb{T}^1)}^2 + \mathbb{E} \int_0^T \| \partial_x V_{\delta} \|_{L^2(\mathbb{T}^1)}^2 ds - a \, \mathbb{E} \int_0^T \| V_{\delta} \|_{L^2(\mathbb{T}^1)}^2 ds \\
	&\quad \leq C \left( \mathbb{E} \| U_{\delta} \|_{L^4([0,T];L^4(\mathbb{T}^1))}^2 + \mathbb{E} \| U_{\delta} \|_{L^2([0,T];L^2(\mathbb{T}^1))}^2 \right) \to 0, \quad \text{as } \delta \to 0.
\end{align*}

Consequently, we deduce the strong convergence of the difference $u_{\delta} - u$ in $L^2([0,T];L^2(\mathbb{T}^1))$ norm in expectation:
\begin{equation*}
	\mathbb{E} \| u_{\delta} - u \|_{L^2([0,T];L^2(\mathbb{T}^1))}^2 \lesssim \mathbb{E} \| U_{\delta} \|_{L^2([0,T];L^2(\mathbb{T}^1))}^2 + \mathbb{E} \| V_{\delta} \|_{L^2([0,T];L^2(\mathbb{T}^1))}^2 \to 0, \quad \text{as } \delta \to 0,
\end{equation*}
thus completing the proof.

\end{proof}

\section{Corollaries and challenges of the square-root coefficient}\label{sec-app-B} 
As mentioned in the introduction, the results of this paper simplify the setting of the conjecture proposed in \cite{GJE99} in two key aspects: first, by replacing the space-time white noise with spatially correlated noise, and second, by regularizing the singular square-root diffusion coefficient with a smooth approximation. The simplification of the noise structure is currently essential due to the absence of a well-posedness theory for supercritical singular stochastic PDEs. On the other hand, addressing the singularity of the square-root coefficient is particularly relevant in the framework of fluctuating hydrodynamics; see, for instance, \cite{FG24, FG23, DFG, GWZ24, WWZ22, WZ24}. 

In this section, we present several corollaries of the results established in Section \ref{sec-3}, focusing on the rescaled fluctuating Ising-Kac equation with a square-root diffusion coefficient. We also discuss the analytical difficulties involved in establishing the nonlinear fluctuation phenomenon in the presence of such a singular coefficient.

Precisely, for every $\gamma\in(0,1]$, we consider the following equation: 
\begin{align}\label{equgamma-3}
\partial_tu_{\gamma}=\gamma^{-2/3}\partial_{xx}^2 u_{\gamma}&-\gamma^{-2/3}(1+a\gamma^{2/3})\partial_{x}[(1-\gamma^{2/3}u_{\gamma}^2)J_{\gamma^{1/3}}\ast\partial_{x} u_{\gamma}]\notag\\
&-\sqrt{2}\partial_{x}\Big(\sqrt{1-\gamma^{2/3}u_{\gamma}^2}dW_{\delta}\Big)+\partial_{x}\Big(F_{1,\delta}\frac{1}{1-\gamma^{2/3}u_{\gamma}^2}\gamma^{4/3}u_{\gamma}^2\partial_{x}u_{\gamma}\Big). 
\end{align}

\subsection{Renormalized kinetic solutions and well-posedness} 
Due to the singularity of the square-root diffusion coefficient, the regularity obtained from the available a priori estimates, such as the entropy dissipation bounds, for equation \eqref{equgamma-3} is insufficient to ensure the integrability of the Stratonovich-to-It\^o correction term. As a result, the notion of weak solutions becomes ill-defined in this context. This necessitates the development of a generalized solution framework capable of accommodating such irregular equations. Following \cite{LPT94, FG24}, and as discussed in \cite{WWZ22}, we introduce the notion of a renormalized kinetic solution.

For every $\gamma\in(0,1]$, define the kinetic function associated with $u_\gamma$ by
$$
\chi_{\gamma}(x, \zeta, t) = I_{\{0 < \zeta < u_{\gamma}(x,t)\}} - I_{\{u_{\gamma}(x,t) < \zeta < 0\}},
$$
where the simplification holds due to the assumption $u_{\gamma} \geq 0$. In the distributional sense, the chain rule yields
$$
\partial_x \chi_{\gamma} = \delta_0(\zeta - u_{\gamma}) \, \partial_x u_{\gamma}, \quad \partial_{\zeta} \chi_{\gamma} = \delta_0(\zeta) - \delta_0(\zeta - u_{\gamma}).
$$

In the following, we introduce the definition of the renormalized kinetic solution rigorously. 

\begin{definition}\label{def-7.2}(Kinetic measure)
	For every $\gamma\in(0,1]$, a kinetic measure is defined as a measurable map  
$$
q_{\gamma}: \Omega \to \mathcal{M}^+(\mathbb{T}^1 \times [-\gamma^{-1/3},\gamma^{1/3}] \times [0,T]),
$$ 
where $\mathcal{M}^+$ denotes the set of finite nonnegative Radon measures. It is required that, for every test function $\psi \in C_c^\infty(\mathbb{T}^d \times [-\gamma^{-1/3},\gamma^{1/3}])$, the process  
$$
(\omega,t) \in \Omega \times [0,T] \;\longmapsto\; 
\int_0^t \int_{\mathbb{R}} \int_{\mathbb{T}^1} \psi(x,\zeta)\,\mathrm{d}q_{\gamma}(x,\zeta,r)
$$  
is $\mathcal{F}_t$-predictable.

\end{definition}

\begin{definition}(Renormalized kinetic solution)
	For every $\gamma\in(0,1]$, let $u_{\gamma,0}\in  \overline{\text{{\rm{Ent}}}}(\mathbb{T}^{1})$. A renormalized kinetic solution of \eqref{eqgamma-1} with initial data $u_{\gamma,0}$ is an almost surely continuous $L^2(\mathbb{T}^1;[-\gamma^{-1/3},\gamma^{-1/3}])$-valued $\mathcal{F}_t$-predictable function $\rho\in L^2\left(\Omega\times[0,T];L^2(\mathbb{T}^1;[-\varepsilon^{-1},\varepsilon^{-1}])\right)$ that satisfies the following properties.
	\begin{enumerate}
		\item Essentially bounded: almost surely for every $t\in[0,T]$,
		\begin{equation}\label{eqq4}
		u_{\gamma}(\cdot,t)\in[-\gamma^{-1/3},\gamma^{-1/3}],\ a.e.
		\end{equation}
		\item Regularity of $\sqrt{1-\gamma^{2/3}u_{\gamma}^2}$: there exists a constant $c\in(0,\infty)$ depending on $T,\rho_{\gamma,0}, J$ and $\gamma$ such that
		\begin{equation}\label{eqq5}
		\mathbb{E}\int^T_0\int_{\mathbb{T}^1}\Big[|\partial_x\sqrt{1-\gamma^{2/3}u_{\gamma}^2}|^2+|\partial_xu_{\gamma}|^2\Big]\mathrm{d}x\mathrm{d}s\le c(T,u_{\gamma,0},J,\gamma).
		\end{equation}
		Furthermore, there exists a finite nonnegative kinetic measure $q_{\gamma}$ satisfying the following properties.
		\item Regularity: almost surely
		\begin{align}\label{eqq6}
		\gamma^{-2/3}\delta_{0}(\zeta-u_{\gamma})|\partial_x u_{\gamma}|^{2}\le q_{\gamma}\quad {\rm{on}}\ \mathbb{T}^1\times[-\gamma^{-1/3},\gamma^{-1/3}]\times[0,T].
		\end{align}
\item Optimal regularity: the measure $\mu_{\gamma}$ defined by
	\begin{align}\label{mu}
		\mathrm{d}\mu_{\gamma}=\left(1-\gamma^{2/3}\zeta^2\right)^{-1}\mathrm{d}q_{\gamma}\ \text{is finite on}\ \mathbb{T}^1\times(-\gamma^{-1/3},\gamma^{-1/3})\times[0,T].
	\end{align}

		\item The equation: for every $\varphi\in\mathrm{C}_{c}^{\infty}\left(\mathbb{T}^{d}\times(-\gamma^{-1/3},\gamma^{-1/3})\right)$, almost surely for every $t\in[0,T]$,
		\begin{align}\label{eqq8}
		\int_{\mathbb{R}}\int_{\mathbb{T}^1}\chi_{\gamma}(x,\zeta,t)\varphi(x,\zeta)dxd\zeta=&\int_{\mathbb{R}}\int_{\mathbb{T}^1}\bar{\chi}(u_{\gamma,0})\varphi(x,\zeta)dxd\zeta-2\gamma^{-2/3}\int_0^t\int_{\mathbb{T}^1}\partial_xu_{\gamma}(\partial_x\varphi)(x,u_{\gamma})dxds\notag\\
		-&\gamma^{4/3}\int_0^t\int_{\mathbb{T}^1}F_{1,\delta}(x)\frac{u_{\gamma}^2}{1-\gamma^{2/3}u_{\gamma}^2}\partial_x u_{\gamma}(\partial_x\varphi)(x,u_{\gamma})dxds\notag\\
		-&\gamma^{-2/3}(1+a\gamma^{2/3})\int_0^t\int_{\mathbb{T}^1}\varphi(x,u_{\gamma})\partial_x((1-\gamma^{2/3}u_{\gamma}^2)\partial_x J_{\gamma^{1/3}}\ast u_{\gamma})dxds\notag\\
		+&\int_0^t\int_{\mathbb{T}^1}(1-\gamma^{2/3}u_{\gamma}^2)F_{3,\delta}(x)(\partial_{\zeta}\varphi)(x,u_{\gamma})dxds-2\int_0^t\int_{\mathbb{R}}\int_{\mathbb{T}^1}\partial_{\zeta}\varphi(x,\zeta)dq_{\gamma}\notag\\
		-&\sqrt{2}\int_0^t\int_{\mathbb{T}^1}\varphi(x,u_{\gamma})\partial_x\Big(\sqrt{1-\gamma^{2/3}u_{\gamma}^2} \mathrm{d}W_{\delta}(s)\Big)dx,
		\end{align}
		where $\bar{\chi}(u_{\gamma,0})(x,\zeta):=I_{\{0<\zeta<u_{\gamma,0}(x)\}}-I_{\{u_{\gamma,0}(x)<\zeta<0\}}$.
	\end{enumerate}	
\end{definition}
Refer to \cite{WWZ22}, we have the following well-posedness result. 
\begin{theorem}
	Assume that the initial data, the coefficient, and the interaction kernel $J$ satisfy Assumptions (A1), \ref{Assump-sigma}, and \ref{Assump-J}, respectively. For every $\gamma\in(0,1]$, let $u_{\gamma,0}\in\overline{{\rm{Ent}}}(\mathbb{T}^1)$. Then there exists a unique renormalized kinetic solution to \eqref{equgamma-3} in the sense of Definition \ref{def-1} with initial data $u_{\gamma,0}$. 
	\end{theorem}
	
\subsection{Uniform estimates and discussion of challenges} 
To establish the nonlinear fluctuation phenomenon for the renormalized kinetic solution of \eqref{equgamma-3}, it is necessary to derive uniform estimates analogous to those stated in Proposition~\ref{prp-entropydissipation} and Proposition~\ref{second-uniform-es}. In the sequel, we demonstrate that the uniform entropy dissipation estimates of Proposition~\ref{prp-entropydissipation} can indeed be extended to the renormalized kinetic solution of \eqref{equgamma-3}. 
\begin{corollary}\label{cor-entropydissipation-kineitc}
Assume that the initial data, the coefficient, and the interaction kernel $J$ satisfy Assumptions (A1), \ref{Assump-sigma}, and \ref{Assump-J}, respectively. For every $\gamma\in(0,1]$, let $u_{\gamma}$ be the renormalized kinetic solution of \eqref{equgamma} with initial data $u_{\gamma,0}$. Let $\Psi_{\gamma}$ be defined by \eqref{rescaled-entropy}. Suppose that $a<0$, then 
\begin{align}\label{uniformestimate-1}
	\sup_{t\in[0,T]}\mathbb{E}\Big(\int_{\mathbb{T}^1}\frac{1}{2}|u_{\gamma}(t)|^2dx\Big)+&\mathbb{E}\int^T_0\int_{\mathbb{T}^1}\frac{u_{\gamma}^2|\partial_xu_{\gamma}|^2}{1-\gamma^{2/3}u_{\gamma}^2}dxdt\notag\\
	&-a\mathbb{E}\int^T_0\|\partial_xu_{\gamma}\|_{L^2(\mathbb{T}^1)}^2ds\lesssim C(\delta),\ \ \text{for every }\gamma\in(0,1].  
\end{align}
\end{corollary}
\begin{proof}
This is a consequence of the combination of the proof for Proposition \ref{prp-entropydissipation} and the entropy dissipation estimates for renormalized kinetic solution in \cite[Proposition 5.18]{FG24} and \cite[Proposition 7.6]{WWZ22}. 	
\end{proof}

However, the identical time-regularity estimates in Proposition \ref{second-uniform-es} can not be established for \eqref{equgamma-3} by the same approach. This is due to the irregularity of the square-root coefficients, that produces the ill-definedness of the concept of weak solutions. To address this issue, \cite{FG24,DFG} provide an approach to obtain time-regularity estimates by a truncation argument and a diagonal approach. Later on, \cite{WWZ22} apply this approach to study the global-in-time existence of \eqref{equgamma-3}. Precisely, for every $\iota \in(0,1/4)$, let $\psi_{\iota} \in C^{\infty}(\mathbb{R})$ be a smooth function satisfying $0 \leq \psi_{\iota} \leq 1$ and
\begin{align*}
	\psi_{\iota}(\zeta)=
	\left\{
	\begin{array}{ll}
		1, & {\rm{if}}\ \zeta \in[-1+\iota, 1-\iota], \\
		0, &{\rm{if}}\  \zeta \in\left(-\infty,-1+\frac{\iota}{2}\right] \cup\left[1-\frac{\iota}{2},\infty\right),\\
		{\rm{smooth}}, &{\rm{otherwise}}.
	\end{array}
	\right.
\end{align*}
Moreover, $\left|\psi_{\iota}^{\prime}(\zeta)\right| \leqslant c / \iota$ for some $c \in(0, \infty)$ independent of $\iota$.
With the aid of $\psi_{\iota}$, we can define $h_{\iota}(\zeta)=\psi_{\iota}(\zeta)\zeta$ for every $\zeta \in \mathbb{R}$. 

In general, the tightness of renormalized kinetic solutions requires a uniform time-regularity estimate of $h_{\iota}(u_{\gamma})$ for any fixed $\iota>0$. Then by sending $\iota\rightarrow0$, a diagonal argument establishes the tightness, see \cite{FG24,FG23,DFG} for more details. Therefore in our setting, it is natural to consider applying this argument to \eqref{eq-remainder}. However, due to the lack of sufficient regularity of $u_{\gamma}$, it is not clear how to estimate $h_{\iota}'(u_{\gamma})\partial_{xxxx}^4u_{\gamma}$ and to derive the chain rule rigorously, 
   \begin{align*}
\partial_th_{\iota}(u_{\gamma})=&-\frac{D}{2}h_{\iota}'(u_{\gamma})\partial_{xxxx}^4u_{\gamma}+h_{\iota}'(u_{\gamma})\partial_{xx}^2\Big[\frac{1}{3}u_{\gamma}^3-au_{\gamma}\Big]-\sqrt{2}h_{\iota}'(u_{\gamma})\partial_{x}(\sigma(1-\gamma^{2/3}u_{\gamma}^2)dW_{\delta})\notag\\
&+h_{\iota}''(u_{\gamma})|\partial_x\sigma(1-\gamma^{2/3}u_{\gamma}^2)|^2F_{1,\delta}+h_{\iota}''(u_{\gamma}\sigma(1-\gamma^{2/3}u_{\gamma}^2)^2F_{3,\delta}\notag\\
&+h_{\iota}'(u_{\gamma})\Big(R_{1,\gamma}+R_{2,\gamma}+R_{3,\gamma}+R_{4,\gamma}\Big), 
\end{align*}
where $R_{i,\gamma}$, $i=1,2,3,4$ are defined by \eqref{remainder}. This presents obstacles to establishing uniform time-regularity estimates for \eqref{equgamma-3}.

\bibliographystyle{alpha}
\bibliography{Ising-Kac.bib}

\end{document}